\documentclass{interact}

\usepackage[caption=false]{subfig}

\usepackage[numbers,sort&compress]{natbib}
\bibpunct[, ]{[}{]}{,}{n}{,}{,}
\makeatletter
\def\NAT@def@citea{\def\@citea{\NAT@separator}}
\makeatother

\usepackage{amsmath, amsthm, amssymb}
\usepackage{thmtools}
\usepackage{mathtools}
\mathtoolsset{centercolon}
\usepackage{enumitem}
\usepackage{graphicx}
\usepackage[hidelinks]{hyperref}
\hypersetup{
  pdftitle={Convexity of the orbit-closed C-numerical range and majorization},
  pdfauthor={Jireh Loreaux, Sasmita Patnaik}
}
\usepackage{tikz}
\usetikzlibrary{cd}
\usepackage{cleveref}

\newcommand{\norm}[1]{\ensuremath{\lVert #1 \rVert}}
\newcommand{\abs}[1]{\ensuremath{\lvert #1 \rvert}}
\newcommand{\angles}[1]{\ensuremath{\langle #1 \rangle}}

\newcommand{\set}[1]{\{#1\}}
\newcommand{\vset}[1]{\left\set{#1\right}}
\newcommand{\vmid}{\,\middle\vert\,}

\expandafter\mathchardef\expandafter\varphi\number\expandafter\phi\expandafter\relax
\expandafter\mathchardef\expandafter\phi\number\varphi

\makeatletter
\newcommand\bigcdot{\mathpalette\bigcdot@{.5}}
\newcommand*\bigcdot@[2]{\mathbin{\vcenter{\hbox{\scalebox{#2}{$\m@th#1\bullet$}}}}}
\makeatother

\newcommand{\term}[1]{\emph{#1}}

\newcommand{\Hil}{\mathcal{H}}
\newcommand{\K}{\mathcal{K}}

\newcommand{\traceclass}{\mathcal{L}_1}
\newcommand{\orbit}{\mathcal{O}}
\newcommand{\ugroup}{\mathcal{U}}

\newcommand{\cz}{c_0}

\newcommand{\czstar}{\cz^{*}}
\newcommand{\eig}{\lambda}
\newcommand{\modeig}{\tilde{\eig}}
\newcommand{\zop}{\mathbf{0}}

\renewcommand{\emptyset}{\varnothing}
\renewcommand{\epsilon}{\varepsilon}

\newcommand{\nats}{\mathbb{N}}
\newcommand{\ints}{\mathbb{Z}}
\newcommand{\reals}{\mathbb{R}}
\newcommand{\complex}{\mathbb{C}}
\newcommand{\mat}{M}

\DeclareFontFamily{U}{mathb}{\hyphenchar\font45}
\DeclareFontShape{U}{mathb}{m}{n}{
<-6> mathb5 <6-7> mathb6 <7-8> mathb7
<8-9> mathb8 <9-10> mathb9
<10-12> mathb10 <12-> mathb12
}{}
\DeclareSymbolFont{mathb}{U}{mathb}{m}{n}
\DeclareMathSymbol{\pprec}{\mathrel}{mathb}{"CE}
\DeclareMathSymbol{\ssucc}{\mathrel}{mathb}{"CF}

\DeclareMathOperator{\spec}{\sigma}
\newcommand{\essspec}{\spec_{\mathrm{ess}}}
\newcommand{\ptspec}{\spec_{\mathrm{pt}}}
\DeclareMathOperator{\diag}{diag}
\DeclareMathOperator{\rank}{rank}
\DeclareMathOperator{\trace}{Tr}
\DeclareMathOperator{\spans}{span}
\DeclareMathOperator{\conv}{conv}

\DeclareMathOperator{\dist}{dist}
\newcommand{\closure}[2][]{\overline{#2}^{#1}}
\newcommand{\cspec}[1][C]{P_{#1}}
\newcommand{\ocspec}[1][C]{\spec_{\scalebox{0.85}[1]{$\scriptscriptstyle\orbit(#1)$}}}
\newcommand{\nr}{W}
\newcommand{\essnr}{\nr_{\textrm{ess}}}
\newcommand{\knr}[1][k]{\nr_{#1}}
\newcommand{\cnr}[1][C]{\nr_{#1}}
\newcommand{\ocnr}[1][C]{\nr_{\orbit(#1)}}
\newcommand{\maj}{\prec}
\newcommand{\submaj}{\pprec}

\setlist[enumerate]{label=\textup{(\roman*)}}

\numberwithin{equation}{section}

\theoremstyle{plain} 
\newtheorem{theorem}{Theorem}[section]
\newtheorem{corollary}[theorem]{Corollary}
\newtheorem{lemma}[theorem]{Lemma}
\newtheorem{proposition}[theorem]{Proposition}

\theoremstyle{definition}
\newtheorem{definition}[theorem]{Definition}

\theoremstyle{remark}
\newtheorem{remark}[theorem]{Remark}
\newtheorem{example}[theorem]{Example}

\newcounter{case}
\makeatletter\@addtoreset{case}{theorem}\makeatother
\newenvironment{case}[1]{\stepcounter{case}\par{\itshape Case \thecase. #1\par\noindent}}{\par}

\title{Convexity of the orbit-closed \\ $C$-numerical range and majorization}

\author{\name{Jireh Loreaux\textsuperscript{a}\thanks{J. Loreaux email: jloreau@siue.edu} and Sasmita Patnaik\textsuperscript{b}\thanks{S. Patnaik email: sasmita@iitk.ac.in}}\affil{\textsuperscript{a}Southern Illinois University Edwardsville, 1 Hairpin Dr, Edwardsville, IL, 62026, USA; \textsuperscript{b}Indian Institute of Technology, Kanpur, Kalyanpur, Kanpur-208016, India.}}

\begin{document}

\maketitle

\begin{abstract}
  We introduce and investigate the orbit-closed $C$-numerical range, a natural modification of the $C$-numerical range of an operator introduced for $C$ trace-class by Dirr and vom Ende.
  Our orbit-closed $C$-numerical range is a conservative modification of theirs because these two sets have the same closure and even coincide when $C$ is finite rank.
  Since Dirr and vom Ende's results concerning the $C$-numerical range depend only on its closure, our orbit-closed $C$-numerical range inherits these properties, but we also establish more.

  For $C$ selfadjoint, Dirr and vom Ende were only able to prove that the closure of their $C$-numerical range is convex, and asked whether it is convex without taking the closure.
  We establish the convexity of the orbit-closed $C$-numerical range for selfadjoint $C$ without taking the closure by providing a characterization in terms of majorization, unlocking the door to a plethora of results which generalize properties of the $C$-numerical range known in finite dimensions or when $C$ has finite rank.
  Under rather special hypotheses on the operators, we also show the $C$-numerical range is convex, thereby providing a partial answer to the question posed by Dirr and vom Ende.
\end{abstract}

\begin{keywords}
  numerical range, $C$-numerical range, convex, trace-class, Toeplitz--Hausdorff Theorem, unitary orbit, Hausdorff distance, essential spectrum
\end{keywords}
\begin{amscode}
  Primary 47A12, 47B15; Secondary 52A10, 52A40, 26D15.
\end{amscode}

\section{Introduction}

Herein we let $\Hil$ denote a separable complex Hilbert space and $B(\Hil)$ the collection of all bounded linear operators on $\Hil$.
For $A \in B(\Hil)$, the \term{numerical range} $\nr(A)$ is the image of the unit sphere of $\Hil$ under the continuous quadratic form $x \mapsto \angles{Ax,x}$, where $\angles{\bigcdot,\bigcdot}$ denotes the inner product on $\Hil$.
Of course, the numerical range has a long history but perhaps the most impactful result is the Toeplitz--Hausdorff Theorem which asserts that the numerical range is convex \
\cite{Toe-1918-MZ,Hau-1919-MZ}; an intuitive proof is given by Davis in \cite{Dav-1971-CMB}.
In this paper we are interested in unitarily invariant generalizations of the numerical range and their associated properties, especially convexity and its relation to majorization.

By considering an alternative definition of the numerical range, some generalizations become readily apparent.
Notice that
\begin{equation*}
  \nr(A) = \set{\angles{Ax,x} \mid x \in \Hil, \norm{x}=1} = \set{ \trace(PA) \mid P\ \text{is a rank-1 projection}}.
\end{equation*}
As Halmos recognized in \cite{Hal-1964-ASM}, one could generalize this by fixing $k \in \nats$ and requiring $P$ to be a rank-$k$ projection.
In this way, we arrive at the \term{$k$-numerical range}
\begin{equation*}
  \knr(A) := \vset{ \trace\Big(\frac{1}{k}PA\Big) \vmid P\ \text{is a rank-$k$ projection}}.
\end{equation*}
The normalization constant $\frac{1}{k}$ is actually quite natural;
among other things, it ensures $\knr(A)$ is bounded independent of $k$ by $\norm{A}$.
In \cite[\textsection{}12]{Ber-1963}, Berger proved a few fundamental facts about the $k$-numerical range including its convexity, as well as the inclusion property $\knr[k+1](A) \subseteq \knr(A)$.
We will see shortly that these convexity and inclusion properties are actually consequences of more general phenomena (see \Cref{cor:c-numerical-range-convex,cor:majorization-c-numerical-range-inclusion}).

In \cite{FW-1971-GMJ}, Fillmore and Williams examined $\knr(A)$, but restricted their attention to the finite dimensional setting.
There they established
\begin{equation}
  \label{eq:knr-majorization-description}
  \knr(A) = \vset{ \frac{1}{k} \trace(XA) \vmid 0 \le X \le I, \trace X = k },
\end{equation}
which was generalized by Goldberg and Straus to the $C$-numerical range, as we describe below.
Moreover, Fillmore and Williams showed that if $A \in \mat_n(\complex)$ is normal, then
\begin{equation}
  \label{eq:knr-extreme-points}
  \knr(A) = \conv \vset{ \frac{1}{k} \sum_{i=1}^k \lambda_i \vmid  \parbox{32ex}{$\lambda_i$ is an eigenvalue of $A$, repeated at most according to multiplicity} \ },
\end{equation}
which says that the extreme points of $\knr(A)$ are contained in the set of averages of $k$-eigenvalues of $A$ (allowing repetitions according to multiplicity).
This is a clear analogue of the standard fact for numerical ranges that $\nr(A) = \conv \spec(A)$ when $A \in \mat_n(\complex)$ is normal.

In order to further generalize the $k$-numerical range, yet another new perspective is necessary.
The unitary group $\ugroup$ of $B(\Hil)$ acts by conjugation on $B(\Hil)$, and the orbit $\ugroup(C)$ of an operator $C \in B(\Hil)$ under this action is called the \term{unitary orbit}.
When $P$ is any rank-$k$ projection ($k < \infty$), $\ugroup(P)$ consists of all rank-$k$ projections in $B(\Hil)$.
Therefore, if $P$ is a rank-$k$ projection, then
\begin{equation*}
  \knr(A) = \vset{ \trace(XA) \vmid X \in \ugroup\left(\frac{1}{k}P\right) }.
\end{equation*}

The above representation of the $k$-numerical range suggests the natural generalization to the \term{$C$-numerical range},
\begin{equation*}
  \cnr(A) := \set{ \trace(XA) \mid X \in \ugroup(C) }.
\end{equation*}
Of course, this requires $\trace(XA)$ to make sense, which can be achieved in several different ways, each investigated by various authors.
In \cite{Wes-1975-LMA}, Westwick considered $\cnr(A)$ when $C$ is a finite rank selfadjoint operator and proved that $\cnr(A)$ is convex by means of Morse theory.
When $\dim \Hil = n < \infty$ so that $B(\Hil) \cong \mat_n(\complex)$, $\cnr(A)$ is well-defined for an arbitrary $C \in \mat_n(\complex)$.
The $C$-numerical range was first studied in this generality by Goldberg and Straus in \cite{GS-1977-LAA}.
There, they proved a generalization of \eqref{eq:knr-majorization-description} for an arbitrary selfadjoint matrix $C$, which we extend to the infinite dimensional setting in \Cref{thm:c-numerical-range-via-majorization}.
Chi-Kwong Li provides in \cite{Li-1994-LMA} a comprehensive survey of the properties of the $C$-numerical range in finite dimensions, including the highlights which we now describe.
When $C$ is selfadjoint the $C$-numerical range is convex, but this may fail even if $C$ is normal \cite{Wes-1975-LMA,AT-1983-LMA}.
However, the $C$-numerical range is always star-shaped relative to the star center $\trace(C) \big(\frac{1}{n} \trace(A)\big)$ \cite{CT-1996-LMA}.
Moreover, there is a set $\cspec(A)$ associated to the pair $C,A$ called the \term{$C$-spectrum of $A$} which, when $C$ is a rank-1 projection, coincides with the usual spectrum of $A$;
Then when $A$ is normal and $C$ is selfadjoint, $\cnr(A) = \conv \cspec(A)$ \cite[Theorem~4]{Mar-1979-ANYAS}, which generalizes \eqref{eq:knr-extreme-points}.

In the recent paper \cite{DvE-2020-LaMA}, Dirr and vom Ende study a generalization of the $C$-numerical range to the infinite dimensional setting.
In this case, it again becomes necessary to ensure that the trace $\trace(XA)$ is well-defined, which they naturally enforce by requiring $C$ to be trace-class.
In \cite{DvE-2020-LaMA}, they prove extensions of some finite dimensional results by means of limiting arguments.
As a result of these limiting arguments, all of their major results pertain to the \emph{closure} $\closure{\cnr(A)}$ of the $C$-numerical range.
Dirr and vom Ende prove that $\closure{\cnr(A)}$ is star-shaped and that any element of $\trace(C) \essnr(A)$ is a star center \cite[Theorem~3.10]{DvE-2020-LaMA}.
They asked explicitly \cite[Open Problem (b)]{DvE-2020-LaMA} whether $\cnr(A)$ is convex without taking the closure, and we provide a partial answer in \Cref{cor:c-numerical-range-convex-diagonalizable}.
Moreover, they show that $\closure{\cnr(A)}$ is convex whenever $C$ is selfadjoint\footnote{or only slightly more generally, $C$ normal with collinear eigenvalues. In this paper, we have many results for selfadjoint $C$, but they generally have trivial unmentioned corollaries for $C$ normal with collinear eigenvalues by means of \Cref{prop:c-numerical-range-basics}\ref{item:similarity-preserving-ocnr}. We neglect these slightly more general statements in favor of the selfadjoint ones solely for clarity and simplicity of exposition.} or $A$ is a rotation and translation of a selfadjoint operator \cite[Theorem~3.8]{DvE-2020-LaMA}.
Additionally, they prove that if $C,A$ are both normal, $A$ is compact and the eigenvalues of either $C$ or $A$ are collinear, then $\closure{\cnr(A)} = \conv(\closure{\cspec(A)})$ \cite[Corollary~3.1]{DvE-2020-LaMA}.

In this paper we introduce and investigate a natural modification of the $C$-numerical range with $C$ trace-class which we call the \term{orbit-closed $C$-numerical range}, denoted $\ocnr(A)$ (see \Cref{def:orbit-closed-c-numerical-range}).
The only difference between $\ocnr(A)$ and $\cnr(A)$ is that the former allows $X$  which are \emph{approximately} unitarily equivalent (in trace norm) to $C$, that is,
\begin{equation*}
  \ocnr(A) := \set{ \trace(XA) \mid X \in \orbit(C) },
\end{equation*}
where $\orbit(C) := \closure[\norm{\bigcdot}_1]{\ugroup(C)}$.
Considering closures of unitary orbits in various operator topologies serves an important purpose and has precedent in the literature, especially in relation to majorization (see the discussion which introduces \cref{sec:orbit-closed-c-numerical-range}).

This relatively small difference between $\cnr(A)$ and $\ocnr(A)$ has significant consequences.
In particular, for $C$ selfadjoint we give a characterization of $\ocnr(A)$ in terms of majorization (\Cref{thm:c-numerical-range-via-majorization}) which is an appropriate extension to infinite dimensions of \cite[Theorem~1.2]{FW-1971-GMJ} (included herein as \eqref{eq:knr-majorization-description}) and its generalization \cite[Theorem~7]{GS-1977-LAA}, and whose proof is inspired by \cite[Theorem~2.14]{DS-2017-PEMSIS}.
Because in general $\cnr(A) \not= \ocnr(A)$, necessarily $\cnr(A)$ cannot enjoy this same characterization.
Moreover, this majorization characterization of $\ocnr(A)$ is the backbone of this paper and it provides a gateway to the rest of our major results.
One immediate corollary is the convexity of $\ocnr(A)$ when $C$ is selfadjoint (\Cref{cor:c-numerical-range-convex}) which generalizes and provides an independent and purely operator-theoretic proof of Westwick's theorem \cite{Wes-1975-LMA} for $C$ a finite rank selfadjoint operator.
Moreover, to our knowledge our \Cref{cor:c-numerical-range-convex} constitutes the \emph{only}\footnote{
  In the \emph{finite dimensional} setting there is an independent proof of Westwick's theorem due to Poon \cite{Poo-1980-LMA} using a result of Goldberg and Straus \cite[Theorem~7]{GS-1977-LAA}.
  This proof is similar in spirit to our \Cref{cor:c-numerical-range-convex} because it involves majorization.
  However, it seems to us that the techniques in \cite{Poo-1980-LMA} cannot be used to reprove Westwick's result in the infinite dimensional setting even for finite rank $C$.
  We say this because both \cite[Lemma~1]{Poo-1980-LMA} and \cite[Theorem~7]{GS-1977-LAA} rely in an essential way on Birkhoff's Theorem \cite{Bir-1946-UNTRA}.
  The dependence of \cite[Theorem~7]{GS-1977-LAA} on Birkhoff's Theorem is not readily apparent, but can be observed through a careful analysis of the proof of \cite[Lemma~7]{GS-1977-LAA}.
}
independent proof of Westwick's convexity result in the infinite dimensional setting found in 45 years, which is especially significant because Westwick's proof used an unusual technique: Morse theory.

In addition, $\ocnr(A)$ is a \emph{conservative} modification of $\cnr(A)$ in the sense that $\cnr(A) \subseteq \ocnr(A) \subseteq \closure{\cnr(A)}$ (see \Cref{thm:cnr-dense-in-ocnr}), and moreover, if $C$ is finite rank, then $\ocnr(A) = \cnr(A)$.
Therefore, the orbit-closed $C$-numerical range constitutes an alternate natural extension of the $C$-numerical range to the infinite dimensional (and infinite rank) setting.
Moreover, because $\closure{\ocnr(A)} = \closure{\cnr(A)}$, all of Dirr and vom Ende's results (which concern the closure of the $C$-numerical range) are inherited by the orbit-closed $C$-numerical range.

Our main results are summarized in the list below.
Here $\lambda(C)$ denotes the eigenvalue sequence of a compact operator $C$ (see \cref{sec:notation}), $\maj$, $\submaj$ denote majorization and submajorization (see \Cref{def:sequence-majorization}), and $\ocspec(A)$ denotes the $\orbit(C)$-spectrum (see \Cref{def:c-spectrum}).
Reference \cref{sec:notation} for any other unfamiliar notation.
\begin{enumerate}
\item $\ocnr(A) = \cnr(A)$ if $C$ is finite rank (\Cref{prop:orbit-closure-equivalences}).
\item $\closure{\ocnr(A)} = \closure{\cnr(A)}$ (\Cref{thm:cnr-dense-in-ocnr}).
\item The map $(C,A) \mapsto \ocnr(A)$ is continuous (\Cref{thm:continuity}).
\item If $C \in \traceclass^{sa}$, then $\ocnr(A) = \set{ \trace(XA) \mid X \in \traceclass^{sa}, \lambda(X) \maj \lambda(C) }$ (\Cref{thm:c-numerical-range-via-majorization}).
\item If $C \in \traceclass^{sa}$, then $\ocnr(A)$ is convex (\Cref{cor:c-numerical-range-convex}).
\item If $C,C' \in \traceclass^{sa}$ and $\lambda(C) \maj \lambda(C')$, then $\ocnr(A) \subseteq \ocnr[C'](A)$ (\Cref{cor:majorization-c-numerical-range-inclusion}).
\item If $C \in \traceclass^{sa}$, $A \in \K$, then
  \begin{equation*}
    \closure{\ocnr(A)} = \set{ \trace(XA) \mid X \in \traceclass^{sa}, \lambda(X) \submaj \lambda(C)} = \ocnr[C \oplus \zop](A \oplus \zop)
  \end{equation*}
  as long as $\zop$ acts on a space of dimension at least $\rank C$
(\Cref{thm:compact-closed-iff-contains-weak-majorization} and \Cref{cor:compact-closure-direct-sum-zero}).
\item\label{item:viii} For $C \in \traceclass^+$, $\ocnr(A)$ is closed if for every $\theta$, $\rank (\Re(e^{i\theta}A)-m_{\theta}I)_+ \ge \rank C$, where $m_{\theta} := \max \essspec(\Re(e^{i\theta} A))$ (\Cref{thm:ocnr-closed-rank-condition}).
\item If $C \in \traceclass^+$, then (\Cref{thm:direct-sum-characterization})
  \begin{equation*}
    \ocnr(A_1 \oplus A_2) = \conv \quad \bigcup_{\mathclap{\quad C_1 \oplus C_2 \in \orbit(C)}} \ \big( \ocnr[C_1](A_1) + \ocnr[C_2](A_2) \big).
  \end{equation*}
\item If $C \in \traceclass^+$, $A \in \K$ normal, then $\ocnr(A) = \conv \ocspec(A)$ (\Cref{thm:normal-convex-c-spectrum}).
\item If $C \in \traceclass^+$ with $\dim \ker C \in \set{0,\infty}$, and $A \in B(\Hil)$ is diagonalizable, then $\cnr(A)$ is convex (\Cref{cor:c-numerical-range-convex-diagonalizable}).
\end{enumerate}
Many of the results listed above are extensions of facts which are known in either the finite dimensional or finite rank setting.
However, to our knowledge, \ref{item:viii} appears to be entirely new.

This paper is structured as follows.
In \cref{sec:notation} we specify some notation.
\Cref{sec:orbit-closed-c-numerical-range} contains fundamental properties of the orbit-closed $C$-numerical range for general trace-class operators $C$.
Then in \cref{sec:majorization-and-convexity} we restrict attention to selfadjoint $C$ and establish a characterization of the orbit-closed $C$-numerical range in terms of majorization (\Cref{thm:c-numerical-range-via-majorization}) which is the main theorem that paves the way for all our other primary results; it has as a direct corollary the convexity of the orbit-closed $C$-numerical range (\Cref{cor:c-numerical-range-convex}).
In \cref{sec:boundary-points} we undertake a thorough investigation of points on the boundary $\partial \ocnr(A)$, including an analysis specific to the case when $A$ is compact in subsection~\ref{subsec:compact-operators}.
We obtain necessary and sufficient conditions for $\ocnr(A)$ to be closed when $A$ is compact and $C$ is selfadjoint (\Cref{thm:compact-closed-iff-contains-weak-majorization}).
Beginning in subsection~\ref{subsec:bounded-operators} we restrict our attention to positive $C$ for the remainder of the paper, and there we provide a sufficient condition for $\ocnr(A)$ to be closed when $A \in B(\Hil)$ (\Cref{thm:ocnr-closed-rank-condition}).
In \cref{sec:oc-spectrum} we characterize the behavior of the orbit-closed $C$-numerical range under finite direct sums (\Cref{thm:direct-sum-characterization}) and prove $\ocnr(A) = \conv \ocspec(A)$ when $A$ is compact normal (\Cref{thm:normal-convex-c-spectrum}).
Finally, in \cref{sec:c-numerical-range-convexity} we use variations of the Schur--Horn theorem for positive compact operators to prove that the $C$-numerical range $W_C(A)$ is convex when $A$ is diagonalizable and $C$ is positive with either trivial or infinite dimensional kernel (\Cref{cor:c-numerical-range-convex-diagonalizable}), thereby providing a partial answer to the question \cite[Open Problem (b)]{DvE-2020-LaMA} posed by Dirr and vom Ende.

\section{Notation}
\label{sec:notation}

Let $\K$ denote the ideal of compact operators in $B(\Hil)$ and $\traceclass$ the ideal of trace-class operators, and $\K^{sa}, \traceclass^{sa}$ and $\K^+, \traceclass^+$ the selfadjoint and positive operators in these ideals.

For a compact operator $C$, let $\lambda(C)$ denote the \term{eigenvalue sequence} of $C$, that is $\lambda(C)$ is the sequence of eigenvalues of $C$ listed in order of decreasing modulus and repeated according to algebraic multiplicity, and concatenated with zeros if there are only finitely many eigenvalues; when $C$ is normal the algebraic and geometric multiplicities coincide.
Note that the sequence is not necessarily uniquely determined (since unequal eigenvalues may have the same modulus), and it omits any zero eigenvalue entirely if $C$ has infinitely many nonzero eigenvalues.

Let $\czstar$ denote the set of all nonnegative nonincreasing sequences converging to zero.
Given a nonnegative sequence $\lambda$ converging to zero (not necessarily monotone), the \term{monotonization} $\lambda^{*} \in \czstar$ of $\lambda$ is the measure-theoretic \term{nonincreasing rearrangement} relative to the counting measure on $\mathbb{N}$.
In other words, $\lambda^{*}_k$ is the $k$\textsuperscript{th} largest entry of $\lambda$ repeated according to multiplicity.
Note that if $\lambda$ has infinite support, then $\lambda^{*}$ is never zero.

For a real-valued sequence $\lambda$ converging to zero, it is often useful to ``split'' $\lambda$ into its positive and negative parts.
To this end, we define $\lambda^+ := (\max\set{\lambda,0})^{*}$, where the maximum is taken pointwise, and $\lambda^- := (-\lambda)^+$.
So the nonzero entries of $\lambda^+$ and $-\lambda^-$ are precisely the nonzero entries of $\lambda$, but it is possible that one of $\lambda^{\pm}$ maybe have zero entries which do not appear in the sequence $\lambda$.

When $C$ is a selfadjoint compact operator, we can apply the above splitting to the eigenvalue sequence $\lambda(C)$.
Then the nonzero entries of $\lambda^+(C)$ and $-\lambda^-(C)$ are precisely the nonzero entries of $\lambda(C)$, but it is possible that one of $\lambda^{\pm}(C)$ maybe have zero entries which are \emph{not} eigenvalues of $C$.
Indeed, this occurs when exactly one of $C_{\pm}$ is finite rank and $\ker C$ is trivial.
This is a technical issue which plays a minor role.

For a compact operator $C$, we denote by $s(C)$ the \term{singular value sequence} ($=\lambda(\abs{C})$), which for $C \in \K^+$ coincides with the eigenvalue sequence.
For a positive compact operator $C$, we will use $s(C)$ to refer to the eigenvalue sequence $\lambda(C)$ in order to emphasize positivity of the operator $C$.

For $A \in B(\Hil)$, we denote by $\ugroup(A)$ the unitary orbit of $A$ under the action of the unitary group $\ugroup$ by conjugation.
For a trace-class operator $C$, we will let $\orbit(C)$ denote the trace-norm closure of the unitary orbit $\ugroup(C)$, and we refer to $\orbit(C)$ as the \term{orbit} of $C$.

For $A \in B(\Hil)$, $\Re A, \Im A$ denote the real and imaginary parts of $A$, and $\spec(A), \ptspec(A), \essspec(A)$ are the spectrum, point spectrum and essential spectrum of $A$, respectively.
If $A$ is selfadjoint, then $A_+,A_-$ denote the positive and negative parts of $A$.
In addition, if $E \subseteq \reals$ is Borel, then $\chi_E(A)$ denotes the spectral projection of $A$ corresponding to the set $E$.

For a set $S$ in a (real or complex) vector space we let $\conv S$ denote the (not necessarily closed) \emph{convex hull} of $S$.
That is, $\conv S$ is the smallest convex set containing $S$.

\section{The orbit-closed $C$-numerical range}
\label{sec:orbit-closed-c-numerical-range}

When working in an infinite dimensional operator algebra such as $B(\Hil), \K$ or a type II factor, it is often important to substitute the unitary orbit $\ugroup(C)$ of an operator with its closure in an appropriate operator topology.
This appears repeatedly throughout the literature, especially in relation to majorization.
For example, Arveson and Kadison \cite{AK-2006-OTOAaA} considered $\orbit(C)$ for $C \in \traceclass^+$ when investigating diagonals of positive trace-class operators and the Schur--Horn theorem, which is a characterization of the diagonals in terms of majorization.
Likewise, when Kaftal and Weiss extended the Schur--Horn theorem to positive compact operators $C \in \K^+$, they implicitly provided their primary characterization in terms\footnote{In \cite{KW-2010-JFA}, this is actually stated in terms of the so-called \emph{partial isometry orbit, $\mathcal{V(C)}$}, but \cite[Proposition~2.1.12]{Lor-2016} guarantees that $\closure[\norm{\bigcdot}]{\ugroup(C)} = \mathcal{V}(C)$ for $C \in \K^+$.} of the norm closure $\closure[\norm{\bigcdot}]{\ugroup(C)}$ of the unitary orbit \cite[Proposition~6.4]{KW-2010-JFA}.
In addition, when Dykema and Skoufranis studied numerical ranges in II$_1$ factors \cite{DS-2017-PEMSIS}, they also used the norm closure of the unitary orbit.
For $C$ selfadjoint, the net effect of taking the closure in each of these situations is to make the eigenvalue sequence\footnote{in the case of II$_1$ factors, the analogous notion is the \emph{spectral scale}.} $\lambda(C)$ a complete invariant for the closure of the unitary orbit of $C$.
The reason this phenomenon does not appear in the finite dimensional setting, or even in the case of $C$ finite rank, is that the unitary orbit is already closed.
The next proposition is a generalization of \cite[Proposition~3.1]{AK-2006-OTOAaA} and makes all of this intuition precise.

\begin{proposition}
  \label{prop:orbit-closure-equivalences}
  If $C \in \K$ is a compact normal operator, then the following are equivalent.
  \begin{enumerate}
  \item\label{item:1} $X \in \closure[\norm{\bigcdot}]{\ugroup(C)}$; that is, $X$ is approximately unitarily equivalent to $C$.
  \item\label{item:2} $X$ is compact normal and $\lambda(X) = \lambda(C)$ (up to a suitable permutation).
  \item\label{item:3} $X \oplus \zop \in \ugroup(C \oplus \zop)$ where the size of $\zop$ is infinite.
  \end{enumerate}
  If in addition $C \in \traceclass$, then these are also equivalent to
  \begin{enumerate}[resume]
  \item\label{item:4} $X \in \orbit(C)$.
  \end{enumerate}
  When $C$ has finite rank, even if $C$ is not normal, then
  \begin{equation*}
    \overline{\ugroup(C)}^{\norm{\bigcdot}} = \orbit(C) = \ugroup(C).
  \end{equation*}
\end{proposition}

\begin{proof}
  \ref{item:1} $\Leftrightarrow$ \ref{item:2}.
  This is due to Gellar and Page \cite[Theorem~1]{GP-1974-DMJ} and the fact that all nonzero eigenvalues of a compact operator are isolated.
  
  \ref{item:2} $\Rightarrow$ \ref{item:3}.
  If $X,C$ are compact normal and $\lambda(X) = \lambda(C)$, then $X$ and $C$ have the same nonzero eigenvalues including multiplicity.
  Therefore $X \oplus \zop, C \oplus \zop$ not only have the same nonzero eigenvalues with multiplicity, but they also have zero as an eigenvalue of infinite multiplicity.
  Therefore $X \oplus \zop, C \oplus \zop$ are unitarily equivalent.

  \ref{item:3} $\Rightarrow$ \ref{item:2}.
  If $X \oplus \zop \in \ugroup(C \oplus \zop)$, then $X$ is compact normal since $C$ is also.
  Moreover, $\lambda(X) = \lambda(X \oplus \zop) = \lambda(C \oplus \zop) = \lambda(C)$.

  Now suppose that $C \in \traceclass$.

  \ref{item:4} $\Rightarrow$ \ref{item:1}.
  Trivial because the trace-norm topology on $\ugroup(C)$ is stronger than the operator norm topology.

  \ref{item:2} $\Rightarrow$ \ref{item:4}.
  Suppose that $X$ is compact normal and $\lambda(X) = \lambda(C)$.
  Clearly this implies that $X \in \traceclass$ since $C \in \traceclass$.
  Let $\epsilon > 0$ and take $N$ so that $\sum_{n=N+1}^{\infty} \abs{\lambda_n(C)} < \frac{\epsilon}{2}$.
  Since $X,C$ are normal and trace-class, they have orthonormal bases $\set{e_n}_{n=1}^{\infty}, \set{f_n}_{n=1}^{\infty}$ consisting of eigenvectors so that for $1 \le n \le N$, $X e_n = \lambda_n(X) e_n$ and $C f_n = \lambda_n(C) f_n$.
  Let $U$ be the unitary for which $Ue_n = f_n$.
  Then $UXU^{*},C$ are both diagonalized by the basis $\set{f_n}_{n=1}^{\infty}$.
  Therefore,
  \begin{equation*}
    \norm{UXU^{*} - C}_1 = \trace\big(\abs{UXU^{*} - C}\big) \le \sum_{n=N+1}^{\infty} \abs{\lambda_n(X)} + \sum_{n=N+1}^{\infty} \abs{\lambda_n(C)} < \epsilon.
  \end{equation*}
  Therefore $X \in \orbit(C)$.

  The claim for finite rank operators follows from the fact that the unitary orbit of an operator is norm closed if and only if the C*-algebra it generates is finite dimensional \cite[Proposition~2.4]{Voi-1976-RRMPA}, which is certainly the case for finite rank operators.
\end{proof}

Of particular importance to us here are the equivalences \ref{item:2} $\Leftrightarrow$ \ref{item:3} $\Leftrightarrow$ \ref{item:4} when $C$ is normal and trace-class, which we will make use of repeatedly throughout.

\begin{definition}
  \label{def:orbit-closed-c-numerical-range}
  Given a trace-class operator $C \in \traceclass$, we define the \term{orbit-closed $C$-numerical range} of an operator $A \in B(\Hil)$ by
  \begin{equation*}
    \ocnr(A) := \{ \trace(XA) \mid X \in \orbit(C) \}.
  \end{equation*}
\end{definition}

It is clear from the definition of the orbit-closed $C$-numerical range that $\cnr(A) \subseteq \ocnr(A)$ but the inclusion is, in general, strict as the next example shows.

\begin{example}
  \label{ex:cnr-not-equal-ocnr}
  Suppose $C$ is a strictly positive trace-class operator and $A$ is a positive operator with infinite dimensional kernel, then $0 \in \ocnr(A) \setminus \cnr(A)$.
  Indeed, if $X \in \ugroup(C)$, then $X$ is strictly positive and therefore $\trace(XA) = \trace(X^{\frac{1}{2}} A X^{\frac{1}{2}}) > 0$ since $X^{\frac{1}{2}} A X^{\frac{1}{2}}$ is a nonzero positive operator and the trace is faithful.
  Therefore $0 \notin \cnr(A)$ since $X \in \ugroup(C)$ was arbitrary.
  On the other hand, since $\ker A$ is infinite dimensional, there is some positive trace-class $X'$ which acts on $\ker A$ with $\lambda(X') = \lambda(C)$.
  Then $X := X' \oplus \zop_{\ker^{\perp}\!\!A}$ satisfies $\lambda(X) = \lambda(C)$, so $X \in \orbit(C)$ by \Cref{prop:orbit-closure-equivalences}.
  Moreover, $0 = \trace(XA) \in \ocnr(A)$.
\end{example}

By \Cref{prop:orbit-closure-equivalences}, for finite rank operators $\ugroup(C) = \orbit(C)$, and hence in this case we have equality $\ocnr(A) = \cnr(A)$.
In particular, if $P$ is a rank-$k$ projection, then $\ocnr[\frac{1}{k}P](A)$ is just the $k$-numerical range $\knr(A)$.
This, along with the following theorem, justifies our claim that the orbit-closed $C$-numerical range is a conservative modification of the $C$-numerical range.

\begin{theorem}
  \label{thm:cnr-dense-in-ocnr}
  If $C \in \traceclass$ is a trace-class operator and $A \in B(\Hil)$, then $\cnr(A)$ is dense in $\ocnr(A)$.
  In particular, $\closure{\cnr(A)} = \closure{\ocnr(A)}$.
\end{theorem}

\begin{proof}
  This is a direct consequence of the continuity of the map $(X,A) \mapsto \trace(XA)$ from $\traceclass \times B(\Hil) \to \complex$, where $\traceclass$ denotes the ideal of trace-class operators equipped with the trace norm.

  To be more specific, if $X \in \orbit(C)$, then there is a sequence of unitaries $U_n \in \ugroup$ such that $U_n C U_n^{*} \xrightarrow{\norm{\bigcdot}_1} X$.
  Then
  \begin{equation*}
    \abs{\trace(XA) - \trace(U_n C U_n^{*} A)} = \abs{\trace\big((X-U_n C U_n^{*})A\big)} \le \norm{X-U_n C U_n^{*}}_1 \norm{A}.
  \end{equation*}
  Therefore $\trace(U_n C U_n^{*} A) \to \trace(XA)$, proving that $\cnr(A)$ is dense in $\ocnr(A)$.
  Because the inclusion $\cnr(A) \subseteq \ocnr(A)$ is trivial, this yields $\closure{\cnr(A)} = \closure{\ocnr(A)}$.
\end{proof}

We note that in general $\ocnr(A)$ is not closed, so it is not simply the closure of $\cnr(A)$.
Indeed, when $C$ is a rank-one projection $\ocnr(A) = \nr(A)$, which need not be closed.

As a follow up to the previous theorem, we prove that the orbit-closed $C$-numerical range is a continuous function from pairs of operators (trace-class and bounded) to bounded subsets of the plane equipped with the Hausdorff distance $d_H$ which is only a pseudometric unless one restricts to compact sets.
The Hausdorff distance on bounded sets is defined as
\begin{equation*}
  d_H(Y,Z) := \max \left\{ \sup_{y\in Y} d(y,Z),\  \sup_{z \in Z} d(z,Y) \right\}.
\end{equation*}
As with any pseudometric, the Hausdorff distance $d_H$ generates a (ironically, non-Hausdorff) topological space whose basis consists of the open balls.
Since this topological space is not Hausdorff, limits are not unique, but two sets $Y,Z$ are limits of the same sequence if and only if $d_H(Y,Z) = 0$ if and only if $\closure{Y} = \closure{Z}$.
This latter fact about the closures follows immediately from the definition of $d_H$, which guarantees that two bounded sets have Hausdorff distance zero if and only if they have the same closure.

\begin{theorem}
  \label{thm:continuity}
  The function $(C,A) \mapsto \ocnr(A)$ from $\traceclass \times B(\Hil)$ equipped with the norm $\norm{(C,A)} = \norm{C}_1 + \norm{A}$ to bounded subsets of $\complex$ is continuous, where the latter is equipped with the Hausdorff pseudometric, denoted $d_H$.
  In fact, the function is Lipschitz in each variable separately with Lipschitz constant the norm (or trace norm) of the fixed operator.
  That is,
  \begin{equation*}
    d_H\big(\ocnr(A),\ocnr[C'](A)\big) \le \norm{C-C'}_1 \norm{A},
  \end{equation*}
  and
  \begin{equation*}
    d_H\big(\ocnr(A),\ocnr(A')\big) \le \norm{C}_1 \norm{A-A'}.
  \end{equation*}
\end{theorem}

\begin{proof}
  This is a direct consequence of the continuity of the map $(X,A) \mapsto \trace(XA)$.
  Indeed, notice that for any $X \in \orbit(C)$ and $A,A' \in B(\Hil)$, we have
  \begin{equation*}
    \abs{\trace(XA) - \trace(XA')} = \abs{\trace\big(X(A-A')\big)} \le \norm{X}_1 \norm{A-A'} = \norm{C}_1 \norm{A-A'}.
  \end{equation*}
  Since $\trace(XA), \trace(XA')$ represent arbitrary members of $\ocnr(A),\ocnr(A')$, we find
  \begin{equation*}
    d_H\big(\ocnr(A),\ocnr(A')\big) \le \norm{C}_1 \norm{A-A'}.
  \end{equation*}

  For the Lipschitz continuity in the other variable, notice that $d_H\big(\ocnr(A),\cnr(A)\big) = 0$ since these sets have the same closure by \Cref{thm:cnr-dense-in-ocnr}.
  So, it suffices to prove the result for the $C$-numerical range.
  Let $X \in \ugroup(C)$, so that $X = UCU^{*}$ for some unitary $U$.
  Then let $X' := UC'U^{*}$.
  Therefore
  \begin{equation*}
    \abs{\trace(XA) - \trace(X'A)} = \abs{\trace\big((X-X')A)\big)} \le \norm{X-X'}_1 \norm{A} = \norm{C-C'}_1 \norm{A}.
  \end{equation*}
  By a symmetric argument we obtain
  \begin{equation*}
    d_H\big(\cnr(A),\cnr[C'](A)\big) \le \norm{C-C'}_1 \norm{A},
  \end{equation*}
  and hence also
  \begin{equation*}
    d_H\big(\ocnr(A),\ocnr[C'](A)\big) \le \norm{C-C'}_1 \norm{A}. \qedhere
  \end{equation*}
\end{proof}

\begin{corollary}
  \label{cor:approximate-unitary-equivalence-same-closure}
  $\closure{\ocnr(A)} = \closure{\ocnr(A')}$ if $A,A'$ are approximately unitarily equivalent.
\end{corollary}

\begin{proof}
  Since $A,A'$ are approximately unitarily equivalent, there are unitaries $U_n$ such that $U_n A U_n^{*} \to A'$, and therefore $d_H \big( \ocnr(U_n A U_n^{*}), \ocnr(A') \big) \to 0$ by \Cref{thm:continuity}.
  However, $\ocnr(A) = \ocnr(U_n A U_n^{*})$ since conjugation by the unitary $U_n$ may be absorbed into $\orbit(C)$, whence $d_H \big( \ocnr(A), \ocnr(A') \big) = 0$.
  Thus $\closure{\ocnr(A)} = \closure{\ocnr(A')}$.
\end{proof}

The following proposition provides some basic facts concerning the orbit-closed $C$-numerical range, all of which follow easily from fundamental properties of the trace.

\begin{proposition}
  \label{prop:c-numerical-range-basics}
  Given a trace-class operator $C \in \traceclass$, $A \in B(\Hil)$ and $a,b \in \complex$,
  \begin{enumerate}
  \item\label{item:positivity-ocnr} if $A \in B(\Hil)^+$ and $C \in \traceclass^+$, then $\ocnr(A) \subseteq [0,\infty)$;
  \item\label{item:hermitian-ocnr} if $C$ is selfadjoint, then for any $X \in \orbit(C)$, $\Re(\trace(XA)) = \trace(X\Re A)$, and so $\Re \ocnr(A) = \ocnr(\Re A)$;
  \item\label{item:similarity-preserving-ocnr} $\ocnr(aI+bA) = a \trace C + b \ocnr(A)$.
  \end{enumerate}
  Moreover, the same results hold for $\cnr(A)$.
\end{proposition}

\begin{proof}
  \ref{item:positivity-ocnr}.
  Consider $X \in \orbit(C)$, so that $X$ is positive and trace-class.
  If $A$ is also positive, then
  \begin{equation*}
    \trace(XA) = \trace(X^{\frac{1}{2}}AX^{\frac{1}{2}}) \ge 0,
  \end{equation*}
  since the trace is a positive linear functional.

  \ref{item:hermitian-ocnr}.
  If $C = C^{*}$, then for any $X \in \orbit(C)$ we have $X = X^{*}$.
  Therefore
  \begin{equation*}
    \trace(XA) + \overline{\trace(XA)} = \trace(XA) + \trace(A^{*}X^{*}) = \trace(XA) + \trace(XA^{*}) = \trace\big(X(A + A^{*})\big).
  \end{equation*}

  \ref{item:similarity-preserving-ocnr}.
  Note that $\trace\big(X(aI+bA)\big) = a \trace X + b \trace(XA)$, and since $\trace X = \trace C$ (because $X \in \orbit(C)$), we obtain $\ocnr(aI+bA) = a \trace C + b \ocnr(A)$.

  Of course, a simple examination of the above proof allows one to conclude that everything works for $X \in \ugroup(C)$.
  Therefore, these results also apply to $\cnr(A)$.
\end{proof}

\section{Majorization and convexity}
\label{sec:majorization-and-convexity}

In this section we establish our main theorem which characterizes the orbit-closed $C$-numerical range for $C$ selfadjoint in terms of majorization (\Cref{thm:c-numerical-range-via-majorization}), which directly yields convexity (\Cref{cor:c-numerical-range-convex}).
We begin by recalling the notion of majorization.

\begin{definition}
  \label{def:sequence-majorization}
  Given nonnegative sequences $d,\lambda$ converging to zero, we say that $d$ is \term{submajorized} by $\lambda$ and write $d \submaj \lambda$ if, for all $n \in \mathbb{N}$,
  \begin{equation*}
    \sum_{k=1}^n d^{*}_k \le \sum_{k=1}^n \lambda^{*}_k.
  \end{equation*}
  If, in addition, equality of the sums holds when $n=\infty$ (including the possibility that both sums are infinite), we say that $d$ is \term{majorized} by $\lambda$ and write $d \maj \lambda$.
  If equality of the sums holds for infinitely many $n \in \nats$, we say that $d$ is \term{block majorized} by $\lambda$.

  For real-valued sequences $d,\lambda \in \ell^1$, we say that $d$ is \term{submajorized} by $\lambda$, and write $d \submaj \lambda$, if $d^+ \submaj \lambda^+$ and $d^- \submaj \lambda^-$.
  If in addition there is equality for $\sum_{k=1}^{\infty} d_k = \sum_{k=1}^{\infty} \lambda_k$, we say that $d$ is \term{majorized} by $\lambda$, and we write $d \maj \lambda$.
\end{definition}

The reader should take note: if $d,\lambda \in \ell^1$ are real-valued sequences, $d \maj \lambda$ is strictly weaker than satisfying both $d^+ \maj \lambda^+$ and $d^- \maj \lambda^-$.
For example, the zero sequence is majorized by any sequence in $\ell^1$ whose sum is zero.

The next two results are due to Hiai and Nakamura in \cite{HN-1991-TAMS} and link majorization and submajorization to the closed convex hulls of unitary orbits in various operator topologies.
Their results apply in von Neumann algebras more generally, not just $B(\Hil)$, so we are stating simplified versions for our own needs.

\begin{proposition}[\protect{\cite[Theorem~3.3]{HN-1991-TAMS}}]
  \label{prop:wot-closure-convex-orbit-weak-majorization}
  For a selfadjoint compact operator $C \in \K^{sa}$,
  \begin{equation*}
    \{ X \in \K^{sa} \mid \lambda(X) \submaj \lambda(C) \} = \closure[\mathrm{wot}]{\conv\ugroup(C)} = \closure[\norm{\bigcdot}]{\conv\ugroup(C)}.
  \end{equation*}
\end{proposition}

Note that for $C$ trace-class, since the trace-norm topology on $\conv \ugroup(C)$ is stronger than the norm topology (or the weak operator topology), we may replace $\ugroup(C)$ in \Cref{prop:wot-closure-convex-orbit-weak-majorization} with $\orbit(C)$.

\begin{proposition}[\protect{\cite[Theorem~3.5(4)]{HN-1991-TAMS}}]
  \label{prop:convex-orbit-majorization}
  For a selfadjoint trace-class operator $C \in \traceclass^{sa}$,
  \begin{equation*}
    \{ X \in \traceclass^{sa} \mid \lambda(X) \maj \lambda(C) \} = \closure[\norm{\bigcdot}_1]{\conv\ugroup(C)} = \closure[\norm{\bigcdot}_1]{\conv\orbit(C)}.
  \end{equation*}
\end{proposition}

Before we begin the proof of the main theorem of this section, which has as a corollary that $\ocnr(A)$ is convex when $C$ is selfadjoint, we must prove a key technical lemma.
This lemma says in a rather strong way that the extreme points of $\set{ X \in \traceclass^{sa} \mid \lambda(X) \maj \lambda(C) }$ form a subset of $\orbit(C)$.

\begin{lemma}
  \label{lem:majorization-trace-zero-perturbation}
  Suppose that $X,C \in \traceclass^{sa}$ with $\lambda(X) \maj \lambda(C)$ but $X \notin \orbit(C)$.
  Then there is a nonzero projection $P$ of rank at least $2$ and an $\epsilon > 0$ such that $\lambda(X+S) \maj \lambda(C)$ for any selfadjoint $S$ with $S = PS = SP$, $\trace S = 0$ and $\norm{S} < \epsilon$.
\end{lemma}

\begin{proof}
  Suppose that $X,C \in \traceclass^{sa}$ with $\lambda(X) \maj \lambda(C)$ but $X \notin \orbit(C)$.
  There are two distinct cases, when $\trace X_+ = \trace C_+$ (necessitating $\trace X_- = \trace C_-$) and when $\trace X_+ < \trace C_+$ (necessitating $\trace X_- < \trace C_-$).

  \begin{case}{$\trace X_+ = \trace C_+$.}
    Since $X \notin \orbit(C)$, we have $\lambda(X) \not= \lambda(C)$ by \Cref{prop:orbit-closure-equivalences}, and hence either $\lambda^+(X) \not= \lambda^+(C)$ or $\lambda^-(X) \not= \lambda^-(C)$.
    Without loss of generality we may assume the former.
    So, in this case it suffices to prove the result when $X,C \in \traceclass^+$ and $s(X) \maj s(C)$ but $s(X) \not= s(C)$ since $s(X) = \lambda(X) = \lambda^+(X)$ for \emph{positive} compact operators.

    Let $n \in \mathbb{N}$ be the first index for which the sequences $s(X), s(C)$ differ.
    Necessarily $s_n(X) < s_n(C)$.
    Moreover, $s_{n+1}(X) > 0$ since
    \begin{equation*}
      \sum_{j=1}^n s_j(X) < \sum_{j=1}^n s_j(C) \quad\text{necessitates}\quad \sum_{j=n+1}^\infty s_j(X) > \sum_{j=n+1}^\infty s_j(C).
    \end{equation*}
    Let $m \ge n+1$ be the first index such that $s_{m+1}(X) < s_{n+1}(X)$, and hence $s_{n+1}(X) = s_{n+2}(X) = \cdots = s_m(X)$.
    Such an index $m$ occurs because $s(X) \in \czstar$ and $s_{n+1}(X) > 0$.
    Let $\delta_k := \sum_{j = 1}^k \big( s_j(C) - s_j(X) \big)$ and note that $\delta_k \ge 0$ for all $k \in \nats$ since $s(X) \maj s(C)$.
    Also $\delta_k = 0$ for $1 \le k < n$, and $\delta_n = s_n(C) - s_n(X) > 0$.
    Moreover, for $n < k \le m$, $s_k(X)$ is constant ($=s_{n+1}(X)$) and therefore on the interval $n \le k \le m$, $\delta_k$ is increasing (as long as $s_k(C) \ge s_{n+1}(X)$) and then (maybe) strictly decreasing (if/once $s_k(C) < s_{n+1}(X)$).
    Consequently, $\delta_{m-1} > 0$, as it is either greater than or equal to $\delta_n > 0$ or strictly greater than $\delta_m \ge 0$.
    Furthermore, $\min_{n \le k < m} \delta_k = \min \set{\delta_n, \delta_{m-1}} > 0$.

    Set $\epsilon := \min \set{ \delta_n, \delta_{m-1}, s_m(X)-s_{m+1}(X) }$.
    Note that
    \begin{equation*}
      \delta_n = s_n(C) - s_n(X) \le s_{n-1}(C) - s_n(X) = s_{n-1}(X) - s_n(X).
    \end{equation*}
    Let $\{e_j\}_{j=1}^{\infty}$ be an orthonormal set of eigenvectors for $X$ corresponding to the eigenvalues in the sequence $s(X)$.
    Let $P$ be the projection onto $\spans\{e_n,e_m\}$.
    Let $S$ be any selfadjoint operator with $S = PS = SP$, $\trace(S) = 0$, and $\norm{S} < \epsilon$.
    Then because $S = PS = SP$, if $m \not= j \not= n$, then $Se_j = SPe_j = 0$, and hence $(X+S)e_j = s_j(X) e_j$.
    Moreover, because $\norm{S} < \epsilon$ and $\trace(S) = 0$, $(X+S)f_n = (s_n(X) + \eta) f_n$ and $(X+S)f_m = (s_m(X) - \eta) f_m$, for some $\abs{\eta} \le \norm{S} < \epsilon$ and orthonormal vectors $f_n,f_m$ with $\spans \set{f_n,f_m} = \spans \set{e_n,e_m}$.

    We will establish $s(X+S) \maj s(C)$, and we deal with the case when $\eta \ge 0$ first because it implies the case when $\eta \le 0$.
    Notice that
    \begin{equation*}
      s_{n+1}(X) \le s_n(X) + \eta < s_n(X) + \big(s_{n-1}(X) - s_n(X)\big) = s_{n-1}(X),
    \end{equation*}
    and also
    \begin{equation*}
      s_{m-1}(X) \ge s_m(X) - \eta > s_m(X) - \big(s_m(X) - s_{m+1}(X)\big) = s_{m+1}(X).
    \end{equation*}
    Therefore the order of the singular values is preserved between $X$ and $X+S$;
    in particular, $s_n(X+S) = s_n(X) + \eta$ and $s_m(X+S) = s_m(X) - \eta$ and $s_j(X+S) = s_j(X)$ for all $n \not= j \not= m$.
    Thus to ensure $s(X+S) \maj s(C)$ we only need to check the partial sums for indices $n \le k < m$, because for all other values of $k$, $\sum_{j=1}^k s(X+S) = \sum_{j=1}^k s(X)$ and $s(X) \maj s(C)$.

    So for any $n \le k < m$, we have
    \begin{align*}
      \sum_{j=1}^k s_j(X+S) &= \sum_{j=1}^{n-1} s_j(X+S) + s_n(X+S) + \sum_{j=n+1}^k s_j(X+S) \\
                            &= \sum_{j=1}^{n-1} s_j(X) + (s_n(X) + \eta) + \sum_{j=n+1}^k s_j(X) \\
                            &= \sum_{j=1}^{n-1} s_j(C) + (s_n(C) + \eta) + \sum_{j=n+1}^k s_j(C) - \delta_k \\
                            &\le \sum_{j=1}^k s_j(C).
    \end{align*}
    where the last line follows because $\eta - \delta_k \le \epsilon - \min \set{\delta_n,\delta_{m-1}} \le 0$.
    Thus $s(X+S) \maj s(C)$.

    Now suppose $\eta \le 0$.
    In this case it is clear that the sequence with $\eta$ is majorized by the same sequence with $\abs{\eta}$.
    Indeed, this is due to a fundamental fact about majorization: given a decreasing nonnegative sequence $(d_j)_{j=1}^{\infty}$, if $n < m$ and $d_n > d_m$, and we consider the sequence $(d'_j)_{j=1}^{\infty}$ which is equal to the original sequence except that $d'_n = d_n - \epsilon$ and $d'_m = d_m + \epsilon$ for some $0 \le \epsilon \le d_n - d_m$, then $(d'_j)_{j=1}^{\infty} \prec (d_j)_{j=1}^{\infty}$, and this happens even if the decreasing order is no longer preserved for the sequence $(d'_j)_{j=1}^{\infty}$.
    In our case, for $\eta \le 0$, we are using $d_n = s_n(X) - \eta$, $d_m = s_n(X) + \eta$ and $\epsilon = 2\abs{\eta}$.
    Therefore we still obtain $s(X+S) \maj s(C)$ even for $\eta \le 0$.
  \end{case}

  \begin{case}{$\trace X_+ < \trace C_+$ and both $X_{\pm}$ are finite rank.}
    Since $\trace X = \trace C$, we must also have $\trace X_- < \trace C_-$.
    If $X_{\pm}$ are finite rank with ranks $n_{\pm}$, then $X$ has two orthonormal eigenvectors $e_{\pm}$ corresponding to the eigenvalue zero.
    Set $\eta_{\pm} := \sum_{n=1}^{n_{\pm}} \big( \lambda^{\pm}_n(C) - \lambda^{\pm}_n(X) \big) + \lambda^{\pm}_{n_{\pm}+1}(C) > 0$ (if either $\eta_{\pm}$ were zero, it would imply both $\trace X_{\pm} = \trace C_{\pm}$).
    Then set $\epsilon := \min \set{\eta_{\pm}, \lambda^{\pm}_{n_{\pm}}(X)}$ and let $P$ be the projection onto $\spans\set{e_+,e_-}$.

    Let $S$ be any selfadjoint operator for which $SP = PS = S$ and $\norm{S} < \epsilon$ and $\trace S = 0$.
    Adding $S$ to $X$ produces two new eigenvalues smaller in modulus than the rest.
    That is, for $1 \le n \le n_{\pm}$, $\lambda^{\pm}_n(X+S) = \lambda^{\pm}_n(X) = \lambda^{\pm}_n(C)$ and $\lambda^{\pm}_{n_{\pm} + 1}(X+S) = \norm{S}$, and $\lambda^{\pm}_n(X+S) = 0$ for all $n > n_{\pm} + 1$.
    Therefore, to see that $\lambda^{\pm}(X+S) \submaj \lambda^{\pm}(C)$, it suffices to check the partial sums for the indices $n_{\pm}+1$.
    Thus,
    \begin{equation*}
      \sum_{n=1}^{n_{\pm} + 1} \big( \lambda^{\pm}_n(C) - \lambda^{\pm}_n(X+S) \big) = \sum_{n=1}^{n_{\pm}} \big( \lambda^{\pm}_n(C) - \lambda^{\pm}_n(X) \big) + \lambda^{\pm}_{n_{\pm} + 1}(C) - \norm{S} \ge \eta_{\pm} - \epsilon \ge 0.
    \end{equation*}
    Finally, since $\trace(X+S) = \trace X + \trace S = \trace C$, we obtain $\lambda^{\pm}(X+S) \maj \lambda^{\pm}(C)$.
  \end{case}

  \begin{case}{$\trace X_+ < \trace C_+$ and one of $X_{\pm}$ is infinite rank.}
    Again $\trace X = \trace C$, we must also have $\trace X_- < \trace C_-$.
    By symmetry, we may assume without loss of generality that $X_+$ has infinite rank.
    Then set $\gamma := \frac{1}{2}\trace (C_+ - X_+) > 0$ and $n_+ \in \nats$ such that for all $N \ge n_+$ we have $\sum_{n=1}^N \big( \lambda^+_n(C) - \lambda^+_n(X) \big) \ge \gamma$.
    Moreover, since $X_+$ is infinite rank, we may select $r > m \ge n_+$ such that $\lambda^+_{m-1}(X) > \lambda^+_m(X) \ge \lambda^+_r(X) > \lambda^+_{r+1}(X)$.
    Set $\epsilon := \min\set{\gamma,\lambda^+_{m-1}(X) - \lambda^+_m(X), \lambda^+_r(X) - \lambda^+_{r+1}(X)}$ and $P$ the projection onto $\spans\set{e_m,e_r}$.

    Let $S$ be a selfadjoint operator for which $SP = PS = S$ and $\norm{S} < \epsilon$ and $\trace S = 0$.
    As with Case 1, $\lambda^+_n(X+S) = \lambda^+_n(X)$ for all $m \not= n \not= r$, and $\lambda^+_m(X+S) = \lambda^+_m(X) + \eta$ and $\lambda^+_r(X+S) = \lambda^+_r(X) - \eta$ for some $0 \le \eta \le \norm{S} < \epsilon$ (the situation when $\eta \le 0$ is handled in the same manner as Case 1).
    Of course, $\lambda^-(X+S) = \lambda^-(X)$.
    To verify that $\lambda^+(X+S) \submaj \lambda^+(C)$, it suffices to check the partial sums for indices $m \le k < r$.
    We obtain
    \begin{align*}
      \sum_{j=1}^k \lambda^+_j(X+S) &= \sum_{j=1}^{m-1} \lambda^+_j(X+S) +  \lambda^+_m(X+S) + \sum_{j=m+1}^k \lambda^+_j(X+S) \\
                                    &= \sum_{j=1}^k \lambda^+_j(X) + \eta \\
                                    &\le \sum_{j=1}^k \lambda^+_j(C) - \gamma + \eta \\
                                    &\le \sum_{j=1}^k \lambda^+_j(C), 
    \end{align*}
    where the last line follows since $\eta \le \norm{S} < \epsilon \le \gamma$.
    Thus $\lambda^{\pm}(X+S) \submaj \lambda^{\pm}(C)$ and $\trace(X+S) = \trace X = \trace C$, so $\lambda^{\pm}(X+S) \maj \lambda^{\pm}(C)$.
    \qedhere
  \end{case}
\end{proof}

We now have the tools necessary (\Cref{prop:convex-orbit-majorization} and \Cref{lem:majorization-trace-zero-perturbation}) to prove our main theorem.
The proof is adapted from and follows closely the one given by Dykema and Skoufranis \cite[Theorem~2.14]{DS-2017-PEMSIS} for numerical ranges in II$_1$ factors, but there is one substantial difference.
A key step in the proof is obtaining an extreme point of a certain closed convex subset of $\{ X \in \traceclass^{sa} \mid \lambda(X) \maj \lambda(C) \}$.
In the context of II$_1$ factors (or in any finite factor), this set happens to be weak* compact and so Dykema and Skoufranis are able to employ the Krein--Milman Theorem.
However, in $B(\Hil)$, this set is definitely not weak* compact since it contains elements of arbitrarily small norm and therefore the zero operator is in the weak* closure.
Instead, $\{ X \in \traceclass^{sa} \mid \lambda(X) \maj \lambda(C) \}$ is only a trace-norm closed and bounded convex set, and so the Krein--Milman Theorem cannot be invoked.
In order to circumvent this issue, we use the Radon--Nikodym Property of the Banach space of trace-class operators to obtain the desired extreme point.

\begin{theorem}
  \label{thm:c-numerical-range-via-majorization}
  For a selfadjoint trace-class operator $C \in \traceclass^{sa}$ and any $A \in B(\Hil)$,
  \begin{equation*}
    \ocnr(A) = \{ \trace(XA) \mid X \in \traceclass^{sa}, \lambda(X) \maj \lambda(C) \}.
  \end{equation*}
\end{theorem}

\begin{proof}
  Given $X \in \traceclass^{sa}$ with $\lambda(X) \maj \lambda(C)$ we will show there is some $Y \in \orbit(C)$ for which $\trace(XA) = \trace(YA)$.
  For this consider the trace-norm continuous map $\Phi : Z \mapsto \trace(ZA)$ from $\{ Z \in \traceclass^{sa} \mid \lambda(Z) \maj \lambda(C) \}$ to $\complex$.
  Then by \Cref{prop:convex-orbit-majorization} and continuity and linearity of $\Phi$, the set
  \begin{equation*}
    \Phi^{-1}(\trace(XA)) = \{ Z \in \traceclass^{sa} \mid \lambda(Z) \maj \lambda(C), \trace(ZA) = \trace (XA) \}
  \end{equation*}
  is a nonempty, convex, trace-norm closed and bounded set.
  The trace-class operators with the trace-norm form a Banach space, and moreover, this space has the Radon--Nikodym Property \cite[Lemma~2]{Chu-1981-JLMSIS}.
  It is well-known (due to Lindenstrauss \cite[Theorem~2]{Phe-1974-JFA}) that the Radon--Nikodym Property implies the Krein--Milman Property: every convex, closed and bounded set is the closed convex hull of its extreme points.
  In particular, $\Phi^{-1}(\trace(XA))$ has an extreme point, which we label $Y$.

  We claim that $Y \in \orbit(C)$.
  Suppose not.
  Then we may apply \autoref{lem:majorization-trace-zero-perturbation} to obtain a nonzero projection $P$ as in that lemma.
  Consider the real vector space $\mathcal{S}_P := \{ S \in B(\Hil) \mid S= S^{*}, S = SP = PS, \trace S = 0 \}$ and the linear map $S \mapsto \trace(SA)$.
  Note that $\mathcal{S}_P$ has dimension at least two since $P$ must have rank at least two, and therefore this linear map has a nonzero element in the kernel.
  By scaling we obtain an $S \in \mathcal{S}_P$ in the kernel of this map for which $\lambda(Y \pm S) \maj \lambda(C)$ by \autoref{lem:majorization-trace-zero-perturbation}.
  Thus $\trace((Y\pm S)A) = \trace(YA) = \trace(XA)$, and therefore $Y \pm S \in \Phi^{-1}(\trace(XA))$, and hence
  \begin{equation*}
    Y = \frac{Y+S}{2} + \frac{Y-S}{2},
  \end{equation*}
  contradicting the fact that $Y$ is extreme in $\Phi^{-1}(\trace(XA))$.
  Thus $Y \in \orbit(C)$.

  Therefore $\{ \trace(XA) \mid X \in \traceclass^{sa}, \lambda(X) \maj \lambda(C) \} \subseteq \ocnr(A)$, and the other inclusion follows since $\orbit(C) \subseteq \{ X \in \traceclass^{sa} \mid \lambda(X) \maj \lambda(C) \}$.
\end{proof}

Since the collection $\{ X \in \traceclass^{sa} \mid \lambda(X) \maj \lambda(C) \}$ is convex (e.g., by \Cref{prop:convex-orbit-majorization}, but this can also be proven directly rather easily) and the map $X \mapsto \trace(XA)$ is linear, it is clear that $\ocnr(A)$ is a convex set.
This generalizes \cite{Wes-1975-LMA} and is, to our knowledge, the only independent proof of this result when the underlying Hilbert space is infinite dimensional.

\begin{corollary}
  \label{cor:c-numerical-range-convex}
  If $C \in \traceclass^{sa}$, then $\ocnr(A)$ is convex.
\end{corollary}

We remark that combining \Cref{cor:c-numerical-range-convex} with \Cref{thm:cnr-dense-in-ocnr} yields an independent proof of \cite[Theorem~3.8]{DvE-2020-LaMA} that $\closure{\cnr(A)}$ is convex under the stated hypothesis that $C \in \traceclass^{sa}$.

In addition, \Cref{thm:c-numerical-range-via-majorization} has as a direct corollary the following inclusion relationship among orbit-closed $C$-numerical ranges.
This extends \cite[Theorem~7]{GS-1977-LAA} to the infinite dimensional setting.

\begin{corollary}
  \label{cor:majorization-c-numerical-range-inclusion}
  If $C,C' \in \traceclass^{sa}$ and $\lambda(C') \maj \lambda(C)$, then $\ocnr[C'](A) \subseteq \ocnr(A)$.
\end{corollary}

\section{Boundary points}
\label{sec:boundary-points}

In the study of numerical ranges, it is often of interest to investigate the boundary.
We now determine some conditions under which points on the boundary $\partial\ocnr(A)$ actually belong to $\ocnr(A)$.
In general, this is a nontrivial question, but in this section we try to provide adequate answers.

\subsection{Compact operators}
\label{subsec:compact-operators}

We begin with the case when $A \in \K$ is a compact operator.
For this we have a very satisfying set of conditions in \Cref{thm:compact-closed-iff-contains-weak-majorization} equivalent to $\ocnr(A)$ being closed, and a simpler sufficient (but not necessary) condition in \Cref{cor:compact-closure-direct-sum-zero}.

In \Cref{thm:c-numerical-range-via-majorization}, we saw that the orbit-closed $C$-numerical range was the image of the operators whose eigenvalue sequences are majorized by $\lambda(C)$ under the map $X \mapsto \trace(XA)$.
A natural question to ask is whether or not there is a similar characterization for the image of those operators which are only submajorized by $C$.
The following lemma proves that this is indeed the case.

\begin{lemma}
  \label{lem:weak-majorization-c-numerical-range}
  For a selfadjoint trace-class operator $C \in \traceclass^{sa}$ and an operator $A \in B(\Hil)$,
  \begin{equation*}
    \set{ \trace(XA) \mid X \in \traceclass^{sa}, \lambda(X) \submaj \lambda(C) } = \conv \ \bigcup_{\mathclap{\qquad 0 \le m_{\pm} \le \rank C_{\pm}}} \  \ocnr[C_{m_-,m_+}](A),
  \end{equation*}
  where $C_{m_-,m_+}$ is the operator $C (P^-_{m_-} + P^+_{m_+})$ where $\trace P^{\pm}_{m_{\pm}} = m_{\pm}$, and for some $\lambda_- \le 0 \le \lambda_+$, $\chi_{(-\infty,\lambda_-)}(C) \le P^-_{m_-} \le \chi_{(-\infty,\lambda_-]}(C)$ and $\chi_{(\lambda_+,\infty)}(C) \le P^+_{m_+} \le \chi_{[\lambda_+,\infty)}(C)$.

  In other words, $C_{m_-,m_+}$ is the selfadjoint operator whose eigenvalues are the smallest $m_-$ negative eigenvalues $C$ along with the largest $m_+$ positive eigenvalues of $C$, namely $-\lambda^-_1(C),\ldots,-\lambda^-_{m_-}(C)$ and $\lambda^+_1,(C),\ldots,\lambda^+_{m_+}(C)$, along with the eigenvalue $0$ repeated with multiplicity $\trace(I-P_{m_-}^- - P_{m_+}^+)$.
\end{lemma}

\begin{proof}
  Notice that the set $\set{ X \in \traceclass^{sa} \mid \lambda(X) \submaj \lambda(C) }$ is convex (e.g., by \Cref{prop:wot-closure-convex-orbit-weak-majorization}, but this can also be proven directly) and the trace is a linear functional, hence the set $\set{ \trace(XA) \mid X \in \traceclass^{sa}, \lambda(X) \submaj \lambda(C) }$ is convex.
  Moreover, any $X \in \orbit(C_{m_-,m_+})$ satisfies $\lambda(X)  = \lambda(C_{m_-,m_+}) \submaj \lambda(C)$.
  Therefore the right-hand set is included in the left-hand set.

  For the other inclusion, take any $X \in \traceclass^{sa}$ with $\lambda(X) \submaj \lambda(C)$.
  If $\trace(X_+) = \trace(C_+)$, set $m_+ = \rank C_+$, and if $\trace(X_-) = \trace(C_-)$, set $m_- = \rank C_-$ (we allow for $m_{\pm} = \infty$).
  Otherwise, let $m_{\pm} \in \nats$ be the smallest (and unique) positive integers for which
  \begin{equation*}
    \sum_{n=1}^{m_{\pm}-1} \lambda^{\pm}_n(C) \le \trace(X_{\pm}) < \sum_{n=1}^{m_{\pm}} \lambda^{\pm}_n(C).
  \end{equation*}
  Then there are $t_{\pm} \in [0,1]$ for which
  \begin{equation*}
    \trace(X_{\pm}) = \sum_{n=1}^{m_{\pm} - 1} \lambda^{\pm}_n(C) + t_{\pm} \lambda^{\pm}_{m_{\pm}}(C).
  \end{equation*}
  Then consider the operator $C'$ which is the convex combination
  \begin{equation*}
    (1-t_-)(1-t_+) C_{m_- - 1,m_+ - 1} + (1-t_-)t_+ C_{m_- - 1,m_+} + t_-(1-t_+) C_{m_- ,m_+ - 1} + t_- t_+ C_{m_-,m_+}.
  \end{equation*}
  Here, for convenience, we simply adopt the convention that $\infty = \infty - 1$ in case either of $m_{\pm}$ is infinite.
  Therefore, the nonzero eigenvalues of $C'$ are $\lambda^{\pm}_1(C),\ldots,\lambda^{\pm}_{m_{\pm} - 1}(C),t_{\pm} \lambda^{\pm}_{m_{\pm}}(C)$ (or if $m_+ = \infty$, the positive eigenvalues are just $\lambda^+(C)$, and similarly for when $m_- = \infty$).

  The operator $C'$ was constructed specifically so that $\lambda(X) \maj \lambda(C')$.
  Therefore, by \Cref{thm:c-numerical-range-via-majorization}
  \begin{equation*}
    \trace(XA) \in \ocnr[C'](A) \subseteq \conv \ \bigcup_{\mathclap{\qquad 0 \le k_{\pm} \le \rank C_{\pm}}} \ \ocnr[C_{k_-,k_+}](A)
  \end{equation*}
  where the second inclusion holds because any $X' \in \orbit(C')$ is a convex combination of four $X_{k_-,k_+} \in \orbit(C_{k_-,k_+})$.
  To see this, notice that the equation defining $C'$ actually establishes that $\lambda(C')$ is a convex combination of four appropriately permuted $\lambda(C_{k_-,k_+})$ (with up to two zeros added to any of these sequences).
  Then by \Cref{prop:orbit-closure-equivalences} any $X' \in \orbit(C')$ has the form $\diag(\lambda(C')) \oplus \zop$ in some basis for an appropriately sized $\zop$, and this is clearly a convex combination of the same four $\diag(\lambda(C_{k_-,k_+})) \oplus \zop \in \orbit(C_{k_-,k_+})$.
\end{proof}

The following theorem provides a complete characterization of when the orbit-closed $C$-numerical range of a compact operator is closed in terms of submajorization.
The equivalence \ref{item:positivity-ocnr} $\Leftrightarrow$ \ref{item:similarity-preserving-ocnr} generalizes \cite[Theorem~1(i)]{dBGS-1972-JLMSIS} for the standard numerical range and \cite[Result~(2.5)]{LP-1995-FDaaMaE} for finite rank $C$.
The proof of \cite[Theorem~1(i)]{dBGS-1972-JLMSIS} utilized weak sequential compactness of the unit ball in $\Hil$ in order to obtain the requisite limit vector, whereas \cite[Result~(2.5)]{LP-1995-FDaaMaE} applied the weak operator topology compactness of the unit ball of $B(\Hil)$ to obtain the limiting operator.
Our proof below shows that the true essence of this phenomenon actually takes place relative to a different topology.
In particular, the key is the nontrivial weak* compactness of $\set{ X \in \traceclass^{sa} \mid \lambda(X) \submaj \lambda(C) }$, where this set is viewed not as a subset of $B(\Hil)$, but as a subset of $\traceclass \cong \K^{*}$ which is why the condition that $A \in \K$ is essential for these limit processes.

\begin{theorem}
  \label{thm:compact-closed-iff-contains-weak-majorization}
  Let $C \in \traceclass^{sa}$ be a selfadjoint trace-class operator and let $A \in \K$ be a compact operator.
  Then
  \begin{equation*}
    \closure{\ocnr(A)} = \set{ \trace(XA) \mid X \in \traceclass^{sa}, \lambda(X) \submaj \lambda(C) }.
  \end{equation*}
  Consequently, the following are equivalent.
  \begin{enumerate}
  \item\label{item:ocnr-compact-closed} $\ocnr(A)$ is closed.
  \item\label{item:ocnr-equal-weak-majorization} $\ocnr(A) = \set{ \trace(XA) \mid X \in \traceclass^{sa}, \lambda(X) \submaj \lambda(C) }$.
  \item\label{item:ocnr-contains-initial-ocnrs} $\ocnr(A) \supseteq \ocnr[C_{m_-,m_+}](A)$ for every $0 \le m_{\pm} \le \rank C_{\pm}$,
  \end{enumerate}
  where $C_{m_-,m_+}$ are defined as in \Cref{lem:weak-majorization-c-numerical-range}.
\end{theorem}

\begin{proof}
  Recall that $\traceclass$ is the dual $\K^{*}$ of the compact operators via the isometric isomorphism $C \mapsto \trace(C\,\bigcdot)$.
  By the Banach--Alaoglu theorem, bounded subsets of $\traceclass$ which are weak* closed are weak* compact.

  Since the weak* topology on $\traceclass$ is finer than the weak operator topology and coarser than, on trace-norm bounded sets, the (operator) norm topology, by \Cref{prop:wot-closure-convex-orbit-weak-majorization} the set $\set{ X \in \traceclass^{sa} \mid \lambda(X) \submaj \lambda(C) }$ is bounded and weak* closed and therefore weak* compact.
  Because $A \in \K$, the map $X \mapsto \trace(XA)$ is weak* continuous, and therefore $\set{ \trace(XA) \mid X \in \traceclass^{sa}, \lambda(X) \submaj \lambda(C) }$ is compact since it is the continuous image of a compact set.

  Because the weak* topology on $\traceclass$ is weaker than the trace-norm topology, we have
  \begin{equation*}
    \closure[w*]{\closure[\norm{\bigcdot}_1]{\conv \ugroup(C)}} = \closure[\mathrlap{w*}]{\conv \ugroup(C)}.
  \end{equation*}
  Moreover, by \Cref{prop:wot-closure-convex-orbit-weak-majorization,prop:convex-orbit-majorization} the weak* closure of $\set{ X \in \traceclass^{sa} \mid \lambda(X) \maj \lambda(C) }$ is $\set{ X \in \traceclass^{sa} \mid \lambda(X) \submaj \lambda(C) }$ and therefore by \Cref{thm:c-numerical-range-via-majorization} and weak* continuity of $X \mapsto \trace(XA)$, $\ocnr(A)$ is dense in $\set{ \trace(XA) \mid X \in \traceclass^{sa}, \lambda(X) \submaj \lambda(C) }$.
  Hence
  \begin{equation*}
    \closure{\ocnr(A)} = \set{ \trace(XA) \mid X \in \traceclass^{sa}, \lambda(X) \submaj \lambda(C) }.
  \end{equation*}

  \ref{item:ocnr-compact-closed} $\Leftrightarrow$ \ref{item:ocnr-equal-weak-majorization}.
  This is immediate from what we have just proven.

  \ref{item:ocnr-equal-weak-majorization} $\Leftrightarrow$ \ref{item:ocnr-contains-initial-ocnrs}.
  This is immediate from \Cref{thm:c-numerical-range-via-majorization}, \Cref{cor:c-numerical-range-convex} and \Cref{lem:weak-majorization-c-numerical-range}.
\end{proof}

As previously remarked, the equivalence \ref{item:positivity-ocnr} $\Leftrightarrow$ \ref{item:similarity-preserving-ocnr} of \Cref{thm:compact-closed-iff-contains-weak-majorization} generalizes the original result of de Barra, Giles and Sims \cite[Theorem~1(i)]{dBGS-1972-JLMSIS} concerning the standard numerical range, which states that if $A$ is a compact operator, then $0 \in W(A)$ if and only if $W(A)$ is closed.

One might wonder if there is a condition analogous to that of de Barra, Giles and Sims which is somehow tied only to $0$.
The following example shows that for a na\"ive analogue, the result is false, but the corollary after that shows that not all hope is lost.

\begin{example}
  This example shows that, unlike for the case of the standard numerical range, it is insufficient to simply have $0 \in W(A)$ for $A$ compact in order to guarantee that $W(A)$ is closed.
  Indeed, it is even insufficient to have an orthonormal basis $\set{e_n}_{n=1}^{\infty}$ for which $\angles{A e_n,e_n} = 0$ for all $n \in \nats$, even if $A$ is selfadjoint and trace-class.

  Consider $A = \diag(-1,\frac{1}{2},\frac{1}{4},\ldots)$.
  Then $A$ is selfadjoint and trace-class and $\trace(A) = 0$.
  Therefore by \cite[Theorem~1]{FF-1980-PAMS} there is an orthonormal basis with respect to which the diagonal of $A$ is the zero sequence.
  Now let $P$ be a rank-$2$ projection.
  From \Cref{thm:c-numerical-range-selfadjoint-formula} we see that $\ocnr[P](A) = (-1,\frac{3}{4}]$, which clearly contains $0$ and yet is not closed.
\end{example}

Although \Cref{thm:compact-closed-iff-contains-weak-majorization} provides a complete characterization of when the orbit-closed $C$-numerical range of a compact operator is closed, the condition seems nontrivial to check.
The following corollary provides a sufficient condition which is hopefully easier to verify in practice.

\begin{corollary}
  \label{cor:compact-closure-direct-sum-zero}
  Let $C \in \traceclass^{sa}$ be a selfadjoint trace-class operator and let $A \in \K$ be a compact operator.
  Then $\closure{\ocnr(A)} = \ocnr[C \oplus \zop](A \oplus \zop)$, where the $\zop$ acts on a space of dimension at least $\rank C$.
  In particular, if $P$ is a projection of rank at least $\rank C$ for which $PA = AP = 0$, then $\ocnr(A)$ is closed.
\end{corollary}

In the following proof of this corollary, we will be considering operators acting on Hilbert spaces $\Hil_1$ (separable, infinite dimensional) and on $\Hil_1 \oplus \Hil_2$ (with $\Hil_2$ separable).
It will sometimes be convenient to think of these operators acting on the same space which we do by selecting a fixed, but arbitrary isometric isomorphism $\Hil_1 \to \Hil_1 \oplus \Hil_2$.
This induces a *-isomorphism $B(\Hil_1) \to B(\Hil_1 \oplus \Hil_2)$.
Crucially, while the resulting *-isomorphism depends on the specific isometric isomorphism, objects and properties that are invariant under unitary conjugation, such as $\ocnr(A)$, $\orbit(C)$ or approximate unitary equivalence, are independent of this choice.
Moreover, under this identification $\orbit(C) = \orbit(C \oplus \zop)$ because $\lambda(C) = \lambda(C \oplus \zop)$ and the eigenvalue sequence is a complete invariant by \Cref{prop:orbit-closure-equivalences}.
This also makes it possible to read \Cref{cor:compact-closure-direct-sum-zero} as $\closure{\ocnr(A)} = \ocnr(A \oplus \zop)$.

\begin{proof}
  Since $A \in \K$, it is clear that $A$ and $A \oplus \zop$ are approximately unitarily equivalent (via the identification $B(\Hil_1) \to B(\Hil_1 \oplus \Hil_2)$ mentioned prior to the proof).
  Indeed, if $A$ acts on $\Hil_1$ and $A \oplus \zop$ acts on $\Hil_1 \oplus \Hil_2$, consider a sequence of finite projections $P_n$ converging in the strong operator topology to the identity, and notice that $P_n A, A P_n \to A$ in norm since $A \in \K$.
  Let $U_n : \Hil_1 \to \Hil_1 \oplus \Hil_2$ be any unitary for which $U_n P_n = P_n \oplus \zop$ (these exist since each $P_n^{\perp}$ is an infinite projection), and notice that $U_n A U_n^{*} \to A \oplus \zop$.
  Therefore, the closures of $\ocnr(A)$ and $\ocnr[C \oplus \zop](A \oplus \zop)$ coincide by \Cref{cor:approximate-unitary-equivalence-same-closure}.

  To complete the proof, it suffices to prove that $\ocnr(A \oplus \zop)$ is closed.
  For this, let $C_{m,k}$ acting on $\Hil_1$ be defined as in
  \Cref{lem:weak-majorization-c-numerical-range}.
  Then there is a $C'_{m,k}$ acting on $\Hil_2$ such that $C_{m,k} \oplus C'_{m,k} \in \orbit(C \oplus \zop)$, where it suffices by \Cref{prop:orbit-closure-equivalences} to select a selfadjoint operator $C'_{m,k}$ whose nonzero eigenvalues are precisely the terms of $\lambda(C)$ missing from $\lambda(C_{m,k})$.
  Thus, for any $X \in \orbit(C_{m,k})$ we have $X \oplus C'_{m,k} \in \orbit(C \oplus \zop)$, and therefore
  \begin{equation*}
    \trace(XA) = \trace \big( (X \oplus C'_{m,k}) (A \oplus \zop) \big) \in \ocnr[C \oplus \zop](A \oplus \zop).
  \end{equation*}
  Since $X \in \orbit(C_{m,k})$ was arbitrary, as were $m,k$, we find that $\ocnr(A \oplus \zop) \supseteq \ocnr[C_{m,k}](A)$ for all $m,k \in \nats \cup \set{0,\infty}$.
  Therefore, by \Cref{lem:weak-majorization-c-numerical-range} and \Cref{thm:compact-closed-iff-contains-weak-majorization},
  \begin{equation*}
    \ocnr[C \oplus \zop](A \oplus \zop) \supseteq \set{ \trace(XA) \mid X \in \traceclass^{sa}, \lambda(X) \submaj \lambda(C) } = \closure{\ocnr(A)} = \closure{\ocnr[C \oplus \zop](A \oplus \zop)}.
  \end{equation*}

  Now suppose that $A \in \K$ is an operator for which there is a projection of rank at least $\rank C$ for which $PA = AP = 0$.
  If $P^{\perp}$ is finite we may pass from $P$ to an infinite, co-infinite subprojection to ensure the complement is infinite.
  Then if $A'$ denotes the compression of $A$ to $P^{\perp} \Hil$, we certainly have $A = A' \oplus \zop$ where $\zop$ acts on $P \Hil$.
  Moreover, since $P^{\perp}$ is infinite, there is some $C'$ acting on $P^{\perp} \Hil$ such that $\lambda(C') = \lambda(C)$.
  Consequently,
  \begin{equation*}
    \ocnr(A) = \ocnr[C' \oplus \zop](A' \oplus \zop) = \closure{\ocnr[C'](A'})
  \end{equation*}
  is closed.
\end{proof}

\subsection{Bounded operators}
\label{subsec:bounded-operators}

The situation for $A \in \K$ compact was made especially tractable because of the duality $\traceclass \cong \K^{*}$.
As we now turn our attention to arbitrary operators $A \in B(\Hil)$, the analysis becomes significantly more complex.
However, as we will observe, much of the analysis can be restricted to the compact portion of $A$ which lies outside the essential spectrum; for selfadjoint $A$, we mean the operator $(A-mI)_+$ where $m := \max \essspec(A)$.

From now on, we will restrict our attention primarily to $C$ \emph{positive} and trace-class.
The reason is essentially to make the complicated analysis somewhat manageable.
In order to emphasize positivity, we will use the singular value sequence $s(C)$ to refer to the eigenvalue sequence (as opposed to $\lambda(C)$) since these coincide.

Since the orbit-closed $C$-numerical range is convex when $C$ is selfadjoint, one natural way to analyze boundary points is to first rotate the operator and then take the real part, as in the diagram:
\begin{center}
  \begin{tikzcd}
    A \arrow[r, maps to] \arrow[d, "\ocnr(\bigcdot)"] & e^{i\theta}A \arrow[r, maps to] \arrow[d, "\ocnr(\bigcdot)"] & \Re(e^{i\theta}A) \arrow[d, "\ocnr(\bigcdot)"] \\
    \ocnr(A) \arrow[r, maps to]                       & e^{i\theta}\ocnr(A) \arrow[r, maps to]                       & \Re(e^{i\theta}\ocnr(A)),                      
  \end{tikzcd}
\end{center}
where we have used \Cref{prop:c-numerical-range-basics} to commute both $\Re$ and multiplication by $e^{i\theta}$ with $\ocnr(\bigcdot)$.
In so doing one is able essentially to reduce the investigation of points on the boundary of the numerical range to the case when $A$ is selfadjoint.
However, there are often technicalities that arise when there is a line segment on the boundary because, after rotation, there is more than one point on the boundary with maximal real part (see \Cref{fig:rotation-technique}).
This rotation and real part technique goes all the way back to Kippenhahn in \cite{Kip-1951-MN} (or the English translation \cite[\textsection{}3]{Kip-2008-LMA}), but appears elsewhere in the literature, such as \cite{Joh-1978-SJNA}.

\begin{figure}[h]
  \centering
  \subfloat[Supporting line with unique intersection]{%
    \includegraphics[width=0.25\textwidth]{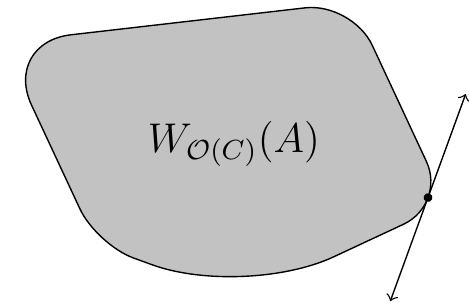}%
    \includegraphics[width=0.25\textwidth]{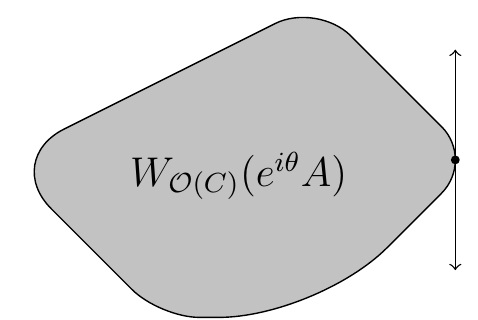}%
  }
  \subfloat[Supporting line with nonunique intersection]{%
    \includegraphics[width=0.25\textwidth]{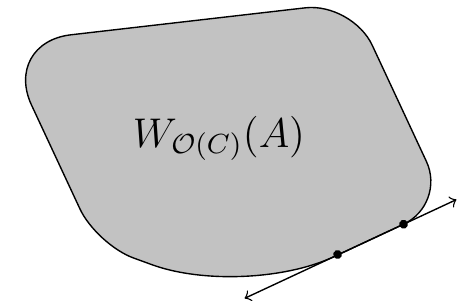}%
    \includegraphics[width=0.18\textwidth]{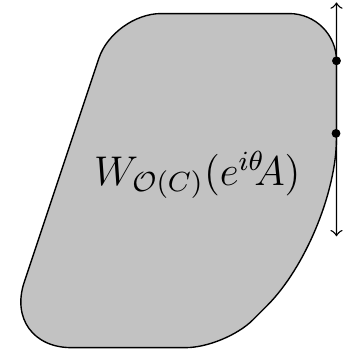}%
  }
  \caption{Rotation technique for points on the boundary of $\ocnr(A)$.}
  \label{fig:rotation-technique}
\end{figure}

We begin with a simple but rather important lemma concerning submajorization which will be essential in our analysis of the boundary.
Effectively, it says that if $(d_n) \submaj (c_n)$ then $(d_n a_n) \submaj (c_n a_n)$ for any nonnegative decreasing sequence $(a_n)$;
moreover, if $(d_n a_n) \maj (c_n a_n)$ and $(a_n) \in \czstar$ is strictly positive, then $(d_n)$ is block majorized by $(c_n)$.
The first part of the lemma, namely that $(d_n) \submaj (c_n)$ implies $(d_n a_n) \submaj (c_n a_n)$, is a known result concerning submajorization (see, for example \cite[5.A.4.d]{MOA-2011}), but to the authors' knowledge the remainder of the lemma has not appeared in the literature and we will make full use of these additional facts later on.
To see the connection between the above formulation in terms of majorization and the actual statement of the lemma, consider $\delta_n := c_n - d_n$.

\begin{lemma}
  \label{lem:majorization-product-decreasing-sequence}
  Suppose that $(\delta_n)$ is a real-valued sequence and $(a_n)$ is a nonnegative decreasing sequence (even finite sequences are considered).
  If for every $N$, $\sum_{n=1}^N \delta_n \ge 0$, then
  \begin{enumerate}
  \item \label{item:weak-majorization-implies-c-weak-majorization} for every $N$, $\displaystyle\sum_{n=1}^N \delta_n a_n \ge 0$;
  \item \label{item:c-majorization-and-jump-implies-majorization} if $\displaystyle\sum_{n=1}^N \delta_n a_n = 0$, then $\displaystyle\sum_{n=1}^L \delta_n = 0$ whenever $a_L > a_{L+1}$ with $L < N$, \\ and if, in addition, $a_N > 0$, then $\displaystyle\sum_{n=1}^N \delta_n = 0$;
  \item \label{item:liminf-c-majorization-implies-block-majorization} if $\displaystyle\liminf_{N \to \infty} \sum_{n=1}^N \delta_n a_n = 0$, then $\displaystyle\sum_{n=1}^L \delta_n = 0$ whenever $a_L > a_{L+1}$.
  \end{enumerate}
\end{lemma}

\begin{proof}
  The proof is a simple application of summation by parts.
  Indeed, for any $N$,
  \begin{equation*}
    \sum_{n=1}^N \delta_n a_n = a_N \sum_{n=1}^N \delta_n + \sum_{L=1}^{N-1} (a_L - a_{L+1}) \sum_{n=1}^L \delta_n.
  \end{equation*}
  Notice that by hypothesis $a_N, \sum_{n=1}^N \delta_n$ are each nonnegative as are $(a_L - a_{L+1})$ and $\sum_{n=1}^L \delta_n$ for each $L < N$.
  Therefore, $\sum_{n=1}^N \delta_n a_n \ge 0$ as well, proving \ref{item:weak-majorization-implies-c-weak-majorization}.

  Moreover, if $\sum_{n=1}^N \delta_n a_n = 0$ and for $L < N$ if $a_L > a_{L+1}$, then $\sum_{n=1}^L \delta_n = 0$.
  In addition, if $a_N > 0$, then $\sum_{n=1}^N \delta_n = 0$, which establishes \ref{item:c-majorization-and-jump-implies-majorization}.

  Finally, notice that
  \begin{align*}
    \liminf_{N \to \infty} \sum_{n=1}^N \delta_n a_n &\ge \liminf_{N \to \infty} a_N \sum_{n=1}^N \delta_n + \liminf_{N \to \infty} \sum_{L=1}^{N-1} (a_L - a_{L+1}) \sum_{n=1}^L \delta_n \\
                                                     &\ge \liminf_{N \to \infty} \sum_{L=1}^{N-1} (a_L - a_{L+1}) \sum_{n=1}^L \delta_n.
  \end{align*}
  Therefore, if the limit inferior on the left-hand side is zero, then we conclude $\sum_{n=1}^L \delta_n = 0$ whenever $a_L > a_{L+1}$.
\end{proof}

The next proposition guarantees that the supremum of the orbit-closed $C$-numerical range is attained whenever $A,C$ are positive and compact.
This is our first sufficient condition for a point on the boundary to be included in $\ocnr(A)$.

\begin{proposition}
  \label{prop:c-numerical-range-maximum}
  Let $C,A$ be positive compact operators with $C$ trace-class.
  Then
  \begin{equation*}
    \sup \ocnr(A) = \sum_{n=1}^{\infty} s_n(C) s_n(A),
  \end{equation*}
  and moreover the supremum is attained.
\end{proposition}

\begin{proof}
  Take any $X \in \ugroup(C)$.
  Then since $X$ is a positive compact operator it is diagonalizable and so in some basis $X = \diag s(C)$.
  Let $(d_n)$ be the diagonal of $A$ in this basis, which is necessarily nonnegative since $A$ is a positive operator.
  It is well-known that $(d_n) \submaj s(A)$ (e.g., see \cite[Theorem~4.2]{AK-2006-OTOAaA}).
  Therefore, since $s(C)$ is a nonincreasing nonnegative sequence, we may apply \autoref{lem:majorization-product-decreasing-sequence} to conclude for all $N \in \nats$,
  \begin{equation*}
    \sum_{n=1}^N d_n s_n(C) \le \sum_{n=1}^N s_n(A) s_n(C) \le \sum_{n=1}^{\infty} s_n(A) s_n(C).
  \end{equation*}
  Taking the limit as $N \to \infty$, we find
  \begin{equation*}
    \trace(XA) = \sum_{n=1}^{\infty} d_n s_n(C) \le \sum_{n=1}^{\infty} s_n(A) s_n(C).
  \end{equation*}
  Moreover, since the trace is trace-norm continuous, we have $\trace(XA) \le \sum_{n=1}^{\infty} s_n(A) s_n(C)$ for any $X \in \orbit(C)$.
  Thus $\sup \ocnr(A) \le \sum_{n=1}^{\infty} s_n(A) s_n(C)$.

  To show equality and thus that the supremum is attained, simply note that there is a (likely different) basis which diagonalizes $A$ since it is a positive compact operator, and in this basis $A = \zop_{\ker A} \oplus \diag s(A)$ (or $A = \diag s(A)$ if $A$ has finite rank).
  Then $X := \zop_{\ker A} \oplus \diag s(C) \in \orbit(C)$ (or $X := \diag s(C) \in \orbit(C)$ if $A$ has finite rank) by \Cref{prop:orbit-closure-equivalences} and we obtain
  \begin{equation*}
    \trace(XA) = \sum_{n=1}^{\infty} s_n(A) s_n(C). \qedhere
  \end{equation*}
\end{proof}

Although the statement of the following theorem is restricted to the selfadjoint case, by the standard rotation argument mentioned at the beginning of this section, the next theorem provides a necessary and sufficient condition for a supporting line\footnotemark{} of $\ocnr(A)$ to contain \emph{at least one} point of $\ocnr(A)$.
Notice that if this supporting line intersects $\closure{\ocnr(A)}$ in exactly one point (in particular, if this point does not lie on a line segment on the boundary), then this theorem gives a necessary and sufficient condition for that point to lie in $\ocnr(A)$ (cf. \Cref{fig:rotation-technique}).

\footnotetext{%
  recall that a \term{supporting line} $L$ for a convex set $C$ in the plane is a line such that $L \cap \closure{C} \not= \emptyset$ and $C$ is entirely contained within one of the closed half-planes determined by $L$.
  Notice that this latter condition ensures $L \cap \closure{C} \subseteq \partial C$.
}

We remark for the reader's convenience a basic fact which will occur in the following theorem and repeatedly throughout the remainder of this paper.
If $A \in B(\Hil)$ and $m := \max \essspec(\Re A)$, then $(\Re A - mI)_+$ is a positive compact operator.
Indeed, it clearly suffices to assume $m = 0$, and then simply notice that the spectral projections $\chi_{(-\infty,-\epsilon)}(\Re A)_+ = 0$ and $\chi_{(\epsilon,\infty)}(\Re A)_+ = \chi_{(\epsilon,\infty)}(\Re A)$ are all finite for every $\epsilon > 0$.

\begin{theorem}
  \label{thm:c-numerical-range-selfadjoint-formula}
  Let $C \in \traceclass^+$ be a positive trace-class operator and suppose $A \in B(\Hil)$ is selfadjoint.
  Let $m := \max \essspec(A)$.
  Then
  \begin{equation*}
    \sup \ocnr(A) = m\trace C + \sup \ocnr(A-mI)_+,
  \end{equation*}
  Moreover, if $P := \chi_{[m,\infty)}(A)$ denotes the spectral projection of $A$ onto the interval $[m,\infty)$, then $\sup \ocnr(A)$ is attained if and only if $\rank C \le \trace P$.
  In fact, if $X \in \orbit(C)$ attains the supremum, then $XP = PX = X$.
\end{theorem}

\begin{proof}
  For $A \in B(\Hil)$ and $m \in \complex$, since $\essspec(A-mI) = \essspec(A) - m$, by \Cref{prop:c-numerical-range-basics}\ref{item:similarity-preserving-ocnr}, we may assume without loss of generality that $m=0$.

  The inequality $\sup \ocnr(A) \le \sup \ocnr(A_+)$ is immediate because for any $X \in \orbit(C)$, since $A_+ - A \ge 0$, we have $\trace(X(A_+ - A)) \ge 0$ by \Cref{prop:c-numerical-range-basics}\ref{item:positivity-ocnr}.
  Therefore
  \begin{equation*}
    \trace(XA) \le \trace(XA_+) \le \sup \ocnr(A_+),
  \end{equation*}
  and taking the supremum over $X \in \orbit(C)$ yields $\sup \ocnr(A) \le \sup \ocnr(A_+)$.

  It remains to prove the reverse inequality and the claim concerning when the supremum is attained.
  We begin by proving the former.

  Now, if $\rank C \le \trace P$, there is some $C' \in \orbit(C)$ such that $PC' = C'P = C'$ by \Cref{prop:orbit-closure-equivalences} (e.g., take $C' := \zop_{P^{\perp}\Hil} \oplus \diag_{P\Hil} \big(s_n(C)\big)_{n=1}^{\trace P}$).
  Since $C',A_+$ are positive compact operators which are zero on $P^{\perp} \Hil$, we may view them as operators acting on $P\Hil$.
  By \Cref{prop:c-numerical-range-maximum}, there is some $X' \in \orbit_{P\Hil}(C')$ for which
  \begin{equation*}
    \trace_{P\Hil}(X' A_+) = \sup \cnr[\orbit_{P\Hil}(C')](A_+) = \sum_{n=1}^{\infty} s_n(C') s_n(A_+) = \sum_{n=1}^{\infty} s_n(C) s_n(A_+) = \sup \ocnr(A_+).
  \end{equation*}
  Then setting $X := \zop_{P^{\perp}\Hil} \oplus X' \in \orbit(C)$ we find that $\trace(XA) = \trace_{P\Hil}(X'A_+) = \sup \ocnr(A_+)$ which we already established is at least $\sup \ocnr(A)$.
  Moreover, notice that $\sup \ocnr(A)$ is attained in this case.

  Now suppose $\rank C > \trace P$.
  For $\epsilon > 0$, let $P_{\epsilon} := \chi_{(-\epsilon,0)}(A)$.
  Since $\rank C > \trace P$, we know that $P$ is a finite projection.
  But since $0 \in \essspec(A)$, we must have that $P_{\epsilon} + P = \chi_{(-\epsilon,\infty)}(A)$ is infinite for every $\epsilon > 0$, and hence $P_{\epsilon}$ is infinite.
  Then consider a basis $\mathfrak{e} = \set{e_n}_{n \in \ints}$ such that for $1 \le n \le \trace P$, $A e_n = s_n(A_+) e_n$, and for which $\set{e_n}_{n=\trace P + 1}^{\infty}$ is an orthonormal set in $P_{\epsilon}\Hil$.
  Define $X \in \orbit(C)$ to be the diagonal operator $X e_n = s_n(C) e_n$ for $n \in \nats$ and $X e_n = 0$ for $n < 0$.
  Then by construction and since $s_n(A_+) = 0$ for $n > \trace P$, we find
  \begin{align*}
    \trace(XA) = \trace(XAP) + \trace(XAP_{\epsilon}) &\ge \sum_{n=1}^{\trace P} s_n(C) s_n(A_+) - \norm{X}_1 \norm{AP_{\epsilon}} \\
                                                      &\ge \sum_{n=1}^{\infty} s_n(C) s_n(A_+) - \epsilon \trace C. \\
                                                      &= \ocnr(A_+) - \epsilon \trace C.
  \end{align*}
  Since $\epsilon$ was arbitrary, this proves $\sup \ocnr(A) \ge \sup \ocnr(A_+)$, and thus we have equality.

  Suppose $X \in \orbit(C)$ attains the supremum, that is, $\trace(XA) = \sup \ocnr(A)$.
  As we have just proved that $\ocnr(A) = \ocnr(A_+)$, so then $\trace(XA) = \sup \ocnr(A_+)$.
  Moreover, as $PA = A_+$ and $P^{\perp} A = -A_-$,
  \begin{equation}
    \label{eq:2}
    \begin{aligned}
      \trace(XA) &= \trace(XPA) + \trace(XP^{\perp}A) \\
      &= \trace(XA_+) - \trace(XP^{\perp}A_-) \\
      &\le \left( \sup \ocnr(A_+) \right) - \trace(XP^{\perp}A_-) \\
      &= \trace(XA) - \trace(XA_-).
    \end{aligned}
  \end{equation}
  Since, $\trace(XA_-) = \trace(XP^{\perp}A_-) = \trace(X^{\frac{1}{2}} P^{\perp} A_- P^{\perp} X^{\frac{1}{2}}) \ge 0$, equality in \eqref{eq:2} holds if and only if $X^{\frac{1}{2}} P^{\perp} A_- P^{\perp} X^{\frac{1}{2}} = 0$ since the trace is faithful, if and only if $A^{\frac{1}{2}}_- P^{\perp} X^{\frac{1}{2}} = 0$.

  Now because $P^{\perp}$ is the spectral projection of $A$ on the interval $(-\infty,0)$, we see that $A_-^{\frac{1}{2}}$ is strictly positive on $P^{\perp}\Hil$ (or $P^{\perp} = 0$).
  Therefore, $A^{\frac{1}{2}}_- P^{\perp} X^{\frac{1}{2}} = 0$ if and only if $P^{\perp} X^{\frac{1}{2}} = 0$ if and only if $R_X = R_{X^{\smash[t]{\frac{1}{2}}}} \le P$ ($R_X$ denotes the range projection of $X$) if and only if $PX=XP=X$.
  This proves the claim about $X \in \orbit(C)$ which attain the supremum.

  Finally, if $\rank C > \trace P$, then for any $X \in \orbit(C)$, $XP \not= X$ and so by the above, $X$ does not attain the supremum.
  Since $X$ was arbitrary, the supremum cannot be attained in this case.
\end{proof}

The following example shows how the techniques developed thus far can be used to compute the orbit-closed $C$-numerical range in certain circumstances.

\begin{example}
  Let $S$ denote the shift operator on either $\ell^2(\nats)$ or $\ell^2(\ints)$.
  It is well known that the standard numerical range is $W(S) = \mathbb{D}$, the open unit disk.
  Let $C \in \traceclass^+$ with $\trace C = \norm{C}_1 = 1$.
  We will show $\ocnr(S) = \mathbb{D}$ also.

  Notice first that $\lambda(C) \maj \lambda(P)$ where $P$ is a rank-$1$ projection, so by \Cref{cor:majorization-c-numerical-range-inclusion}, $\ocnr(S) \subseteq \ocnr[P](S) = W(S) = \mathbb{D}$.
  Moreover, $S$ is unitarily equivalent to $e^{i\theta} S$ via the diagonal unitary $U_{\theta} := \diag(e^{in\theta})$.
  Therefore, since the orbit-closed $C$-numerical range is unitarily invariant and using \Cref{prop:c-numerical-range-basics}\ref{item:similarity-preserving-ocnr},  $\ocnr(S) = \ocnr(e^{i\theta}S) = e^{i\theta}\ocnr(S)$ and so $\ocnr(S)$ is radially symmetric.
  Because $\max \spec(\Re S) = \max \essspec(\Re S) = 1$, we know $(\Re S - I)_+ = 0$, and therefore by \Cref{prop:c-numerical-range-basics}\ref{item:hermitian-ocnr} and \Cref{thm:c-numerical-range-selfadjoint-formula},
  \begin{equation*}
    \sup \Re \ocnr(S) = \sup \ocnr(\Re S) = \trace C + \sup \ocnr(\Re S - I)_+ = \trace C = 1.
  \end{equation*}
  Consequently, by the radial symmetry and convexity (using \Cref{cor:c-numerical-range-convex}) of $\ocnr(S)$ it must contain the open unit disk.
  Therefore $\ocnr(S) = \mathbb{D}$.

  We note that $\cnr(S)$ must be dense in $\mathbb{D}$ by \Cref{thm:cnr-dense-in-ocnr}, but it seems rather hard to conclude these sets are equal without convexity.
\end{example}

We now build towards \Cref{thm:ocnr-closed-rank-condition} which provides a sufficient condition for $\ocnr(A)$ to be closed for $A \in B(\Hil)$.
We begin with a bootstrapping of a standard result by induction.

\begin{lemma}
  \label{lem:close-projections-unitary-close-to-1}
  Given $\epsilon > 0$ there is some $\delta > 0$ such that whenever $\set{P_j}_{j=1}^N, \set{Q_j}_{j=1}^N$ are each collections of mutually orthogonal projections with $\norm{P_j - Q_j} < \delta$ for each $1 \le j \le N$, then there is a unitary $U$ conjugating each pair $P_j,Q_j$ such that $\norm{U-I} < \epsilon$.
\end{lemma}

\begin{proof}
  We proceed by induction on $N$.
  The case when $N=1$ is standard, but a good reference is \cite[II.3.3.4]{Bla-2006}.
  The argument is essentially this: set $Z := P_1 Q_1 + (1-P_1)(1-Q_1)$, then $Z$ is invertible and $U = Z\abs{Z}^{-1}$ is the desired unitary.

  Now let $N \in \nats$ and suppose the result holds for pairs of collections of mutually orthogonal projections of length at most $N$.
  Let $\epsilon > 0$, then there is some $\delta > 0$ corresponding to $\frac{\epsilon}{2}$ by the inductive hypothesis.
  Moreover, there is some $\eta > 0$ corresponding to $\min\set{\frac{\delta}{3},\frac{\epsilon}{2}}$.
  Suppose that $\set{P_j}_{j=1}^{N+1}, \set{Q_j}_{j=1}^{N+1}$ are each collections of mutually orthogonal projections with $\norm{P_j - Q_j} < \min\set{\eta,\frac{\delta}{3}}$.

  Then there is a unitary $U'$ with $\norm{U' - I} < \min\set{\frac{\delta}{3},\frac{\epsilon}{2}}$ conjugating $P_{N+1}$ to $Q_{N+1}$.
  Then $U'$ also conjugates $\set{P_j}_{j=1}^N$ to a mutually orthogonal collection $\set{P'_j}_{j=1}^N$.
  Moreover,
  \begin{equation*}
    \norm{P'_j - Q_j} = \norm{U' P_j U'^{*} - Q_j} \le \norm{U'-I} + \norm{U'^{*} - I} + \norm{P_j - Q_j} \le 2\min\vset{\frac{\delta}{3},\frac{\epsilon}{2}} + \norm{P_j - Q_j} < \delta.
  \end{equation*}

  Then let $V$ be a unitary conjugating $\set{P'_j}_{j=1}^N$ to $\set{Q_j}_{j=1}^N$ inside the Hilbert space $Q_{N+1}^{\perp} \Hil$ such that $\norm{V - I_{Q_{N+1}^{\perp} \Hil}} < \frac{\epsilon}{2}$.
  Then $W := I_{Q_{N+1}\Hil} \oplus V$ is a unitary on $\Hil$ and $\norm{W - I} < \frac{\epsilon}{2}$.
  Finally, set $U = WU'$ and notice that $U$ conjugates $\set{P_j}_{j=1}^{N+1}$ to $\set{Q_j}_{j=1}^{N+1}$.
  Moreover,
  \begin{equation*}
    \norm{U-I} \le \norm{W-I} + \norm{U'-I} < \frac{\epsilon}{2} + \min\vset{\frac{\delta}{3},\frac{\epsilon}{2}} \le \epsilon.
  \end{equation*}
  By the induction, the proof is complete.
\end{proof}

Using \Cref{lem:close-projections-unitary-close-to-1} we now establish a sufficient condition for when certain points on the boundary $\partial\ocnr(A)$ can be obtained by elements of $\orbit(C)$ which are close in trace norm.
This approximation result is a key step in the proof of \Cref{thm:ocnr-closed-rank-condition}.

\begin{proposition}
  \label{prop:boundary-approximation}
  Let $C \in \traceclass^+$ and suppose that $\rank (\Re A-mI)_+ \ge \rank C$, where $m := \max \essspec(\Re A)$.
  Let $[x_-,x_+]$ denote the (possibly degenerate) line segment on $\partial\ocnr(A)$ consisting of the points with maximal real part.

  Furthermore, suppose that there are arbitrarily small $\theta > 0$ for which there is a point $x_{\theta} \in \ocnr(A)$ on its boundary whose supporting line intersects $\ocnr(A)$ only at this point $x_{\theta}$, and that $x_{\theta} \to x_-$ as $\theta \to 0$.
  Then given any $\epsilon > 0$, for sufficiently small $\theta$ there are some $X_{\theta}, X \in \orbit(C)$ with $\trace(X_{\theta}A) = x_{\theta}$ and $\Re\trace(XA) = \sup \Re\ocnr(A)$, and $\norm{X_{\theta} - X}_1 < \epsilon$.

  Consequently, $\ocnr(A)$ contains points on the line segment arbitrarily close to $x_-$.
\end{proposition}

\begin{proof}
  By translating, we may clearly suppose that $m = 0$.
  Define for each $\theta \in \reals$ the selfadjoint operator $A_{\theta} := \Re(e^{i\theta} A)$.

  Let $\epsilon > 0$.
  Since $C \in \traceclass^+$, there is some $N \in \nats$ such that $\sum_{n=N+1}^{\infty} s_n(C) < \frac{\epsilon}{8\norm{A}}$;
  if $\rank C < \infty$, set $N := \rank C$.
  Let $\lambda_1 > \cdots > \lambda_m > 0$ be the $m$ largest eigenvalues of $A_0 = \Re A$ with associated (mutually orthogonal) spectral projections $P_j := \chi_{\set{\lambda_j}}(A_0)$ for $1 \le j \le m$.
  Choose $m$ so that $\sum_{j=1}^m \trace(P_j) \ge N$, which is possible since $\rank (A_0)_+ \ge \rank C$ by hypothesis.
  Set $P := \sum_{j=1}^m P_j$ and define $n_j := \sum_{i=1}^j \trace P_i$ and $n_0 := 0$.
  We remark for future reference that $s_k((A_0)_+) = \lambda_j$ when $n_{j-1} < k \le n_j$.

  Notice that
  \begin{equation*}
    A_{\theta} = \Re(e^{i\theta} A) = (\cos\theta) \Re A + (\sin\theta) \Im A = A_0 + B_{\theta},
  \end{equation*}
  where $B_{\theta} := (\cos\theta - 1) \Re A + (\sin\theta) \Im A$ and that $\norm{B_{\theta}} \le 2\theta \norm{A}$.
  Set
  \begin{equation*}
    \delta_1 = \frac{1}{4} \min_{1 \le j \le m} \dist(\lambda_j, \spec(A_0) \setminus \set{\lambda_j}).
  \end{equation*}
  By the upper semicontinuity of the spectrum (and the essential spectrum), for all sufficiently small $\theta > 0$ we can guarantee that $m_{\theta} := \max \essspec(A_{\theta}) < \lambda_m - \delta_1$ and that $\spec(A_{\theta})$ is contained in the $\delta_1$-neighborhood of $\spec(A_0)$.
  
  By \Cref{lem:close-projections-unitary-close-to-1}, there is some $\delta > 0$ associated to $\frac{\epsilon}{8 \norm{A} \trace C}$.
  Then we may choose $\theta > 0$ small enough so that both $\abs{x_{\theta} - x_-} < \frac{\epsilon}{2}$ and $\norm{B_{\theta}}$ is small enough \cite[Theorem~3.4]{MS-2015-JRAM}\footnote{This result is actually much stronger than we need because it provides tight bounds on the required size of the norm $\norm{B_{\theta}}$. For our purposes, the result we need could be obtained by straightforward, albeit somewhat tedious, arguments using the continuous functional calculus.} that if $Q_j := \chi_{[\lambda_j-\delta_1,\lambda_j+\delta_1]}(A_{\theta})$, then $\norm{P_j - Q_j} < \delta$.
  Moreover, let $Q := \sum_{j=1}^m Q_j$.
  By \Cref{lem:close-projections-unitary-close-to-1} there is a unitary $U$ with $\norm{U-I} < \frac{\epsilon}{8 \norm{A} \trace C}$ conjugating $Q_j$ to $P_j$ (i.e., $U Q_j U^{*} = P_j$) for each $1 \le j \le m$.

  Let $\mathfrak{e} := \set{e_k}_{k \in \ints}$ be an orthonormal basis so that for $1 \le k \le \max \set{\rank C, \trace Q}$ (note: $\trace Q = \trace P$), $e_k$ is an eigenvector of $A_{\theta}$ for the eigenvalue $m_{\theta} + s_k((A_{\theta} - m_{\theta}I)_+)$;
  this is possible since by hypothesis $\rank (A_{\theta} - m_{\theta}I)_+ \ge \rank C$, and also $\chi_{(m_{\theta},\infty)}(A_{\theta}) \ge Q$ so $\rank (A_{\theta} - m_{\theta}I)_+ \ge \trace Q$.
  The eigenvectors $\set{e_k}_{k=1}^{\trace Q}$ are in the subspaces $Q_j \Hil$.
  More specifically, $\set{e_k}_{k=n_{j-1} + 1}^{n_j}$ is a basis for $Q_j \Hil$.
  Consequently, $\set{U e_k}_{k=n_{j-1} + 1}^{n_j}$ is a basis for $P_j \Hil$ since $U$ conjugates $Q_j$ to $P_j$.
  Therefore, for $n_{j-1} < k \le n_j$, we have $A_0 U e_k = \lambda_j U e_k$.
  So, these are eigenvectors for $A_0$.

  Now let $\mathfrak{f} := \set{f_k}_{k \in \ints}$ be an orthonormal basis for which $A_0 f_k = s_k((A_0)_+) f_k$ when $1 \le k \le \max \set{\rank C, \trace P}$ (note: $\trace P = \trace Q$);
  again, this is possible since $\rank (A_0)_+ \ge \rank C$, and because $\chi_{(0,\infty)}(A_0) \ge P$, so $\rank(A_0)_+ \ge \trace P$.
  By the previous paragraph we may select $f_k = U e_k$ for $1 \le k \le \trace Q = \trace P$.
  Let $V$ be the unitary which maps $U e_k$ to $f_k$ for all $k \in \ints$.
  Notice that $PV = VP = P$ since $P\Hil = \spans \set{f_k}_{k=1}^{\trace P}$ and $V$ acts as the identity here since $Ue_k = f_k$ for $1 \le k \le \trace P$.

  Define $X_{\theta}$ to be the operator which is diagonal with respect to the basis $\mathfrak{e}$ such that $X_{\theta} e_k = s_k(C) e_k$ for $1 \le k \le \rank C$ and $X_{\theta} e_k = 0$ for all other values of $k$.
  Clearly $X_{\theta} \in \orbit(C)$ by \Cref{prop:orbit-closure-equivalences}.
  Moreover, notice that
  \begin{align*}
    \trace(X_{\theta} A_{\theta}) &= \sum_{k=1}^{\rank C} s_k(C) \big( m_{\theta}  + s_k((A_{\theta}-m_{\theta}I)_+) \big) \\
                                  &= m_{\theta} \trace C + \sum_{k=1}^{\infty} s_k(C) s_k((A_{\theta}-m_{\theta}I)_+) \\
                                  &= m_{\theta} \trace C + \sup \ocnr \big( (A_{\theta} - m_{\theta} I)_+ \big) \\
                                  &= \sup \ocnr(A_{\theta}) = \sup \Re(\ocnr(e^{i\theta}A)).
  \end{align*}
  Then since $\Re \trace(X_{\theta} e^{i\theta} A) = \trace(X_{\theta} A_{\theta})$ maximizes $\Re(\ocnr(e^{i\theta}A))$, and because the supporting line for $x_{\theta}$ intersects the boundary only at that point, we must have $\trace(X_{\theta} A) = x_{\theta}$.

  Now define $X := (VU)X_{\theta}(VU)^{*} \in \orbit(C)$.
  Since $VU$ maps the basis $\mathfrak{e}$ onto the basis $\mathfrak{f}$, we see that $X$ is diagonal with respect to the basis $\mathfrak{f}$.
  Moreover,
  \begin{equation*}
    \trace(X A_0) = \sum_{k=1}^{\rank C} s_k(C) s_k((A_0)_+) = \sum_{k=1}^{\infty} s_k(C) s_k((A_0)_+) = \sup \ocnr(A_0) = \sup \Re(\ocnr(A)).
  \end{equation*}
  Since $\Re \trace(XA) = \trace(XA_0)$, this entails $\trace(XA) \in [x_-,x_+]$.

  We now estimate the trace norm of $X - X_{\theta}$.
  Since $P,X$ (or $Q,X_{\theta}$) are diagonal with respect to the basis $\mathfrak{f}$ (or $\mathfrak{e}$) and therefore commute, we have
  \begin{equation*}
    X - X_{\theta} = PXP - QX_{\theta}Q + P^{\perp}XP^{\perp} - Q^{\perp}X_{\theta}Q^{\perp}.
  \end{equation*}
  Additionally, since $U$ conjugates $Q$ to $P$, we know $UQU^{*} = P$ and so $QU^{*} = U^{*}P$ and $UQ = PU$.
  In addition, $PV = VP = P$, and combining these we obtain
  \begin{align*}
    PXP - QX_{\theta}Q &= PXP - QU^{*}V^{*}XVUQ \\
                       &= PXP - U^{*}PV^{*} XVPU \\
                       &= PXP - U^{*}PXPU \\
                       &= PX - U^{*}PX + U^{*}PX - U^{*}PXPU.
  \end{align*}
  Combining these we find
  \begin{align*}
    \norm{X-X_{\theta}}_1 &\le \norm{PXP - QX_{\theta}Q + P^{\perp}XP^{\perp} - Q^{\perp}X_{\theta}Q^{\perp}}_1 \\
                          &\le \norm{PX - U^{*}PX}_1 + \norm{U^{*}PX - U^{*}PXPU}_1 + \norm{P^{\perp}XP^{\perp}}_1 + \norm{Q^{\perp}X_{\theta}Q^{\perp}}_1 \\
                          &\le \norm{P-U^{*}P} \norm{X}_1 + \norm{U^{*}P} \norm{X}_1 \norm{P - PU} + \trace(P^{\perp}XP^{\perp}) + \trace(Q^{\perp}X_{\theta}Q^{\perp}) \\
                          &\le \frac{\epsilon}{8 \norm{A} \trace C} \trace C + \frac{\epsilon}{8 \norm{A} \trace C} \trace C + \sum_{n=1+\trace P}^{\infty} s_n(C) + \sum_{n=1+\trace Q}^{\infty} s_n(C) \\
                          &\le \frac{\epsilon}{4\norm{A}} + 2 \sum_{n=N+1}^{\infty} s_n(C) \\
                          &< \frac{\epsilon}{2\norm{A}}.
  \end{align*}
  Therefore,
  \begin{equation*}
    \abs{\trace(XA) - x_-} \le \abs{\trace(XA) - \trace(X_{\theta}A)} + \abs{x_{\theta} - x_-} < \norm{X-X_{\theta}}_1 \norm{A} + \frac{\epsilon}{2} < \epsilon.
  \end{equation*}
  Since $\trace(XA) \in [x_-,x_+]$ and $\epsilon > 0$ was arbitrary, we conclude that $\ocnr(A)$ contains points on $[x_-,x_+]$ which are arbitrarily close to $x_-$.
\end{proof}

We are almost ready to provide a sufficient condition for $\ocnr(A)$ to be closed when $C \in \traceclass^+$ and $A \in B(\Hil)$, but before we proceed we need two more technical results concerning majorization, spectral projections and operators which maximize $\ocnr(A)$ for $A$ selfadjoint.
\Cref{lem:extreme-points-k-numerical-range} concerns, in essence, the properties of projections which maximize the $k$-numerical range.
Then \Cref{prop:block-diagonal-decomposition} bootstraps \Cref{lem:extreme-points-k-numerical-range} to conclude that a maximizer of the orbit-closed $C$-numerical range has a certain block diagonal decomposition.

\begin{lemma}
  \label{lem:extreme-points-k-numerical-range}
  Let $X$ be a positive compact operator and $P$ a rank-$N$ projection.
  If 
  \begin{equation*}
    \trace (P X) = \sup \ocnr[P](X) = \sum_{n=1}^N s_n(X),
  \end{equation*}
  and $Q := \chi_{[s_N(X),\infty)}(X)$, then $P \le Q$ and $Q - P \le \chi_{\set{s_N(X)}}(X)$.
  Consequently, $X$ commutes with $P$.
\end{lemma}

\begin{proof}
  In the case when $s_N(X) = 0$, then $Q = I \ge P$ and $Q - P = I - P = P^{\perp}$.
  Therefore $\trace(P^{\perp} X P^{\perp}) = \trace(P^{\perp} X) = \trace X - \trace(P X) = 0$.
  Since the trace is faithful, this implies $P^{\perp} X P^{\perp} = 0$, and therefore that $P^{\perp} X^{\frac{1}{2}} = 0$.
  Therefore $P^{\perp} \le \chi_{\set{0}}(X)$.

  Therefore we may suppose $s_N(X) > 0$.
  Let $\lambda_1 > \lambda_2 > \cdots > \lambda_m = s_N(X) > 0$ be the distinct eigenvalues of $X$ greater than or equal to $s_N(X)$, and let $Q_j := \chi_{\set{\lambda_j}}(X)$ for $1 \le j \le m$ be the associated spectral projections.
  Then $Q = \sum_{j=1}^m Q_j$, and $X Q_j = \lambda_j Q_j$.
  Let $\lambda'$ be the largest eigenvalue of $X$ less than $\lambda_m$.
  Set $\lambda_{m+1} := \frac{1}{2} ( \lambda_m + \lambda' )$, so that $\lambda_m > \lambda_{m+1} > \lambda' \ge 0$.
  For convenience of notation we set $Q_{m+1} := Q^{\perp}$.
  We remark that $X Q_{m+1} \le \lambda_{m+1} Q_{m+1}$.

  Notice that for $1 \le j \le m+1$, $\trace (P Q_j) = \trace (Q_j P Q_j) \le \trace Q_j$ and that
  \begin{equation*}
    \sum_{j=1}^{m+1} \trace (P Q_j) = \trace \big( P (Q_1 + \cdots + Q_{m+1}) \big) = \trace \big( P (Q+Q^{\perp}) \big) = \trace P = N.
  \end{equation*}
  Therefore, we have majorization of the finite sequences
  \begin{equation}
    \label{eq:finite-majorization-pq}
    \big( \trace(P Q_1), \ldots, \trace(P Q_{m+1}) \big) \maj \bigg( \trace Q_1, \ldots, \trace Q_{m-1}, \trace P - \sum_{j=1}^{m-1} \trace Q_j, 0 \bigg).
  \end{equation}
  Consider the difference of these sequences which, since $\trace Q_j - \trace (PQ_j) = \trace((I-P)Q_j) = \trace(P^{\perp}Q_j)$, has the form
  \begin{equation*}
    (\delta_j)_{j=1}^{m+1} := \bigg( \trace (P^{\perp}Q_1), \ldots, \trace (P^{\perp}Q_{m-1}), \trace P - \sum_{j=1}^m \trace( P Q_j ), - \trace (PQ^{\perp})  \bigg).
  \end{equation*}

  Then because $\sum_{j=1}^{m+1} Q_j = I$ and $XQ_j \le \lambda_j Q_j$ for $1 \le j \le m+1$, and since $(\delta_j)_{j=1}^{m+1}$ has nonnegative partial sums,
  \begin{align*}
    \trace (PX) = \sum_{j=1}^{m+1} \trace (PXQ_j) &\le \sum_{j=1}^{m+1} \lambda_j \trace (PQ_j) \\
                                                  &\le \sum_{j=1}^{m-1} \lambda_j \trace Q_j + \lambda_m \bigg( \trace P - \sum_{j=1}^{m-1} \trace Q_j \bigg)  \tag*{by \Cref{lem:majorization-product-decreasing-sequence} with \eqref{eq:finite-majorization-pq},}\\
                                                  &= \sum_{j=1}^M s_j(X) + \lambda_m(X) (N - M) \\
                                                  &= \sum_{j=1}^N s_j(X),
  \end{align*}
  where $M := \sum_{j=1}^{m-1} \trace Q_j$.
  By hypothesis the first and last expressions in the above chain are equal, and therefore we must have equality throughout.

  Since from the previous display $\sum_{j=1}^{m+1} \delta_j \lambda_j = 0$, and because the $\lambda_j$ are distinct and positive, \Cref{lem:majorization-product-decreasing-sequence} guarantees $\delta_j = 0$ for all $1 \le j \le m+1$.
  Therefore, $\trace (PQ^{\perp}) = \trace (P Q_{m+1}) = 0$ and hence $P \le Q$.
  Similarly, for $1 \le j \le m-1$, $\trace (P^{\perp} Q_j) = 0$, and thus $Q_j \le P$,
  so we may write $P = Q_1 + \cdots Q_{m-1} + P'$ for some projection $P'$.
  Finally, the projection $Q-P = Q_m - P' \le Q_m = \chi_{\set{s_N(X)}}(X)$.

  Notice that $X$ commutes with any subprojection of $\chi_{\set{s_N(X)}}(X)$ (because $X$ is scalar relative to this subspace), hence $X$ commutes with $Q-P$.
  Since $X$ also commutes with $Q$ (because it is a spectral projection), it must commute with $P$ as well.
\end{proof}

The next proposition guarantees a kind of block diagonal decomposition for those $X \in \orbit(C)$ which maximize $\ocnr(A)$ for selfadjoint $A \in B(\Hil)$.
This proposition is essential in proving: \Cref{thm:ocnr-closed-rank-condition}, which establishes a sufficient condition for $\ocnr(A)$ to be closed for some $A \in B(\Hil)$; \Cref{thm:direct-sum-characterization}, which characterizes the behavior of the orbit-closed $C$-numerical range under direct sums; and \Cref{thm:normal-convex-c-spectrum}, which establishes an analogue for the orbit-closed $C$-numerical range of $\nr(A) = \conv \ptspec(A)$ when $A \in \K$ is normal.

\begin{proposition}
  \label{prop:block-diagonal-decomposition}
  Suppose that $C \in \traceclass^+$ is a positive trace-class operator and $A \in B(\Hil)$ is selfadjoint with $m := \max \essspec(A)$.
  Let $\set{\lambda_l}_{l=1}^N$ denote the distinct elements of $\spec(A)$ greater than $m$ listed in decreasing order, and including $\lambda_N = m$ when this list is finite.
  
  If $X \in \orbit(C)$ is a maximizer of $\ocnr(A)$, that is, if $\trace(XA) = \sup \ocnr(A)$, then $X$ commutes with each of the projections $P_l = \chi_{\set{\lambda_l}}(A)$ for $1 \le l \le N$.
  Moreover, for $1 \le l < N$, the compression of $X$ to $P_l$ is unitarily equivalent to $\diag(s_{n_{l-1}+1}(C),\ldots,s_{n_l}(C))$ where $n_0 := 0$ and $n_l := \sum_{j=1}^l \trace P_j$.
  Furthermore, if $N < \infty$, then the compression of $X$ to $\chi_{\set{0}}(A)$ lies in $\orbit\big( \diag(s_{n_{N-1}+1}(C),s_{n_{N-1}+2}(C),\ldots,s_{n_N}(C)) \big)$, where this sequence is infinite if $n_N = \infty$.
\end{proposition}

\begin{proof}
  By translating and applying \Cref{prop:c-numerical-range-basics} and \Cref{thm:c-numerical-range-selfadjoint-formula} we may assume without loss of generality that $m = 0$.
  Then $\set{\lambda_l}_{l=1}^N$ are the distinct terms in the sequence $s(A_+)$ listed in decreasing order, and for $1 \le l < N$, the multiplicity of $\lambda_l$ in this sequence is exactly $\trace P_l$.
  Set $P_0 := I - \sum_{l=1}^N P_l$.

  Let $\set{e_j}_{j=-M}^{n_N}$ be an orthonormal basis where $\set{e_j}_{j=-M}^0$ is a basis for $P_0 \Hil$ and for each $1 \le l \le N$, the collection $\set{e_j}_{j=n_{l-1} + 1}^{n_l}$ is a basis for $P_l \Hil$.
  Then let $(d_n)_{n=-M}^{n_N}$ be the diagonal of $X$ relative to this basis.
  We have $(d_n)_{n=1}^{\rank A_+} \submaj s(X) = s(C)$.
  Therefore
  \begin{align*}
    \sum_{n=1}^{\rank A_+} s_n(C) s_n(A_+) &= \sum_{n=1}^{\infty} s_n(C) s_n(A_+) \\
                                        &= \sup \ocnr(A) \\
                                        &= \trace(XA) \\
                                        &= \trace(XA_-) + \sum_{l=1}^{N-1} \trace(XAP_l) \\
                                        &= \sum_{l=1}^{N-1} \trace(XP_l \lambda_l) \\
                                        &= \sum_{l=1}^{N-1} \sum_{n=n_{l-1}+1}^{n_l} d_n \lambda_l \\
                                        &= \sum_{n=1}^{\rank A_+} d_n s_n(A_+).
  \end{align*}
  Then by \Cref{lem:majorization-product-decreasing-sequence} we obtain for each $1 \le l < N$, $\sum_{n=1}^{n_l} d_n = \sum_{n=1}^{n_l} s_n(C)$.
  Therefore, for each $1 \le l < N$, we find
  \begin{equation*}
    \trace (P_l X) = \sum_{n=n_{l-1} + 1}^{n_l} d_n = \sum_{n=n_{l-1} + 1}^{n_l} s_n(C).
  \end{equation*}
  Then by \Cref{lem:extreme-points-k-numerical-range}, $P_1$ commutes with $X$;
  moreover, if $X_1$ denote the compression of $X$ to $P_1$, then $X_1 \in \ugroup\big(\diag(s_1(C),\ldots,s_{n_1}(C))\big)$.

  Consequently, if we consider $X'_1$ to be the compression of $X$ to $P_1^{\perp}$, then $X'_1$ lies in $\orbit\big(\diag(s_{n_1 + 1}(C),s_{n_1 + 2}(C),\ldots)\big)$.
  Therefore, we may again apply \Cref{lem:extreme-points-k-numerical-range} to conclude that $X'_1$ (and hence also $X$) commutes with $P_2$.
  Moreover, if $X_2$ denotes the compression of $X$ to $P_2$, then $X_2 \in \ugroup\big(\diag(s_{n_1 + 1}(C),\ldots,s_{n_2}(C))\big)$.

  Continuing this procedure, by induction on $l$ we obtain for each $1 \le l < N$ that $P_l$ commutes with $X$, and that the compression $X_l$ of $X$ to $P_l$ is a matrix of size $n_l - n_{l-1} = \trace P_l$ with $X_l \in \ugroup \big(\diag(s_{n_{l-1} + 1}(C),\ldots,s_{n_l}(C))\big)$.

  If $\rank C \le \rank A_+ = n_{N-1}$, then the proof is already complete.
  If, on the other hand, $\rank C > \rank A_+$, then $N < \infty$ and we must consider $X'_{N-1}$, which is the compression of $X$ to $P_0 + P_N$.
  From the above, we know that $X'_N \in \orbit\big(\diag(s_{n_{N-1} + 1}(C),s_{n_{N-1} + 2}(C),\ldots)\big)$.
  However, $P_0 = \chi_{(-\infty,0)}(A)$ and $P_N := \chi_{\set{0}}(A)$.
  By \Cref{thm:c-numerical-range-selfadjoint-formula}, since $\trace(XA)$ is a maximizer of $\ocnr(A)$, we must have that $P_0 X = X P_0 = 0$, and therefore the compression of $X$ to $P_N$ lies in $\orbit\big(\diag(s_{n_N + 1}(C),s_{n_N + 2}(C),\ldots,s_{n_N}(C))\big)$.
\end{proof}

Using \Cref{prop:block-diagonal-decomposition}, it is possible to give a condition under which $\sup \ocnr(A)$, or even $\sup \cnr(A)$, is attained when $A \in \K^+$ and $C \in \traceclass^+$.

\begin{remark}
  \label{rem:equal-kernel-positive-compact}
  Suppose that $A,C$ are positive compact operators with infinite rank and that $C$ is trace-class.
  Every $X \in \orbit(C)$ for which $\trace(XA) = \sup \ocnr(A)$ satisfies $\ker X = \ker A$.
  Indeed, by \Cref{prop:block-diagonal-decomposition}, for the projections $P_l$ for $l \in \nats$, the operator $X$ commutes with each $P_l$ and the compression $X_l$ of $X$ to $P_l$ lies in $\ugroup( \diag(s_{n_{l-1}+1}(C),\ldots,s_{n_l}(C)) )$.
  Consequently, $X$ is strictly positive on $(\sum_{l=1}^{\infty} P_l) \Hil$ and must be zero on the complement.
  Notice that $P_0 := I - \sum_{l=1}^{\infty} P_l$ is the projection onto $\ker A$.
  Thus $\ker A = \ker X$.

  Therefore, there is an $X \in \ugroup(C)$ for which $\trace(XA) = \sup \cnr(A) = \sup \ocnr(A)$ if and only if\footnote{if $\dim \ker C = \dim \ker A$, then $X := \zop_{\ker_A} \oplus \diag(s(C)) \in \ugroup(C)$, and $A = \zop_{\ker A} \oplus \diag(s(A))$ relative to the proper basis so $\trace(XA) = \sup \ocnr(A) = \sup \cnr(A)$ by \Cref{prop:c-numerical-range-maximum}.} $\dim \ker C = \dim \ker A$.
\end{remark}

We conclude this section by using \Cref{thm:c-numerical-range-selfadjoint-formula} and \Cref{prop:boundary-approximation,prop:block-diagonal-decomposition} to establish in \Cref{thm:ocnr-closed-rank-condition} a sufficient condition for $\ocnr(A)$ to be closed.

\begin{theorem}
  \label{thm:ocnr-closed-rank-condition}
  Let $C$ be a positive trace-class operator and let $A \in B(\Hil)$.
  Then $\ocnr(A)$ is closed if for every $\theta$, $\rank (\Re(e^{i\theta}A)-m_{\theta}I)_+ \ge \rank C$, where $m_{\theta} := \max \essspec(\Re(e^{i\theta} A))$.
\end{theorem}

\begin{proof}
  Suppose that the rank condition holds for every angle $\theta$.
  Let $x \in \partial\ocnr(A)$.
  There are two possibilities.

  \begin{case}{There is a supporting line for $\closure{\ocnr(A)}$ which intersects $\closure{\ocnr(A)}$ only at $x$.}
    After applying a suitable rotation, we may assume that $\Re x = \sup \ocnr(\Re(A))$ and that the supporting line is vertical, so that $x$ is the unique point of $\closure{\ocnr(A)}$ with maximal real part.
    Since $\rank (\Re A - m_0 I)_+ \ge \rank C$, \Cref{thm:c-numerical-range-selfadjoint-formula} guarantees that $\sup \ocnr(\Re(A)) = \sup \Re(\ocnr(A))$ is attained, and by uniqueness this must be achieved by $x \in \ocnr(A)$.
  \end{case}

  \begin{case}{The only supporting line for $\closure{\ocnr(A)}$ containing $x$ intersects $\closure{\ocnr(A)}$ in a line segment $[x_-,x_+]$.}
    After applying a suitable rotation, we may assume that $\Re x = \sup \ocnr(\Re(A))$, so the line segment $[x_-,x_+]$ is vertical and has maximal real part.
    Moreover, by translating we may further assume $m_0 = 0$.
    In order to prove that $x \in \ocnr(A)$, it suffices to show that $x_{\pm} \in \ocnr(A)$ since this set is convex by \Cref{cor:c-numerical-range-convex}.

    We consider only $x_-$, as the analysis for $x_+$ is identical.
    Then there are two possibilities.
    The first is that $x_-$ itself has a (different) supporting line for $\closure{\ocnr(A)}$ which intersects $\closure{\ocnr(A)}$ only at $x_-$, in which case $x_- \in \ocnr(A)$ by Case 1;
    this happens precisely when $x_-$ is a corner of $\ocnr(A)$.

    The alternative is that there are no other supporting lines passing through $x_-$.
    This implies that for any $\theta > 0$, the supporting line of $\closure{\ocnr(A)}$ with slope $\cot \theta$ intersects the boundary at a point distinct from $x_-$.
    Now, we claim that there are arbitrarily small $\theta > 0$ such that this line intersects $\closure{\ocnr(A)}$ at a \emph{unique} point $x_{\theta}$, which must be in $\ocnr(A)$ by Case 1.
    Indeed, if not, for each sufficiently small angle $\theta > 0$, $\closure{\ocnr(A)}$ would contain a nondegenerate line segment with slope $\cot \theta$, but this would imply that $\partial\ocnr(A)$ has infinite length (since the sum of uncountably many positive numbers is necessarily infinite), which would violate the fact that $\partial\ocnr(A)$ is rectifiable --- a well-known consequence of being a bounded convex curve.
    Moreover, it is clear that $x_{\theta} \to x_-$ as $\theta \to 0^+$ for whichever positive $\theta$ the point $x_{\theta}$ is defined.

    Thus the situation satisfies the hypotheses of \Cref{prop:boundary-approximation},  and so we are guaranteed that $\ocnr(A)$ contains points on the line segment $[x_-,x_+]$ arbitrarily close to $x_-$.
    So consider a sequence of points $(x_j)$ in  $\ocnr(A) \cap [x_-,x_+]$ converging to $x_-$.
    Then there are $X_j \in \orbit(C)$ with $\trace(X_j A) = x_j$, and hence $\trace(X_j \Re A) = \Re x_j = \sup \ocnr(\Re A)$.
    We may therefore apply \Cref{prop:block-diagonal-decomposition} to obtain finite projections $\set{P_l}_{l=1}^N$ such that the compression of $X_j$ to $P_l$ lies in $\ugroup\big(\diag(s_{n_{l-1}+1}(C),\ldots,s_{n_l}(C))\big)$ and moreover $X_j$ commutes with each $P_l$.
    If we set $P_0 := I - \sum_{l=1}^N P_l$, then since $\rank C \le \rank A_+$, we see that $X_j P_0 = 0$.
    So $X_j$ is block diagonal with respect to the blocks $P_l$, and $P_0 X_j = 0$.
    Note, the projections for these blocks are \emph{independent} of $j$.

    Now, by the Schur--Horn theorem (\cite{Sch-1923-SBMG,Hor-1954-AJM}, but see \cite[Theorem~1.1]{KW-2010-JFA} for a concise, self-contained statement), there are block unitaries $U^{(j)} = \bigoplus_{l \ge 0} U^{(j)}_l$ for which $X_j = U^{(j)}(\zop_{P_0 \Hil} \oplus \diag(s(C)) U^{(j)*}$.
    Moreover, we can select $U^{(j)}_0 = I_{P_0 \Hil}$.
    It is important to note that for each $l \ge 1$,  $U^{(j)}_l$ is a finite matrix of size $n_l - n_{l-1}$.

    We now apply the standard recursive subsequence technique to obtain a subsequence of the unitaries $U^{(j)}$ with desirable properties.
    More specifically, by compactness of the unitary group in finite dimensions, there is a subsequence $U^{(j_{1,n})}$ such that $U^{(j_{1,n})}_1$ converges to some unitary matrix $U_1$ (of size $n_1 - n_0$).
    Then for $l \ge 1$ we inductively construct a subsequence $U^{(j_{l+1,n})}$ of $U^{(j_{l,n})}$ for which $U^{(j_{l+1,n})}_{l+1}$ converges to some unitary matrix $U_{l+1}$.
    Then consider the subsequence of the original sequence $U^{(j)}$ given by $V_n := U^{(j_{n,n})}$.
    Define $U := \bigoplus_{l \ge 0} U_l$ and $X := U(\zop_{P_0 \Hil} \oplus \diag(s(C))) U^{*} \in \orbit(C)$.
    Note that $V_n$ converges entrywise to $U$, but not necessarily in any operator topology.
    
    We claim that $X_{j_{n,n}} = V_n (\zop_{P_0 \Hil} \oplus \diag(s(C)) V_n^{*}$ converges in trace-norm to $X$.
    Indeed, let $\epsilon > 0$ and since $C$ is trace-class, there is some $M$ such that $\sum_{n=n_M+1}^{\infty} s_n(C) < \frac{\epsilon}{4}$.
    Set $P := \bigoplus_{l=1}^M P_l$, which is a finite projection that commutes with $U, V_n, C' := \zop_{P_0 \Hil} \oplus \diag(s(C))$ because each $P_l$ does.
    Moreover, because $V_n$ converges entrywise to $U$ and $P$ is a finite projection, $V_n P$ converges to $UP$ in trace norm (or any other norm topology since all norms on a finite dimensional space are equivalent).
    Therefore there is some $K$ such that for all $k \ge K$, $\norm{V_k P - U P}_1 < \frac{\epsilon}{4 \norm{C}}$.
    Thus we obtain
    \begin{align*}
      \norm{(V_k C' V_k^{*} - U C' U^{*})P}_1 &\le \norm{V_k PC' V_k^{*} - U PC' V_k^{*}}_1 + \norm{U C'P V_k^{*} - U C'P U^{*}}_1 \\
                                              &\le \norm{V_k P - U P}_1 \norm{C'V_k^{*}} + \norm{UC'} \norm{PV_k^{*} - PU^{*}}_1 \\
      &< \frac{\epsilon}{4\norm{C}} \norm{C} + \norm{C} \frac{\epsilon}{4\norm{C}} = \frac{\epsilon}{2}.
    \end{align*}
    In addition,
    \begin{align*}
      \norm{(V_k C' V_k^{*} - U C' U^{*})P^{\perp}}_1 &\le \norm{V_k(C' P^{\perp})V_k^{*}}_1 + \norm{U(C' P^{\perp})U^{*}}_1 \\
                                                      &\le\norm{V_k}\norm{C' P^{\perp}}_1 \norm{V_k^{*}} + \norm{U}\norm{C' P^{\perp}}_1 \norm{U^{*}} \\
                                                      &< \frac{\epsilon}{4} + \frac{\epsilon}{4} = \frac{\epsilon}{2}.
    \end{align*}
    Therefore, combining the above displays yields
    \begin{align*}
      \norm{X_{j_{k,k}} - X}_1 &= \norm{V_k C' V_k^{*} - U C' U^{*}}_1 \\
                             &= \norm{(V_k C' V_k^{*} - U C' U^{*})P}_1 + \norm{(V_k C' V_k^{*} - U C' U^{*})P^{\perp}}_1 \\
                             &< \frac{\epsilon}{2} + \frac{\epsilon}{2} = \epsilon.
    \end{align*}
    Thus $x_{j_{n,n}} = \trace(X_{j_{n,n}} A) \to \trace(XA)$.
    Since $x_{j_{n,n}} \to x_-$, we find $x_- = \trace(XA) \in \ocnr(A)$.

    Finally, a symmetric argument applies to $x_+$, and hence $x_{\pm} \in \ocnr(A)$.
    Because $\ocnr(A)$ is convex by \Cref{cor:c-numerical-range-convex}, $x \in [x_-,x_+] \subseteq \ocnr(A)$, thereby completing the proof. \qedhere
  \end{case}
\end{proof}

The following example shows that although the hypothesis of \Cref{thm:ocnr-closed-rank-condition} is not a necessary condition for $\ocnr(A)$ to be closed, it is in some sense sharp.

\begin{example}
  \label{ex:single-angle-nonclosed}
  This example shows that if the rank condition in \Cref{thm:ocnr-closed-rank-condition} fails for even a single angle $\theta$, it is possible for $\ocnr(A)$ not to be closed, even if $\ocnr(A)$ contains at least one boundary point for every angle $\theta$.
  In fact, this example even uses the usual numerical range $\nr(A)$ and a diagonalizable operator $A$.

  Consider the diagonalizable operator $A$ whose eigenvalues are $1$ and $\pm i + e^{\pm \frac{i\pi}{n}}$ for all $n \in \nats$, each with multiplicity one.
  Then the line segment $[1-i,1+i]$ lies on the boundary $\partial \nr(A)$, but $\nr(A) \cap [1-i,1+i] = \set{1}$, although nonempty, contains only a single point.
  Moreover, $\rank \Re(A - m_0 I)_+ = 0$, but $\rank \Re(A - m_{\theta} I)_+ = \infty$ for any $0 < \theta < 2\pi$.
  See \Cref{fig:example-rank} for a diagram of this situation.

  The part of the proof which breaks down because $\rank \Re(A - m_0 I)_+ = 0$ is in the approximation result \Cref{prop:boundary-approximation}.
  In particular, there aren't enough (any) spectral projections $P_j$ corresponding to nonzero eigenvalues of $\Re(A - m_0 I)_+$.

  Of course, if we modify $A$ to have the eigenvalues $1 \pm i$ as well, then $\nr(A)$ becomes closed even though we still have $\rank \Re(A - m_0 I)_+ = 0$.
  Therefore the sufficient condition given in \Cref{thm:ocnr-closed-rank-condition} is not necessary.
  There are even simpler examples: the orbit-closed $C$-numerical range of a scalar is closed (a singleton), but $\rank \Re(A - m_\theta I)_+ = 0$ for all $\theta$.
\end{example}

\begin{figure}
  \centering
  \includegraphics[width=0.75\textwidth]{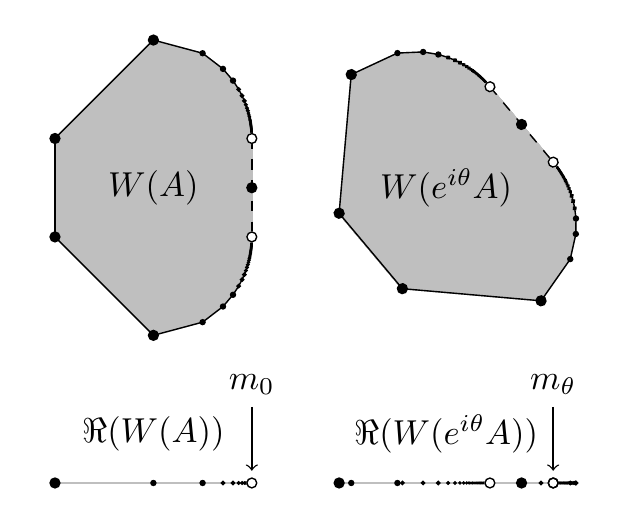}
  \caption{The numerical range and eigenvalues of a diagonalizable operator for which $\rank \Re(A-m_0 I)_+ = 0$ and $\rank \Re(A-m_{\theta} I)_+ = \infty$ for each $0 < \theta < 2\pi$. All eigenvalues have multiplicity 1 and are indicated by filled circles.}
  \label{fig:example-rank}
\end{figure}

\section{Compact normal operators and the $\orbit(C)$-spectrum}
\label{sec:oc-spectrum}

It is a standard result in linear algebra that for a normal matrix $A \in \mat_n(\complex)$ the standard numerical range satisfies $\nr(A) = \conv \spec(A)$, which is an immediate consequence of the elementary (finite or infinite dimensional) fact that $\nr(A_1 \oplus A_2) = \conv \big( \nr(A_1) \cup \nr(A_2) \big)$.
Of course, there are many ways to extend or generalize this result.

One can extend it to the infinite dimensional setting in two ways.
For normal $A \in B(\Hil)$, there is the folklore result $\closure{\nr(A)} = \conv \spec(A)$.
However, restricting to normal $A \in \K$, de Barra, Giles and Sims proved $\nr(A) = \conv \ptspec(A)$ \cite[Theorem~2]{dBGS-1972-JLMSIS}.

The other option is to generalize the matrix result to other numerical ranges, such as the $k$-numerical range or the $C$-numerical range.
In this case, one needs a substitute for the spectrum $\spec(A)$ which is somehow relativized to the matrix $C$.
For $C \in \mat_n(\complex)$ normal, Marcus \cite{Mar-1979-ANYAS} introduced a substitute, now referred to as the \term{$C$-spectrum} and denoted\footnote{
  This notation is common in the later literature, but Marcus actually used the notation $\cspec(A)$ to refer to the \emph{convex hull} of the $C$-spectrum, and he called this the \term{$C$-eigenpolygon}.
}
$\cspec(A)$, consisting of the sums of products of the eigenvalues of $C,A$.
There he proved that if $A \in \mat_n(\complex)$ is also normal, then $\cnr(A) = \conv \cspec(A)$.

Dirr and vom Ende extended the notion of $C$-spectrum to the infinite dimensional setting with $C \in \traceclass$ and $A \in \K$ \cite[Definition~3.2]{DvE-2020-LaMA}, where they also managed to prove \cite[Theorem~3.4, Corollary~3.1]{DvE-2020-LaMA} that if $C,A$ are both normal, then
\begin{equation}
  \label{eq:dirr-vom-ende-c-spectrum-results}
  \cspec(A) \subseteq \cnr(A) \subseteq \conv \closure{\cspec(A)} \quad\text{and thus, if $C=C^{*}$,}\quad \closure{\cnr(A)} = \conv \closure{\cspec(A)}.
\end{equation}
Moreover, if $C$ is normal and $A$ is upper triangular, or vice versa, $\cspec(A) \subseteq \cnr(A)$ \cite[Theorem~3.5]{DvE-2020-LaMA}.

While Dirr and vom Ende's results are impressive, because of the hypothesis $A \in \K$ and de Barra, Giles and Sims result $\nr(A) = \conv \ptspec(A)$, one might hope for the chance to remove the closures from $\closure{\cnr(A)} = \conv \closure{\cspec(A)}$ in \eqref{eq:dirr-vom-ende-c-spectrum-results}.
In this section, for $C \in \traceclass^+$, we do precisely that for the orbit-closed $C$-numerical range and the $\orbit(C)$-spectrum (see \Cref{def:c-spectrum}), which is a (not necessarily closed) slight modification of the $C$-spectrum defined by Dirr and vom Ende (see \Cref{rem:relation-cspec-to-ocspec}).
In particular, \Cref{thm:normal-convex-c-spectrum} says that if $C \in \traceclass^+$ and $A \in \K$ is normal, then $\ocnr(A) = \conv \ocspec(A)$.
Along the way, with \Cref{thm:direct-sum-characterization} we characterize the behavior of the orbit-closed $C$-numerical range under direct sums, thereby generalizing the finite rank result \cite[Result~(4.4)]{LP-1995-FDaaMaE}.

\begin{lemma}
  \label{lem:selfadjoint-direct-sum-characterization}
  Let $C$ be a positive trace-class operator, $P$ an arbitrary projection, and suppose $A = A_{P^{\vphantom{\perp}}} \oplus A_{P^{\perp}} \in B(\Hil)$ is selfadjoint, where $A_{P^{\vphantom{\perp}}},A_{P^{\perp}}$ act on $P\Hil, P^{\perp} \Hil$, respectively.
  Then
  \begin{equation*}
    \sup \ocnr(A) = \sup \set{ \trace(XA) \mid X = X_{P^{\vphantom{\perp}}} \oplus X_{P^{\perp}} \in \orbit(C) },
  \end{equation*}
  and if either supremum is attained, then they both are.
\end{lemma}

\begin{proof}
  The inequality $\sup \ocnr(A) \ge \sup \set{ \trace(XA) \mid X = X_{P^{\vphantom{\perp}}} \oplus X_{P^{\perp}} \in \orbit(C) }$ is trivial since the latter set is a subset of the former.
  We split the remainder of the proof into cases.
  By translating, we may assume $\max \essspec(A) = 0$.
  
  \begin{case}{The supremum $\sup \ocnr(A)$ is attained.}
    We must produce an $X = X_{P^{\vphantom{\perp}}} \oplus X_{P^{\perp}} \in \orbit(C)$ with $\trace(XA) = \sup \ocnr(A)$.
    By \Cref{thm:c-numerical-range-selfadjoint-formula} we are guaranteed that $\rank C \le \trace \chi_{[0,\infty)}(A)$.
    Let $\set{ \lambda_l }_{l=1}^N$ denote the distinct nonnegative eigenvalues of $A$ listed in decreasing order and including zero if and only if this set is finite.
    Then let $P_l$ be the associated spectral projections.
    Set $P_0 := I - \sum_{l=1}^N P_l$.

    Since $P$ commutes with $A$ it commutes with each $P_l$, so we may write $P_l = P_{l} P \oplus P_l P^{\perp}$ which is a sum of orthogonal projections.
    Let $n_0 := 0$ and for $1 \le l \le N$ set $n_l := \sum_{j=1}^l \trace P_j$.
    Note that if $N < \infty$, then $P_N = \chi_{\set{0}}(A)$ which may be either a finite or infinite projection, so $n_N$ may be either finite or infinite.
    Then for each $1 \le l < N$, we may select the finite matrix $\diag((s_n(C))_{n=n_{l-1}+1}^{n_l})$ acting on $P_l \Hil$ and respecting the decomposition $P_l \Hil = P_l P \Hil \oplus P_l P^{\perp} \Hil$, so that $\diag((s_n(C))_{n=n_{l-1}+1}^{n_l}) = X'_l \oplus X''_l$, where $X'_l$ is a matrix of size $\trace P_l P$ and $X''_l$ is a matrix of size $\trace P_l P^{\perp}$.
    The situation for $P_N$ is similar, except that the matrices involved might be infinite.
    That is, we can consider the operator $\diag((s_n(C))_{n=n_{N-1}+1}^{n_N}) = X'_N \oplus X''_N$ acting on the (possibly infinite dimensional) space $P_N \Hil = P_N P \Hil \oplus P_N P^{\perp} \Hil$.
    Set $X'_0 = 0 = X''_0$ acting on $P_0 P \Hil$ and $P_0 P^{\perp} \Hil$, respectively.


    Setting $X_{P^{\vphantom{\perp}}} := \bigoplus_{l=0}^N X'_l$ and $X_{P^{\perp}} := \bigoplus_{l=0}^N X''_l$, we claim that $X = X_{P^{\vphantom{\perp}}} \oplus X_{P^{\perp}}$ is the desired operator.
    Indeed, notice that $\sum_{l=1}^N \trace P_l$ is either infinite or equal to $\trace \chi_{[0,\infty)}(A)$, which is in either case greater than or equal to $\rank C$.
    Therefore, the operators $X'_l,X''_l$ have exhausted all the nonzero values of the sequence $s(C)$ (i.e., $s_j(C) = 0$ if $j > n_N$) and hence $X \in \orbit(C)$.
    Moreover, since $s_n(A_+) = \lambda_l$ whenever $n_{l-1} < n \le n_l$, we find
    \begin{align*}
      \trace (XA) = \sum_{l=0}^N \trace (X P_l A) &= \sum_{l=0}^N \trace ((X'_l \oplus X''_l) \lambda_l) \\
                                                  &= \sum_{l=1}^N \sum_{n=n_{l-1}+1}^{n_l} s_n(C) \lambda_l \\
                                                  &= \sum_{l=1}^N \sum_{n=n_{l-1}+1}^{n_l} s_n(C) s_n(A_+) \\
                                                  &= \sum_{n=1}^{n_N} s_n(C) s_n(A_+).
    \end{align*}
    Now $n_N \ge \rank A_+$, and so we have
    \begin{equation*}
      \trace (XA) = \sum_{n=1}^{n_N} s_n(C) s_n(A_+) = \sum_{n=1}^{\infty} s_n(C) s_n(A_+) = \sup \ocnr(A),
    \end{equation*}
    where the last equality is due to \Cref{thm:c-numerical-range-selfadjoint-formula}.
  \end{case}

  \begin{case}{The supremum $\sup \ocnr(A)$ is not attained.}
    In this case, by \Cref{thm:c-numerical-range-selfadjoint-formula} $M := \trace \chi_{[0,\infty)}(A) < \rank C$, and so this projection is finite.
    Since $0 \in \essspec(A)$, the projection $P_{\epsilon} := \chi_{(-\epsilon,0)}(A)$ must be infinite for any $\epsilon > 0$, and since $P_{\epsilon}$ is a spectral projection for $A$, it commutes with $P$ because $A$ does.
    Therefore $P_{\epsilon} = P_{\epsilon} P + P_{\epsilon} P^{\perp}$ is a sum of projections and at least one of these projections must be infinite.

    Let $C' := \diag(s_1(C),\ldots,s_M(C),0,0,\ldots)$.
    Since $\rank C' = M$, by Case 1 we know that there is some $X = X_{P^{\vphantom{\perp}}} \oplus X_{P^{\perp}} \in \orbit(C')$ for which
    \begin{equation*}
      \trace (X A) = \sup \ocnr[C'](A) = \sum_{n=1}^M s_n(C) s_n(A_+) = \sum_{n=1}^{\infty} s_n(C) s_n(A_+) = \sup \ocnr(A),
    \end{equation*}
    where the third equality follows because $s_n(A_+) = 0$ for $n > M$.
    Moreover, by \Cref{thm:c-numerical-range-selfadjoint-formula} $X P_{\epsilon} = P_{\epsilon} X = 0$.

    Now let $Y$ be the operator given by $\diag(s_{M+1}(C),s_{M+2}(C),\ldots)$ on the subspace $P_{\epsilon} P\Hil$ (or $P_{\epsilon} P^{\perp}\Hil$, whichever is infinite) and zero on the orthogonal complement in $\Hil$.
    Then $X' = X + Y \in \orbit(C)$ commutes with $P$ (so it has a direct sum decomposition), and
    \begin{align*}
      \trace(X' A) = \trace(XA) + \trace (YA) &= \trace(XA) + \trace (Y P_{\epsilon} A) \\
                                              &\ge \sup \ocnr(A) - \norm{P_{\epsilon}A} \norm{Y}_1 \\
                                              &\ge  \sup \ocnr(A) - \epsilon \norm{C}_1.
    \end{align*}
    Since $\epsilon > 0$ is arbitrary, this proves
    \begin{equation*}
      \sup \ocnr(A) \le \sup \set{ \trace(XA) \mid X = X_{P^{\vphantom{\perp}}} \oplus X_{P^{\perp}} \in \orbit(C) },
    \end{equation*}
    and therefore we must have equality.
    Finally, because $\set{ \trace(XA) \mid X = X_{P^{\vphantom{\perp}}} \oplus X_{P^{\perp}} \in \orbit(C) } \subseteq \ocnr(A)$ and $\sup \ocnr(A)$ is not attained, the equality of the suprema guarantees that $\sup \set{ \trace(XA) \mid X = X_{P^{\vphantom{\perp}}} \oplus X_{P^{\perp}} \in \orbit(C) }$ is not attained either. \qedhere
  \end{case}
\end{proof}

Of course, by replacing $A$ with $-A$ in \Cref{lem:selfadjoint-direct-sum-characterization} one immediately obtains the exact same result with the suprema replaced by infima.
Moreover, the finite dimensional counterpart of \Cref{lem:selfadjoint-direct-sum-characterization} is a known result (see \cite[Result~(4.2)]{Li-1994-LMA}), and in that case the suprema are always attained by compactness.
We will make use of both of these facts in order to establish the following theorem.

\begin{theorem}
  \label{thm:direct-sum-characterization}
  Let $C$ be a positive trace-class operator, $P$ an arbitrary projection, and suppose $A = A_{P^{\vphantom{\perp}}} \oplus A_{P^{\perp}} \in B(\Hil)$, where $A_{P^{\vphantom{\perp}}},A_{P^{\perp}}$ act on $P\Hil, P^{\perp} \Hil$, respectively.
  Then
  \begin{equation*}
    \ocnr(A_{P^{\vphantom{\perp}}} \oplus A_{P^{\perp}}) = \conv \quad \bigcup_{\mathclap{\quad C_{P^{\vphantom{\perp}}} \oplus C_{P^{\perp}} \in \orbit(C)}} \ \big( \ocnr[C_{P^{\vphantom{\perp}}}](A_{P^{\vphantom{\perp}}}) + \ocnr[C_{P^{\perp}}](A_{P^{\perp}}) \big).
  \end{equation*}
\end{theorem}

\begin{proof}
  The case when $C$ has finite rank appears in \cite[Result~(4.4)]{LP-1995-FDaaMaE}, so we will prove the result when $C$ has infinite rank.

  One inclusion is immediate.
  Indeed, given $X_{P^{\vphantom{\perp}}} \in \orbit(C_{P^{\vphantom{\perp}}}), X_{P^{\perp}} \in \orbit(C_{P^{\perp}})$ with $C_{P^{\vphantom{\perp}}} \oplus C_{P^{\perp}} \in \orbit(C)$, it is clear that $X_{P^{\vphantom{\perp}}} \oplus X_{P^{\perp}} \in \orbit(C)$.
  Therefore,
  \begin{equation*}
    \trace(X_{P^{\vphantom{\perp}}} A_{P^{\vphantom{\perp}}}) + \trace(X_{P^{\perp}} A_{P^{\perp}}) = \trace \big( (X_{P^{\vphantom{\perp}}} \oplus X_{P^{\perp}})(A_{P^{\vphantom{\perp}}} \oplus A_{P^{\perp}}) \big) \in \ocnr(A).
  \end{equation*}
  Since $C_{P^{\vphantom{\perp}}} \oplus C_{P^{\perp}}\in \orbit(C)$ was arbitrary, as were $X_{P^{\vphantom{\perp}}} \in \orbit(C_{P^{\vphantom{\perp}}})$ and $X_{P^{\perp}} \in \orbit(C_{P^{\perp}})$, we obtain
  \begin{equation*}
    \bigcup_{\mathclap{C_{P^{\vphantom{\perp}}} \oplus C_{P^{\perp}} \in \orbit(C)}} \  \big( \ocnr[C_{P^{\vphantom{\perp}}}](A_{P^{\vphantom{\perp}}}) + \ocnr[C_{P^{\perp}}](A_{P^{\perp}}) \big) \subseteq \ocnr(A).
  \end{equation*}
  By \Cref{cor:c-numerical-range-convex}, $\ocnr(A)$ is convex and so contains the convex hull of this union.

  We now prove the other inclusion.
  For convenience, we replace $C$ by $\diag(s(C))$ since $\orbit(C) = \orbit(\diag(s(C)))$.
  For $m \in \nats$ set $C_m := \diag \big( s_1(C),\ldots,s_m(C),0,0,\ldots \big)$ and notice that $C_m \xrightarrow{\norm{\bigcdot}_1} C$.
  Therefore $\ocnr[C_m](A) \to \ocnr(A)$ in the Hausdorff pseudometric by \Cref{thm:continuity}.

  For any $X_{P^{\vphantom{\perp}}} \oplus X_{P^{\perp}} \in \orbit(C)$, and let $Q_m$ denote the projection onto the span of the eigenvectors associated to $s_1(C),\ldots,s_m(C)$.
  Note that $Q_m$ commutes with $P$ and therefore we can naturally obtain $Y_{P^{\vphantom{\perp}}} \oplus Y_{P^{\perp}} \in \orbit(C_m)$ via $Y_{P^{\vphantom{\perp}}} \oplus Y_{P^{\perp}} = (X_{P^{\vphantom{\perp}}} \oplus X_{P^{\perp}}) Q_m$
  Doing so ensures that $\norm{X_i - Y_i}_1 \le \norm{C_m - C}_1$ for $i=1,2$.
  Conversely, given $Y_{P^{\vphantom{\perp}}} \oplus Y_{P^{\perp}} \in \orbit(C_m)$, both $Y_{P^{\vphantom{\perp}}},Y_{P^{\perp}}$ are finite rank and positive, therefore at least one of them has an infinite dimensional reducing subspace because one of them must act on an infinite dimensional space.
  Then adding $\diag(s_{m+1}(C),s_{m+2}(C),\ldots)$ acting on that subspace yields an operator $X_{P^{\vphantom{\perp}}} \oplus X_{P^{\perp}} \in \orbit(C)$ which again satisfies $\norm{X_i - Y_i}_1 \le \norm{C_m - C}_1$ for $i=1,2$.
  
  Hence, by \Cref{thm:continuity} the Hausdorff distance between these orbit-closed $C$-numerical ranges satisfies
  \begin{equation*}
    d_H \big( \ocnr[X_i](A_i), \ocnr[Y_i](A_i) \big) \le \norm{C_m - C}_1 \norm{A_i} \le \norm{C_m - C}_1 \norm{A}.
  \end{equation*}
  Therefore, the corresponding unions converge in Hausdorff pseudometric
  \begin{equation*}
    \quad\bigcup_{\mathclap{\quad Y_{P^{\vphantom{\perp}}} \oplus Y_{P^{\perp}} \in \orbit(C_m)}} \ \big( \ocnr[Y_{P^{\vphantom{\perp}}}](A_{P^{\vphantom{\perp}}}) + \ocnr[Y_{P^{\perp}}](A_{P^{\perp}}) \big) \xrightarrow{d_H} \quad \bigcup_{\mathclap{\quad X_{P^{\vphantom{\perp}}} \oplus X_{P^{\perp}} \in \orbit(C)}} \ \big( \ocnr[X_{P^{\vphantom{\perp}}}](A_{P^{\vphantom{\perp}}}) + \ocnr[X_{P^{\perp}}](A_{P^{\perp}}) \big).
  \end{equation*}
  Additionally, their convex hulls converge in the Hausdorff pseudometric as well.
  Since the theorem is valid when the trace-class operator is finite rank \cite[Result (4.4)]{LP-1995-FDaaMaE} we have the convergence
  \begin{align*}
    &\ocnr[C_m](A)& &\qquad=& &\conv \quad \bigcup_{\mathclap{\quad Y_{P^{\vphantom{\perp}}} \oplus Y_{P^{\perp}} \in \orbit(C_m)}} \ \big( \ocnr[Y_{P^{\vphantom{\perp}}}](A_{P^{\vphantom{\perp}}}) + \ocnr[Y_{P^{\perp}}](A_{P^{\perp}}) \big)& \\
    &\qquad \bigg\downarrow d_H & & & &\qquad\qquad\qquad\qquad\qquad\bigg\downarrow d_H& \\
    &\ocnr(A)& &\text{same closure}& &\conv \quad \bigcup_{\mathclap{\quad X_{P^{\vphantom{\perp}}} \oplus X_{P^{\perp}} \in \orbit(C)}} \ \big( \ocnr[X_{P^{\vphantom{\perp}}}](A_{P^{\vphantom{\perp}}}) + \ocnr[X_{P^{\perp}}](A_{P^{\perp}}) \big).&
  \end{align*}
  Since $\ocnr(A)$ and $\conv \big( \bigcup_{X_{P^{\vphantom{\perp}}} \oplus X_{P^{\perp}} \in \orbit(C)} \big( \ocnr[X_{P^{\vphantom{\perp}}}](A_{P^{\vphantom{\perp}}}) + \ocnr[X_{P^{\perp}}](A_{P^{\perp}}) \big) \big)$ are convex and have the same closure, they must have the same interior.
  Hence it suffices to prove any boundary point of $\ocnr(A)$ lies in the above convex hull.
  
  Now suppose $x = \trace(XA) \in \ocnr(A)$ lies on the boundary.
  By the usual rotation and translation technique, we may assume $x$ has maximal real part and $\max \essspec(\Re(A)) = 0$.
  Since $\trace(XA) = \sup \Re \ocnr(A) = \sup \ocnr(\Re A)$, we may apply \Cref{prop:block-diagonal-decomposition}.
  Let $\set{\lambda_l}_{l=1}^N$ denote the distinct nonnegative eigenvalues of $\Re A$ listed in decreasing order, and including zero if and only if $N < \infty$.
  Let $\set{P_l}_{l=1}^N$ be the associated spectral projections and set $P_0 := I - \sum_{l=1}^N P_l$.
  Let $n_0 := 0$ and for $1 \le l \le N$, $n_l := \sum_{j=1}^l \trace P_j$.
  Then \Cref{prop:block-diagonal-decomposition} guarantees that $X$ commutes with each $P_l$, so that $X = \bigoplus_{l=0}^N X_l$ where $X_l$ acts on $P_l \Hil$.
  Moreover, $X_0 = 0$ and for $1 \le l < N$, $X_l \in \ugroup(C_l)$, where $ C_l := \diag(s_{n_{l-1}+1}(C),\ldots,s_{n_l}(C))$.
  If $N < \infty$, then $X_N \in \orbit(C_N)$ where $C_N := \diag((s_n(C))_{n=n_{N-1}+1}^{n_N})$.
  
  Let the reader take note that \emph{any} operator $Y$ with properties of $X$ listed in the previous paragraph (block diagonal with respect to $P_l$ with blocks in the associated orbits) has the property that $\Re(\trace(YA)) = \sup \Re \ocnr(A)$.
  We will use this property shortly.

  Let $A_l$ denote the compression of $\Im A$ to $P_l \Hil$.
  In general, $\Im A$ will not be block diagonal with respect to these blocks because $A$ is not necessarily normal so $\Im A$ may not commute with $P_l$.
  However, $P$ commutes with $A$, and therefore with $\Re A$ and $\Im A$, which implies that it also commutes with each spectral projection $P_l$.
  Therefore, we may write each $P_l = P_l P + P_l P^{\perp}$ as a sum of projections, and also $A_l = A'_l \oplus A''_l$, where $A'_l, A''_l$ act on $P_l P \Hil$ and $P_l P^{\perp} \Hil$, respectively.

  Now,
  \begin{equation}
    \label{eq:imaginary-part-inequality}
    \Im(\trace(XA)) = \trace(X \Im A) = \sum_{l=1}^N \trace(X_l A_l) \le \sum_{l=1}^N \sup \ocnr[C_l](A_l).
  \end{equation}
  where we have omitted the $l=0$ term from the sum since $X_0 = 0$.
  For each $1 \le l < N$, the operators $C_l,A_l$ act on the finite dimensional space $P_l \Hil$, and so the supremum $\sup \ocnr[C_l](A_l)$ is attained.
  Moreover, because $A_l = A'_l \oplus A''_l$,  by \Cref{lem:selfadjoint-direct-sum-characterization} (or rather, its finite dimensional counterpart) there exists $Y_l = Y'_l \oplus Y''_l \in \ugroup(C_l)$ such that $\trace(Y_l A_l) = \sup \ocnr[C_l](A_l)$.
  If $N < \infty$, then \Cref{lem:selfadjoint-direct-sum-characterization} still allows us to obtain $Y_N = Y'_N \oplus Y''_N \in \orbit(C_N)$ such that $\trace(X_N A_N) \le \trace(Y_N A_N)$ regardless of whether or not $\sup \ocnr[C_N](A_N)$ is attained.
  Set $Y'_0 = 0 = Y''_0$.
  Then set $Y_{P^{\vphantom{\perp}}} := \bigoplus_{l=0}^N Y'_l$ and $Y_{P^{\perp}} := \bigoplus_{l=0}^N Y''_l$ and $Y := Y_{P^{\vphantom{\perp}}} \oplus Y_{P^{\perp}} \in \orbit(C)$.
  As previously remarked, $Y$ satisfies the same decomposition property as $X$, and therefore $\Re \trace (YA) = \sup \Re \ocnr(A) = \Re \trace (XA)$.
  Moreover, $\Im \trace (XA) \le \Im \trace (YA)$.

  Notice that
  \begin{equation*}
    \Im(\trace(XA)) \ge \sum_{l=1}^N \inf \ocnr[C_l](A_l).
  \end{equation*}
  Therefore, a symmetric argument to the one given in the previous paragraph allows us to produce a $Z = Z_{P^{\vphantom{\perp}}} \oplus Z_{P^{\perp}} \in \orbit(C)$ such that $\Re(\trace XA) = \Re(\trace ZA)$ and $\Im(\trace XA) \ge \Im(\trace ZA)$.
  This proves that
  \begin{equation*}
    x = \trace(XA) \in \conv \set{ \trace(YA), \trace(ZA) } \subseteq \conv \quad \bigcup_{\mathclap{\quad C_{P^{\vphantom{\perp}}} \oplus C_{P^{\perp}} \in \orbit(C)}} \quad \big( \ocnr[C_{P^{\vphantom{\perp}}}](A_{P^{\vphantom{\perp}}}) + \ocnr[C_{P^{\perp}}](A_{P^{\perp}}) \big) ,
  \end{equation*}
  as desired.
\end{proof}

In \cite{DvE-2020-LaMA}, Dirr and vom Ende introduced an analogue of the $C$-spectrum for trace-class $C$ when $A$ is compact, which they also denoted $\cspec(A)$.
We will also need a notion of the $C$-spectrum of a compact operator $A$, but ours will differ slightly from the one given by Dirr and vom Ende, and for this reason we will instead use the notation $\ocspec(A)$ and refer to it as the $\orbit(C)$-spectrum.

As in \cite{DvE-2020-LaMA}, we must invoke the concept of the \term{modified eigenvalue sequence} of a compact operator $A$.
This is the sequence $\modeig(A)$ obtain by mixing $\dim \ker A$ many zeros into the usual eigenvalue sequence $\lambda(A)$.
For the purposes of $\orbit(C)$-spectrum the order of these eigenvalues does not matter.

\begin{definition}
  \label{def:c-spectrum}
  Let $C$ be a trace-class operator of (possibly infinite) rank $N$, and let $A \in \K$.
  The \term{$\orbit(C)$-spectrum} of $A$ is the collection
  \begin{equation*}
    \ocspec(A) := \vset{ \sum_{n=1}^{\infty} \lambda_n(C) \modeig_{\pi(n)}(A) \vmid \pi : \nats \to \nats\ \text{injective}  }.
  \end{equation*}
  One should think of the $\orbit(C)$-spectrum $\ocspec(A)$ as an $\orbit(C)$-relativized analogue of the point spectrum $\ptspec(A)$.
\end{definition}

This definition of $\orbit(C)$-spectrum differs from the definition of $C$-spectrum given in \cite{DvE-2020-LaMA} only in that we allow $\pi$ to be injective instead of a permutation, and that we use the standard eigenvalue sequence of $C$ instead of the modified eigenvalue sequence.

\begin{remark}
  \label{rem:relation-cspec-to-ocspec}
  The terminology $\orbit(C)$-spectrum and notation $\ocspec(A)$ is not haphazard, but alludes to the following relationship between the $C$-spectrum, the $\orbit(C)$-spectrum and the point spectrum.
  For a normal operator $C \in \traceclass$ and $A \in \K$, by \Cref{prop:orbit-closure-equivalences}
  \begin{equation*}
    \ocspec(A) = \ \bigcup_{\mathclap{X \in \orbit(C)}} \ \cspec[X](A) = \  \bigcup_{\mathclap{0 \le n \le \infty}} \ \cspec[C \oplus \zop_n](A).
  \end{equation*}
  So, in essence, the $\orbit(C)$-spectrum is just a version of the $C$-spectrum where the size of the kernel of $C$ can vary, at least when $C$ is normal.

  Moreover, if $C$ is finite rank, then $\ocspec(A) = \cspec(A)$, and if $P$ is a rank-$1$ projection, then $\ocspec[P](A) = \cspec[P](A) = \ptspec(A)$.
\end{remark}

It is a trivial fact that the point spectrum $\ptspec(A)$ of an operator $A$ is contained in the numerical range $\nr(A)$, and by convexity $\conv \ptspec(A) \subseteq \nr(A)$.
The following proposition establishes an analogous fact for the $\orbit(C)$-spectrum and the orbit-closed $C$-numerical range.

\begin{proposition}
  \label{prop:c-spectrum-subset-ocnr}
  If $C \in \traceclass$ is normal and $A \in \K$ is upper triangular relative to some orthonormal basis, then $\ocspec(A) \subseteq \ocnr(A)$.
  If, in addition, $C$ is selfadjoint, then the inclusion $\conv \ocspec(A) \subseteq \ocnr(A)$ also holds.
\end{proposition}

\begin{proof}
  Suppose that $A \in \K$ is upper triangular relative to an orthonormal basis $\set{e_n}_{n=1}^{\infty}$ for $\Hil$.
  Then it is well known that the diagonal entries of $A$ are precisely the (suitably permuted) modified eigenvalue sequence $\modeig(A)$.
  Indeed, the sequence of subspaces
  \begin{equation*}
    \set{0} \subseteq \spans \set{e_1} \subseteq \spans \set{e_1,e_2} \subseteq \spans \set{e_1,e_2,e_3} \subseteq \cdots \subseteq \Hil,
  \end{equation*}
  forms a triangularizing chain for $A$, and so the nonzero diagonal entries are precisely the eigenvalues by Ringrose's Theorem, and they are repeated according to algebraic multiplicity (see \cite[Theorems~7.2.3 and 7.2.9]{RR-2000}).

  Then take any injective $\pi : \nats \to \nats$ and define a sequence $(x_n)$ by
  \begin{equation*}
    x_n :=
    \begin{cases}
      \eig_{\pi^{-1}(n)}(C) & \text{if } n \in \pi(\nats), \\
      0 & \text{otherwise.} \\
    \end{cases}
  \end{equation*}
  Then since $C$ is normal, by \Cref{prop:orbit-closure-equivalences} $X := \diag(x_n) \in \orbit(C)$.
  Moreover,
  \begin{equation*}
    \trace(XA) = \sum_{n=1}^{\infty} x_n \modeig_n(A) = \sum_{n \in \pi(\nats)} \eig_{\pi^{-1}(n)}(C) \modeig_n(A) = \sum_{n=1}^{\infty} \eig_n(C) \modeig_{\pi(n)}(A).
  \end{equation*}
  Since $\pi$ was arbitrary, $\ocspec(A) \subseteq \ocnr(A)$.

  Finally, if $C$ is selfadjoint, then by \Cref{cor:c-numerical-range-convex}, $\ocnr(A)$ is convex and therefore $\conv \ocspec(A) \subseteq \ocnr(A)$.
\end{proof}

Before we prove our last main theorem in this section (\Cref{thm:normal-convex-c-spectrum}), which says that $\conv \ocspec(A) = \ocnr(A)$ when $A \in \K$ is normal and $C \in \traceclass^+$, we need lemmas corresponding to two special cases: $A$ selfadjoint, and $A$ normal with finite spectrum.

\begin{lemma}
  \label{lem:selfadjoint-convex-c-spectrum}
  If $A \in \K^{sa}$ and $C \in \traceclass^+$, then $\ocnr(A) = \conv \ocspec(A)$.
\end{lemma}

\begin{proof}
  Since $A \in \K^{sa}$ is diagonalizable, by \Cref{prop:c-spectrum-subset-ocnr} we only need to prove the inclusion $\ocnr(A) \subseteq \conv \ocspec(A)$.

  Notice that $\ocnr(A)$ is an interval since it is convex and contained in $\reals$.
  We will prove that when $\sup \ocnr(A)$ is attained then it is an element of $\ocspec(A)$, and when the supremum is not attained, $\ocspec(A)$ contains elements arbitrarily close to $\sup \ocnr(A)$.
  Of course, symmetric arguments apply to the infimum, thereby establishing the desired equality $\ocnr(A) = \conv \ocspec(A)$.

  By \Cref{thm:c-numerical-range-selfadjoint-formula} and \Cref{prop:c-numerical-range-maximum} we know that
  \begin{equation*}
    \sup \ocnr(A) = \sum_{n=1}^{\infty} s_n(C) s_n(A_+),
  \end{equation*}
  and that this supremum is attained if and only if $N := \trace \chi_{[0,\infty)}(A) \ge \rank C$.
  Moreover, since for $1 \le n \le N$, $s_n(A_+)$ is an eigenvalue for $A$, when the inequality $N \ge \rank C$ holds, we obtain
  \begin{equation*}
    \sup \ocnr(A) = \sum_{n=1}^{\rank C} s_n(C) s_n(A_+) \in \ocspec(A).
  \end{equation*}

  In the case when $\sup \ocnr(A)$ is not attained, we know that $N < \rank C$ (so $N < \infty$), and therefore $\chi_{[0,\infty)}(A)$ is a finite projection.
  Since $0 \in \essspec(A)$, for every $\epsilon > 0$, the projection $\chi_{(-\epsilon,0)}(A)$ is infinite.
  Therefore $A$ has infinitely many arbitrarily small negative eigenvalues.
  Let $\pi : \nats \to \nats$ be an injective function such that $\modeig_{\pi(n)}(A) = s_n(A_+)$ for $1 \le n \le N$ and for $n > N$, $\frac{-\epsilon}{\sum_{k=N+1}^{\infty} s_k(C)} < \modeig_{\pi(n)}(A) < 0$.
  Multiplying this inequality by $s_n(C)$ and summing over $n > N$ yields
  $-\epsilon < \sum_{n=N+1}^{\infty} s_n(C) \modeig_{\pi(n)}(A) < 0$.
  Therefore,
  \begin{align*}
    \sup \ocnr(A) - \epsilon &= \sum_{n=1}^N s_n(C) s_n(A_+) - \epsilon \\
                             &< \sum_{n=1}^N s_n(C) \modeig_{\pi(n)}(A) + \sum_{n=N+1}^{\infty} s_n(C) \modeig_{\pi(n)}(A) \\
                             &= \sum_{n=1}^{\infty} s_n(C) \modeig_{\pi(n)}(A) \in \ocspec(A).
  \end{align*}
  Therefore $\ocspec(A)$ contains elements which are arbitrarily close to $\sup \ocnr(A)$.

  As remarked at the beginning of the proof, symmetric arguments hold for $\inf \ocnr(A)$, and therefore $\ocnr(A) = \conv \ocspec(A)$.
\end{proof}

\begin{lemma}
  \label{lem:finite-direct-sum-c-spectrum}
  If $A \in \K$ is normal with finite spectrum and $C \in \traceclass^+$, then $\ocnr(A) = \conv \ocspec(A)$.
\end{lemma}

\begin{proof}
  Let $\spec(A) = \set{\lambda_1,\ldots,\lambda_m}$ listed in order of decreasing modulus, and let $P_1,\ldots,P_m$ be the corresponding spectral projections.
  Of course, $0 \in \spec(A)$, so after relabeling we may assume $\lambda_m = 0$, and therefore $P_m$ is the only infinite projection among the list since $A \in \K$.

  Now $A = \bigoplus_{j=1}^m \lambda_j I_{P_j \Hil}$, by \Cref{thm:direct-sum-characterization}, every element of $\ocnr(A)$ is a convex combination of terms of the form $\trace(XA)$ where $X = \bigoplus_{j=1}^m X_j \in \orbit(C)$ and $X_j$ acts on $P_j \Hil$.
  We claim that any such term lies in $\ocspec(A)$.
  Indeed, suppose that for $1 \le j < m$, $\set{e_k}_{k=n_{j-1} + 1}^{n_j}$ is a basis for $P_j \Hil$ which diagonalizes $X_j$, and that $\set{e_k}_{k=n_{m-1}+1}^{\infty}$ is a basis for $P_m \Hil$ which diagonalizes $X_m$.

  Since $X \in \orbit(C)$ is diagonal relative to this orthonormal basis $\set{e_k}_{k=1}^{\infty}$, the nonzero terms of its diagonal sequence $(d_k)_{k=1}^{\infty}$ must consist precisely of the nonzero terms of $s(C)$.
  Moreover, relative to this basis, $A$ is already diagonalized and its diagonal is precisely $\eig(A) = \modeig(A)$.
  Let $\pi : \nats \to \nats$ be an injective function such that $d_{\pi(k)} = s_k(C)$.
  Then we find
  \begin{equation*}
    \trace(XA) = \sum_{k=1}^{\infty} d_k \modeig_k(A) = \sum_{k \in \pi(\nats)} d_k \modeig_k(A) = \sum_{n = 1}^{\infty} d_{\pi(n}) \modeig_n(A) = \sum_{n=1}^{\infty} s_n(C) \modeig_{\pi(n)}(A) \in \ocspec(A).
  \end{equation*}
  Since any element of $\ocnr(A)$ is a convex combination of such terms, we obtain the inclusion $\ocnr(A) \subseteq \conv \ocspec(A)$, and equality follows from \Cref{prop:c-spectrum-subset-ocnr}.
\end{proof}

\begin{theorem}
  \label{thm:normal-convex-c-spectrum}
  If $C \in \traceclass^+$ is a positive trace-class operator and $A \in \K$ is compact normal, then the orbit-closed $C$-numerical range and the convex hull of the $\orbit(C)$-spectrum coincide.
  That is,
  \begin{equation*}
    \ocnr(A) = \conv \ocspec(A).
  \end{equation*}
\end{theorem}

\begin{proof}
  Since a normal operator $A \in \K$ is diagonalizable, by \Cref{prop:c-spectrum-subset-ocnr} we only need to prove the inclusion $\ocnr(A) \subseteq \conv \ocspec(A)$.

  Consider a basis diagonalizing $A$, so that $A = \diag(\modeig(A))$.  Then for $m \in \nats$ we define the finite rank operators $A_m := \diag \big( \modeig_1(A), \ldots, \modeig_m(A), 0, \ldots \big)$ and notice $A_m \xrightarrow{\norm{\bigcdot}} A$.
  Therefore $\ocnr(A_m) \to \ocnr(A)$ in the Hausdorff pseudometric by \Cref{thm:continuity}.

  Now $A_m$ is a normal compact operator with finite spectrum, so by \autoref{lem:finite-direct-sum-c-spectrum} we obtain $\conv \ocspec(A_m) = \ocnr(A_m)$.

  We now prove that $\ocspec(A_m) \to \ocspec(A)$.
  Let $\epsilon > 0$, and choose $M \in \nats$ such that for all $m \ge M$, $\norm{A_m - A} < \frac{\epsilon}{\norm{C}_1}$.
  Let $\pi : \nats \to \nats$ be any injective function.
  Then
  \begin{align*}
    \sum_{n=1}^{\infty} s_n(C) \modeig_{\pi(n)}(A_m) - \sum_{n=1}^{\infty} s_n(C) \modeig_{\pi(n)}(A)
    &= \sum_{n=1}^{\infty} s_n(C) \big( \modeig_{\pi(n)}(A_m) - \modeig_{\pi(n)}(A) \big) \\
    &\le \norm{C}_1 \norm{A_m - A} < \norm{C}_1 \frac{\epsilon}{\norm{C}_1} = \epsilon.
  \end{align*}

  Consequently, $d_H \big( \ocspec(A_m), \ocspec(A) \big) < \epsilon$, so $\ocspec(A_m) \to \ocspec(A)$.
  Moreover, this implies $\conv \ocspec(A_m)$ converges to $\conv \ocspec(A)$ in the Hausdorff pseudometric as well.
  Thus
  \begin{equation*}
    \conv \ocspec(A) \xleftarrow{d_H} \conv \ocspec[C_m](A) = \ocnr[C_m](A) \xrightarrow{d_H} \ocnr(A),
  \end{equation*}
  and hence $\closure{\conv \ocspec(A)} = \closure{\ocnr(A)}$.

  By the above, it suffices to prove that every element of the boundary of $\ocnr(A)$ is also an element of $\conv \ocspec(A)$.
  The argument is very similar to the one in the proof of \Cref{thm:direct-sum-characterization}, except we apply \Cref{lem:selfadjoint-convex-c-spectrum} in place of \Cref{lem:selfadjoint-direct-sum-characterization}.

  Suppose that $\trace(XA) \in \ocnr(A)$ lies on the boundary.
  By rotating, we may suppose that $\Re \trace(XA) = \sup \Re \ocnr(A) = \sup \ocnr(\Re A)$.
  Then by \Cref{prop:block-diagonal-decomposition} we get spectral projections $\set{ P_l }_{l=1}^N$ associated to the distinct nonnegative eigenvalues $\set{ \lambda_l }_{l=1}^N$ of $\Re A$, including zero if and only if $N < \infty$.
  We set $P_0 := I - \sum_{l=1}^N P_l$ and $n_0 := 0$, and $n_l := \sum_{j=1}^l \trace P_l$.
  In addition, $X$ commutes with each $P_l$ and if $X_l$ denotes the compression of $X$ to $P_l \Hil$, then $X = \bigoplus_{l=0}^N X_l$.
  Moreover, $X_0 = 0$ and for $1 \le l < N$ we have $X_l \in \ugroup(C_l)$ where $C_l := \diag(s_{n_{l-1}+1}(C),\ldots,s_{n_l}(C))$.
  If $N < \infty$, then $X_N \in \orbit(C_N)$ where $C_N := \diag(s_{n_{N-1} + 1}(C),s_{n_{N-1} + 2}(C),\ldots, s_{n_N}(C))$.
   
  Since $A$ is normal, it is clear that $\Re A, \Im A$ commute, and therefore $\Im A$ commutes with each $P_l$, and so $A$ commutes with $P_l$ too.
  Let $A_l$ be the compression of $A$ to $P_l \Hil$, so $A = \bigoplus_{l=0}^N A_l$.
  Now for each $1 \le l \le N$, by \Cref{lem:selfadjoint-convex-c-spectrum} we know that there are elements $y_l, z_l \in \ocspec[C_l](\Im A_l)$ such that $z_l \le \trace(X_l \Im A_l) \le y_l$.
  Moreover, since $\Re A_l = \lambda_l I_{P_l \Hil}$, then $z'_l := \lambda_l \trace C_l + i z_l, y'_l := \lambda_l \trace C_l + i y_l \in \ocspec[C_l](A_l)$.
  Notice that $\trace(X_l A_l) = \lambda_l \trace C_l + i \trace(X_l \Im A_l)$, hence $\trace (X_l A_l) \in [z'_l, y'_l]$.
  Summing over $1 \le l \le N$, we obtain
  \begin{equation*}
      \trace(XA) = \sum_{l=1}^N \trace(X_l A_l) \in [z,y], \quad\text{where}\quad  z := \sum_{l=1}^N z'_l \quad \text{and}\quad  y := \sum_{l=1}^N y'_l.
  \end{equation*}
  Finally, $z,y \in \sum_{l=1}^N \ocspec[C_l](A_l) \subseteq \ocspec[\bigoplus_{l=1}^N C_l] \big( \bigoplus_{l=1}^N A_l \big) \subseteq \ocspec(A)$, and therefore we obtain $\trace(XA) \in [z,y] \subseteq \conv \ocspec(A)$.
\end{proof}

\section{Convexity of the $C$-numerical range.}
\label{sec:c-numerical-range-convexity}

In their paper \cite{DvE-2020-LaMA}, Dirr and vom Ende asked whether the $C$-numerical range $\cnr(A)$ is convex when $C$ is normal with collinear eigenvalues.
We will now show that when $A$ is diagonalizable and $C$ is positive and has either trivial or infinite dimensional kernel, then this is indeed the case (see \Cref{cor:c-numerical-range-convex-diagonalizable}).
We make no claim that these circumstances are exhaustive, but we are limited by the proof technique and the underlying results.
Nevertheless, we felt that a partial answer to the question of the convexity of $\cnr(A)$ would contribute some value.

Let $E: B(\Hil) \to \mathcal{D}$ denote the canonical trace-preserving conditional expectation onto a diagonal masa $\mathcal{D}$.
In other words, $E$ is the operation of ``taking the main diagonal.''
When applied to the unitary orbit of an operator $C$, there is a natural bijection between $E(\ugroup(C))$ and the set of all diagonal sequences of $C$ as the orthonormal basis giving rise to the matrix representation of $C$ varies.
The study of diagonal of operators has a rich history in the literature.
For a survey, see \cite{LW-2020-OT27}.

The following\footnote{In \cite{KW-2010-JFA}, this is stated in terms of the so-called \emph{partial isometry orbit, $\mathcal{V(C)}$}, but \cite[Proposition~2.1.12]{Lor-2016} guarantees that $\closure[\norm{\bigcdot}]{\ugroup(C)} = \mathcal{V}(C)$ for $C \in \K^+$.} gives a complete characterization of diagonals of compact operators modulo the dimension of the kernel.

\begin{proposition}[\protect{\cite[Proposition~6.4]{KW-2010-JFA}}]
  \label{prop:compact-schur-horn}
  For a positive compact operator $C$,
  \begin{equation*}
    E\Big(\closure[\norm{\bigcdot}]{\ugroup(C)}\Big) = \{ X \in \mathcal{D} \cap \K^+ \mid s(X) \maj s(C) \}.
  \end{equation*}
\end{proposition}

Since the set $\{ X \in \mathcal{D} \cap \K^+ \mid s(X) \maj s(C) \}$ is readily seen to be convex, \Cref{prop:compact-schur-horn} can be used to give a one-line proof that $\ocnr(A)$ is convex whenever $A$ is diagonalizable, thereby providing yet another proof of \Cref{thm:c-numerical-range-via-majorization} in this restricted setting.
Indeed, suppose $A \in \mathcal{D}$ and let $t \in [0,1]$ and suppose $X_1,X_2 \in \orbit(C)$.
Then there is some $X \in \orbit(C)$ for which $E(X) = E(t X_1 + (1-t) X_2)$, and therefore
\begin{equation*}
  \trace(E((t X_1 + (1-t)X_2)A)) = \trace(E(t X_1 + (1-t)X_2)A) = \trace(E(X)A) = \trace(E(XA)).
\end{equation*}
Since the conditional expectation is trace-preserving, $\trace((t X_1 + (1-t)X_2)A) = \trace(XA)$.

It turns out that there are certain circumstances under which $E(\ugroup(C))$ has been characterized, namely when $\ker C$ is either trivial \cite[Proposition~6.6]{KW-2010-JFA} or infinite dimensional \cite[Corollary~3.5]{LW-2015-JFA}.
In both cases, the characterization is still linked to majorization but the details of the definitions are a bit too technical for our present purposes.
Nevertheless, it is known that $E(\ugroup(C))$ is convex if $\ker C$ is trivial \cite[Corollary~6.7]{KW-2010-JFA} or infinite dimensional\footnote{In the case when $\ker C$ is nontrivial but finite dimensional, the first author has conjectured a characterization of $E(\ugroup(C))$ and has established that this conjectured set is convex. See \cite[Conjecture~3.6, Lemma~4.2]{LW-2015-JFA} for details.} \cite[Corollary~4.3]{LW-2015-JFA}.

\begin{proposition}[\protect{\cite[Corollary~6.7]{KW-2010-JFA},\cite[Corollary~4.3]{LW-2015-JFA}}]
  \label{prop:convex-schur-horn}
  Let $C$ be a positive compact operator.
  If $\ker C$ is either trivial or infinite dimensional, then $E(\ugroup(C))$ is convex.
\end{proposition}

This immediately yields the following corollary concerning the convexity of $\cnr(A)$.

\begin{corollary}
  \label{cor:c-numerical-range-convex-diagonalizable}
  Let $C$ be a positive trace-class operator and suppose that $\ker C$ is either trivial or infinite dimensional.
  For any diagonalizable operator $A$, $\cnr(A)$ is convex.
\end{corollary}

\begin{proof}
  Suppose $A$ is diagonalizable.
  Then after conjugating by a suitable unitary, which we can absorb into $\ugroup(C)$, we may assume $A \in \mathcal{D}$.
  Let $t \in [0,1]$ and suppose $X_1,X_2 \in \ugroup(C)$.
  Then by \Cref{prop:convex-schur-horn} there is some $X \in \ugroup(C)$ for which $E(X) = E(t X_1 + (1-t) X_2)$, and therefore
  \begin{equation*}
    \trace(E((t X_1 + (1-t)X_2)A)) = \trace(E(t X_1 + (1-t)X_2)A) = \trace(E(X)A) = \trace(E(XA)).
  \end{equation*}
  Then $\trace((t X_1 + (1-t)X_2)A) = \trace(XA) \in \cnr(A)$ since the conditional expectation is trace-preserving.
\end{proof}

\bibliographystyle{tfnlm}
\bibliography{references.bib}

\begin{thebibliography}{10}
\providecommand{\url}[1]{\normalfont{#1}}
\providecommand{\urlprefix}{Available from: }

\bibitem{Toe-1918-MZ}
Toeplitz~O. {Das algebraische Analogon zu einem Satze von Fejér}. Math Z.
  1918;\hspace{0pt}2(1-2):187--197.

\bibitem{Hau-1919-MZ}
Hausdorff~F. {Der Wertvorrat einer Bilinearform}. Math Z.
  1919;\hspace{0pt}3(1):314--316.

\bibitem{Dav-1971-CMB}
Davis~C. The {T}oeplitz-{H}ausdorff theorem explained. Canad Math Bull.
  1971;\hspace{0pt}14:245--246.

\bibitem{Hal-1964-ASM}
{Halmos}~PR. Numerical ranges and normal dilations. Acta Sci Math.
  1964;\hspace{0pt}25:1--5.

\bibitem{Ber-1963}
Berger~CA. Normal dilations [dissertation]. Cornell University; 1963. Advisor:
  Morris Schreiber;
  \urlprefix\url{https://hdl.handle.net/2027/coo.31924001140247}.

\bibitem{FW-1971-GMJ}
{Fillmore}~PA, {Williams}~JP. Some convexity theorems for matrices. Glasg Math
  J. 1971;\hspace{0pt}12:110--117.

\bibitem{Wes-1975-LMA}
{Westwick}~R. {A theorem on numerical range.} {Linear Multilinear Algebra}.
  1975;\hspace{0pt}2:311--315.

\bibitem{GS-1977-LAA}
{Goldberg}~M, {Straus}~EG. {Elementary inclusion relations for generalized
  numerical ranges.} {Linear Algebra Appl}. 1977;\hspace{0pt}18:1--24.

\bibitem{Li-1994-LMA}
{Li}~CK. {\(C\)-numerical ranges and \(C\)-numerical radii.} {Linear
  Multilinear Algebra}. 1994;\hspace{0pt}37(1-3):51--82.

\bibitem{AT-1983-LMA}
{Au-Yeung}~YH, {Tsing}~NK. {A conjecture of Marcus on the generalized numerical
  range.} {Linear Multilinear Algebra}. 1983;\hspace{0pt}14:235--239.

\bibitem{CT-1996-LMA}
{Cheung}~WS, {Tsing}~NK. {The \(C\)-numerical range of matrices is
  star-shaped.} {Linear Multilinear Algebra}. 1996;\hspace{0pt}41(3):245--250.

\bibitem{Mar-1979-ANYAS}
Marcus~M. Some combinatorial aspects of numerical range. Ann NY Acad Sci.
  1979;\hspace{0pt}319(1):368--376.

\bibitem{DvE-2020-LaMA}
{Dirr}~G, {vom Ende}~F. {The \(C\)-numerical range in infinite dimensions.}
  {Linear Multilinear Algebra}. 2020;\hspace{0pt}68(4):652--678.

\bibitem{DS-2017-PEMSIS}
{Dykema}~K, {Skoufranis}~P. {Numerical ranges in II\(_1\) factors.} {Proc Edinb
  Math Soc, II Ser}. 2017;\hspace{0pt}61(1):31--55.

\bibitem{Poo-1980-LMA}
{Poon}~YT. {Another proof of a result of Westwick.} {Linear Multilinear
  Algebra}. 1980;\hspace{0pt}9:35--37.

\bibitem{Bir-1946-UNTRA}
Birkhoff~G. Three observations on linear algebra. Univ Nac Tucumán Revista A.
  1946;\hspace{0pt}5:147--151.

\bibitem{AK-2006-OTOAaA}
Arveson~W, Kadison~RV. Diagonals of self-adjoint operators. In: Han~D,
  Jorgensen~PE, Larson~DR, editors. Operator theory, operator algebras, and
  applications. (Contemp. Math.; Vol. 414). Amer. Math. Soc., Providence, RI;
  2006. p. 247--263.

\bibitem{KW-2010-JFA}
Kaftal~V, Weiss~G. {An infinite dimensional Schur--Horn Theorem and
  majorization theory}. J Funct Anal. 2010;\hspace{0pt}259(12):3115--3162.
  \urlprefix\url{http://www.sciencedirect.com/science/article/pii/S0022123610003563}.

\bibitem{Lor-2016}
Loreaux~J. Diagonals of operators [dissertation]. University of Cincinnati;
  2016.

\bibitem{GP-1974-DMJ}
{Gellar}~R, {Page}~L. {Limits of unitarily equivalent normal operators.} {Duke
  Math J}. 1974;\hspace{0pt}41:319--322.

\bibitem{Voi-1976-RRMPA}
Voiculescu~D. A non-commutative {W}eyl-von {N}eumann theorem. Rev Roumaine Math
  Pures Appl. 1976;\hspace{0pt}21(1):97--113.

\bibitem{HN-1991-TAMS}
{Hiai}~F, {Nakamura}~Y. {Closed convex hulls of unitary orbits in von Neumann
  algebras.} {Trans Am Math Soc}. 1991;\hspace{0pt}323(1):1--38.

\bibitem{Chu-1981-JLMSIS}
{Chu}~CH. {A note on scattered C*-algebras and the Radon-Nikodym property.} {J
  Lond Math Soc, II Ser}. 1981;\hspace{0pt}24:533--536.

\bibitem{Phe-1974-JFA}
{Phelps}~RR. {Dentability and extreme points in Banach spaces.} {J Funct Anal}.
  1974;\hspace{0pt}17:78--90.

\bibitem{dBGS-1972-JLMSIS}
{de Barra}~G, {Giles}~JR, {Sims}~B. {On the numerical range of compact
  operators on Hilbert spaces.} {J Lond Math Soc, II Ser}.
  1972;\hspace{0pt}5:704--706.

\bibitem{LP-1995-FDaaMaE}
{Li}~CK, {Poon}~YT. {Some results on the \(c\)-numerical range.} In: {Five
  decades as a mathematician and educator. On the 80th birthday of Professor
  Yung-Chow Wong}. Singapore: World Scientific; 1995. p. 247--258.

\bibitem{FF-1980-PAMS}
Fan~P, Fong~CK. Which operators are the self-commutators of compact operators?
  Proc Amer Math Soc. 1980;\hspace{0pt}80:58--60.

\bibitem{Kip-1951-MN}
{Kippenhahn}~R. {\"Uber den Wertevorrat einer Matrix.} {Math Nachr}.
  1951;\hspace{0pt}6:193--228.

\bibitem{Kip-2008-LMA}
{Kippenhahn}~R. {\"Uber den Wertevorrat einer Matrix.} {Linear Multilinear
  Algebra}. 2008;\hspace{0pt}56(1-2):185--225.

\bibitem{Joh-1978-SJNA}
{Johnson}~CR. {Numerical determination of the field of values of a general
  complex matrix.} {SIAM J Numer Anal}. 1978;\hspace{0pt}15:595--602.

\bibitem{MOA-2011}
{Marshall}~AW, {Olkin}~I, {Arnold}~BC. {Inequalities: theory of majorization
  and its applications. 2nd edition.} 2nd ed. New York, NY: Springer; 2011.

\bibitem{Bla-2006}
{Blackadar}~B. {Operator algebras. Theory of \(C^*\)-algebras and von Neumann
  algebras.} Berlin: Springer; 2006.

\bibitem{MS-2015-JRAM}
{Makarov}~KA, {Seelmann}~A. {The length metric on the set of orthogonal
  projections and new estimates in the subspace perturbation problem.} {J Reine
  Angew Math}. 2015;\hspace{0pt}708:1--15.

\bibitem{Sch-1923-SBMG}
Schur~I. Über eine klasse von mittelbildungen mit anwendungen auf der
  determinantentheorie. Sitzungsber Berliner Mat Ges.
  1923;\hspace{0pt}22:9--29.

\bibitem{Hor-1954-AJM}
Horn~A. Doubly stochastic matrices and the diagonal of a rotation matrix. Amer
  J Math. 1954;\hspace{0pt}76:620--630.

\bibitem{RR-2000}
{Radjavi}~H, {Rosenthal}~P. {Simultaneous triangularization.} New York, NY:
  Springer; 2000.

\bibitem{LW-2020-OT27}
Loreaux~J, Weiss~G. On diagonals of operators: selfadjoint, normal and other
  classes. In: Operator Theory: Themes and Variations --- Conference
  Proceedings, Timişoara, July 2–6, 2018; 2020. p. 193--214.

\bibitem{LW-2015-JFA}
Loreaux~J, Weiss~G. {Majorization and a Schur--Horn Theorem for positive
  compact operators, the nonzero kernel case}. J Funct Anal. 2015
  February;\hspace{0pt}268(3):703--731.
  \urlprefix\url{http://www.sciencedirect.com/science/article/pii/S0022123614004625}.

\end{thebibliography}

\end{document}